\documentclass[final,reqno]{elsarticle}
\usepackage{mathrsfs}
\usepackage{amssymb,amsmath,amsfonts,latexsym}
\usepackage{amsthm}
\usepackage{graphicx,float,epsfig,color,fancyhdr}
\usepackage{framed}
\usepackage{todonotes}
\usepackage{comment}
\usepackage{multirow}
\usepackage{relsize}
\usepackage[utf8x]{inputenc}
\usepackage{scalerel,subfigure,booktabs}
\usepackage{enumitem}
\usepackage{multicol}
\usepackage{siunitx}
\usepackage{array}

\usepackage[colorlinks=true,breaklinks=true,linkcolor=lightblue,citecolor=lightblue,urlcolor=lightblue]{hyperref}

\allowdisplaybreaks

\newtheorem{lemma}{Lemma}[section]
\newtheorem{remark}{Remark}[section]
\newtheorem{proposition}{Proposition}[section]

\newtheorem{theorem}{Theorem}[section]
\newtheorem{corollary}{Corollary}[section]

\def\bbL{\mathbb{L}}

\def\l{\lambda}
\def\PiK{\Pi_{K}^{\Delta}}

\def\sp{\mathop{\mathrm{sp}}\nolimits}

\newcommand\vdiv{\mathop{\mathrm{div}}\nolimits}
\newcommand\bdiv{\mathop{\mathbf{div}}\nolimits}
\newcommand\cF{\mathcal{F}}
\newcommand\cJ{\mathcal{J}}
\newcommand\rP{\mathrm{P}}

\newcommand{\norm}[1]{\|#1\|}
\newcommand{\jump}[1]{[\![#1]\!]}

\journal{}
\date{\today}

\def\bL{\mathbf{L}}
\def\bP{\mathbf{P}}
\def\bp{\mathbf{p}}
\def\bv{\mathbf{v}}
\def\bn{\boldsymbol{n}}
\def\bkappa{\boldsymbol{\kappa}}
\def\bzero{\boldsymbol{0}}
\def\bPi{\mathbf{\Pi}}
\def\rH{\mathrm{H}}

\def\rP{\mathrm{P}}
\def\rV{\mathrm{V}}

\DeclareMathOperator{\spec}{sp}

\begin{document}
\begin{frontmatter}

\title{A posteriori analysis of a virtual element approach on polytopal meshes for the buckling eigenvalue problem}

\author[1]{Franco Dassi}
\ead{franco.dassi@unimib.it}
\address[1]{Dipartimento di Matematica e Applicazioni, Università degli studi di Milano Bicocca, Via Roberto Cozzi 55 - 20125 Milano, Italy.}
\author[2]{Andr\'es E. Rubiano}
\ead{andres.rubianomartinez@monash.edu}
\address[2]{School of Mathematics, Monash University, 9 Rainforest Walk, Melbourne, VIC 3800, Australia.}
\author[3]{Iv\'an Vel\'asquez}
\ead{ivan.velasquez@unimilitar.edu.co}
\address[3]{Departamento de Matem\'aticas, Universidad Militar Nueva Granada, Bogot\'a, Colombia.}

\begin{abstract} 
We introduce a novel residual-based a posteriori error estimator for the conforming $C^1$ Virtual Element Method (VEM) applied to the buckling eigenvalue problem, incorporating nonlinear plane stress effects in both two and three dimensions. The estimator is fully computable on general polyhedral meshes and implemented within the open-source \texttt{vem++} library. Its reliability is rigorously justified via bounds on the residual equation using polynomial projections, stabilisation contributions, and interpolation estimates, while efficiency is ensured through the use of bubble function arguments. Comprehensive numerical experiments in 2D and 3D illustrate the estimator’s optimal accuracy and robustness, highlighting its potential for predictive analysis of complex plate structures.

\end{abstract}

\begin{keyword} 
A posteriori error analysis in 2D and 3D 
\sep virtual element method
\sep buckling spectral problem.
\end{keyword}

\end{frontmatter}
\setcounter{equation}{0}
\section{Introduction}\label{SEC:Introduction}

The instability of thin plates under in-plane loading can be formulated as a buckling eigenvalue problem for the transverse displacement of the plate. The governing equation involves the fourth-order bending operator together with additional terms accounting for the effect of the (possibly nonlinear) in-plane stresses on the deflection. In this formulation, the eigenvalue represents a critical load parameter rather than a vibration frequency. Buckling occurs when the load reaches values for which the homogeneous problem admits nontrivial solutions. The corresponding eigenfunctions describe the buckling modes, while the smallest eigenvalue determines the first buckling load, marking the loss of stability of the flat configuration. This class of problems has attracted significant attention due to its importance in engineering applications, such as the design of automotive, underwater and aerospace structures, where predicting instability is crucial for safety and performance (e.g. \cite{BYUN2011274,YORK1998665}).

Several numerical methods have been proposed for the approximation of the buckling eigenvalue problem, including conforming, nonconforming, and mixed Finite Element Method (FEM) and Virtual Element Method (VEM) discretisations, as well as discrete singular convolution (DSC) methods. We refer to 
\cite{ADAK2023115763,Brenner_VK2017,CWCC2017,CKD,I2,MM2015,MoraVelasquez2020,Ra} for a non-exhaustive list of contributions on this topic. 

In this work, we employ a conforming VEM for the $H^2$ variational formulation of the buckling eigenvalue problem \cite{ABSV2016,BDR2019C1Polyhedral}. Unlike classical $C^1$ FEM (see e.g. \cite{ciarlet}), which typically require high polynomial degrees and a large number of degrees of freedom (DoFs) to ensure conformity, the lowest-order conforming VEM achieves this with significantly fewer DoFs ---specifically, 3 per vertex in 2D and 4 per vertex in 3D. Furthermore, the VEM naturally accommodates general polyhedral elements, including non-convex shapes, and efficiently handles hanging nodes, simplifying the implementation of adaptive refinement strategies.

A posteriori error estimation plays a central role in the numerical approximation of Partial Differential Equations (PDEs), particularly when the exact solution is unknown. These estimators are computable quantities that depend only on the discrete solution and the problem data, which provide practical information about the quality of a given numerical approximation. In particular, adaptive schemes use a posteriori error estimators to identify regions of the computational domain where the error is large, which enables a selective refinement of the mesh that leads to more efficient algorithms that concentrate computational effort where it is most needed. This is particularly relevant for problems exhibiting localised phenomena, such as boundary layers, singularities, or sharp gradients. 

From a mathematical viewpoint, an a posteriori error estimator should satisfy two key properties: reliability, meaning that it provides an upper bound for the true error, and efficiency, meaning that it also reflects the error from below, at least up to multiplicative constants, oscillation terms (polynomial projection or stabilisation terms), and higher-order terms that vanish for very small mesh size ($h\rightarrow0$). In the context of buckling eigenvalue problem a posteriori error estimation for FEM based discretisations of $H^1$ variational formulation have been proposed \cite{HLM2015,Stein1994AdaptiveFE} and adaptive based on FEM/VEM discretisations for $H^2$ variational formulation of the biharmonic problem can be found in \cite{CARSTENSEN2024,dassi2025posteriorierrorestimatesc1,feng2023,Li2018,MR4047014}.

\paragraph{Main contributions}
To the best of the author’s knowledge, this work is the first that:
\begin{itemize}
    \item Develop a reliable and efficient residual-based a posteriori error estimator for the $C^1$-VEM applied to the buckling eigenvalue problem in both 2D and 3D.
    \item Account for nonlinear plane stress effects in the a posteriori error estimator, enabling accurate prediction of buckling modes in plates under complex multi-axial loading.
    \item Provide an open-source implementation within the \texttt{vem++} library \cite{dassi2025vem++}, supporting computations on general polytopal meshes.
\end{itemize}

In Section~\ref{sec:main_problem}, we introduce the $H^2$ variational formulation for the buckling eigenvalue problem and its spectral characterisation. Section~\ref{SEC:DISCRETE} introduces the VEM space for both 2D and 3D together with a priori error estimates. The residual-based a posteriori error analysis is developed in Section~\ref{sec:apost}. Finally, several numerical results illustrating the accuracy and applicability of the method are provided in Section~\ref{sec:numerical_examples}.

\paragraph{Recurrent notation} This paper employs usual notation for differential operators in $\mathbb{R}^d$ ($d=2,3$). In particular, the gradient, vector divergence, and tensor divergence operators are denoted by $\nabla:\mathbb{R}\rightarrow \mathbb{R}^d$, $\vdiv:\mathbb{R}^d\rightarrow \mathbb{R}$, and $\bdiv:\mathbb{R}^{d\times d}\rightarrow \mathbb{R}^d$. The fourth-order biharmonic operator is denoted by $\Delta^2: \mathbb{R} \rightarrow \mathbb{R}$, where $\Delta = \vdiv(\nabla)$. Finally, the Hessian matrix consists of all second-order partial derivatives and is denoted by $\nabla^2: \mathbb{R}\rightarrow \mathbb{R}^{d\times d}$. 

We adopt standard notation for Sobolev spaces, together with the associated norms and semi-norms (see e.g. \cite{adams2003sobolev}). By $a \lesssim b$ we denote the inequality $a\leq C b$, where $C$ is a generic positive constant independent of the mesh size $h$ and can take different values on each occurrence. 

\section{The buckling eigenvalue problem}\label{sec:main_problem} Let $\Omega\subseteq \mathbb{R}^d$ be a polytopal bounded domain with boundary $\Gamma:=\partial \Omega$. The buckling mode $u:\Omega\rightarrow \mathbb{R}$, subjected to a plane stress tensor field $\bkappa : \Omega \rightarrow \mathbb{R}^{d\times d}$ with non-dimensional critical load factor $\lambda\in \mathbb{R}$ satisfy the following eigenvalue problem:
\begin{subequations}\label{eq:MPr}
    \begin{alignat}{2}
      \Delta^2 u &= -\lambda \vdiv (\bkappa\nabla u), \quad && \text{in } \Omega, \label{eq:balance}\\
      \mathcal{B}^{j}u &= 0,\quad  && \text{on } \Gamma. \label{eq:bc}
    \end{alignat}
\end{subequations}
Note that the non-dimensional critical load factor is normalised by multiplying by the factor $L/D$, where $L$ denotes the plate size and $D=(Et^{3})/(12(1-\nu^{2}))$ is the bending stiffness, with $E$, $t$, and $\nu$ representing the Young’s modulus, plate thickness, and Poisson’s ratio. Whereas, the plane stress tensor is assumed to satisfy the following conditions
\begin{align*}
    \bkappa \in [\bbL^\infty(\Omega)]^{d\times d}\setminus\{\bzero\}, \quad \bkappa = \bkappa^{\top}, \quad \bdiv(\bkappa) = \bzero, \quad \text{in } \Omega.
\end{align*}

The boundary conditions are defined through the linear differential operator $\mathcal{B}^{j}$ (cf. \eqref{eq:bc}) where $j\in\{\textbf{SSP},\textbf{CP}\}$ (see e.g. \cite[Section~2.3]{gazzola}). The notation \textbf{SSP} and \textbf{CP} indicate Simply Supported Plate and Clamped Plate type boundary conditions, defined respectively by 
\begin{subequations}
    \begin{alignat}{2}
        \mathcal{B}^{\textbf{SSP}}u&:=\Delta^{\beta-1}u=0,\quad &&\mbox{on }\Gamma,\label{eq:SSP}\\
        \mathcal{B}^{\textbf{CP}}u&:=\partial_{\bn}^{\beta-1}u=0,\quad &&\mbox{on }\Gamma,\label{eq:CP}
    \end{alignat}      
\end{subequations}
where $\beta \in \{1,2\}$, $\bn$ denotes be the outward unit normal vector to $\Gamma$, and $\partial_{\boldsymbol{n}}$ the normal derivative.

\subsection{Weak formulation}
In view of the boundary conditions \eqref{eq:SSP}-\eqref{eq:CP}, basic integration by parts lead to: find $(\lambda,u)\in \mathbb{R}\times (\rV\setminus \{0\})$ such that
\begin{align}\label{WMPr}
    a(u,v) = \lambda b(u,v), \quad \forall v\in \rV ,
\end{align}
where the bilinear forms $a: \rV\times \rV  \to \mathbb{R}$, $b:\rV\times \rV\to \mathbb{R}$ are defined by 
\begin{align*}
     a(u,v):=\int_{\Omega}\nabla^2 u: \nabla^2 v, \quad b(u,v):=\int_{\Omega}(\bkappa\nabla u)\cdot \nabla v.
\end{align*}
Without losing generality, the space $\rV$ refers to $H^2(\Omega)\cap H_0^1(\Omega)$ and $H_0^2(\Omega)$ for the \textbf{SSP} and \textbf{CP} boundary conditions, respectively. It is well-known that standard arguments in Sobolev spaces imply the continuity of $a(\cdot,\cdot)$ and $b(\cdot,\cdot)$. Whereas, the generalised Poincar\'e inequality implies the ellipticity of $a(\cdot,\cdot)$ on  $\rV$. Thus, a straight-forward application of the Lax–Milgram Theorem proves the well-posedness of \eqref{WMPr}.
\subsection{Spectral characterisation}
It is well known that the spectrum of problem \eqref{WMPr} can be characterised by the spectrum of a bounded, compact, and self-adjoint operator defined as follows (see, for instance, \cite{ciarlet}):

\begin{equation*}
    \begin{array}{rl}
&T:V\longrightarrow V  \\
& \hspace*{0.55cm} f \longmapsto  Tf,     
    \end{array}
\end{equation*}
where, $Tf$ is the unique solution of the following source problem:
\begin{equation}\label{defOfT}
a(Tf,v)=b(f,v),\quad \forall v\in V.
\end{equation}
Moreover, the following result summarises some important properties of the operator $T$, we refer to \cite{dassi2025posteriorierrorestimatesc1,MoraVelasquez2020} for further details.
\begin{theorem} The following statments hold true
\begin{enumerate}
    \item[(i)] A pair $(\lambda,u)$ solves \eqref{WMPr} if and only if $(\mu,u)$ satisfies $Tu = \mu u$, for $\lambda \neq 0$ and $\mu := \lambda^{-1}$.
    \item[(ii)] The spectrum of $T$  is given by $
\spec(T)=\{0 \}\cup \{ \mu_k\}_{k\in \mathbb{N}}$, where $\{ \mu_k\}_{k\in \mathbb{N}}\subseteq \mathbb{R}$ is a sequence converging to zero, and each eigenvalue $\mu_k$ has finite multiplicity.
\item[(iii)] If $\widetilde{T}:H_0^1(\Omega)\longrightarrow V$ denotes the extension of $T$ (cf. \eqref{defOfT}) to $H_0^1(\Omega)$, then $\spec(T)=\spec(\widetilde{T})$.
\end{enumerate}
\end{theorem}

\setcounter{equation}{0}
\section{A priori error analysis for the discrete spectral problem using virtual elements
}
\label{SEC:DISCRETE}
In this section, we introduce the VEM spaces on polygonal/polyhedral meshes related to the discretisation of \eqref{eq:MPr}. First, we define the virtual spaces and projections over facets (which in 2D reduces to the local space over polygons). Then, the bulk spaces are defined and glued together to construct the global VEM space. We refer to \cite{ABSV2016,BDR2019C1Polyhedral} for further details regarding the construction of the discrete spaces and well-posedness of the discrete problem.

Let $\Omega^h$ be a collection of polygonal/polyhedral meshes on $\Omega$ and $\mathcal{F}^h$ be the set of all facets (edges in 2D and faces in 3D ). The diameter of a polygon/polyhedron $K$ is denoted by $h_K$ and the diameter of a facet $f$ is given by $h_f$. The maximum diameter of elements in $\Omega^h$ is represented by $h$. It is assumed that there exists a uniform constant $\rho>0$ such that
\begin{enumerate}[label={\textbf{(M\arabic*)}}, align=left, leftmargin=*, labelwidth=!, labelsep=1em]
    \item \label{A1} Every element $K$ is star-shaped with respect to a ball with a radius greater than $\rho h_K$. 
    \item \label{A2} Every facet $f\in \partial K$ is star-shaped with respect to a ball with a radius greater than $\rho h_K$.
    \item \label{A3} Every facet $f\in \partial K$ satisfies the inequality $h_f \geq \rho h_K$. 
\end{enumerate}

We split the set of all facets as $\mathcal{F}^h = \mathcal{F}^h_\Omega \cup \mathcal{F}^h_{\Gamma}$, where $\mathcal{F}^h_\Omega = \{ f\in \mathcal{F}^h: f\subset \Omega\}$, $\mathcal{F}^h_{\Gamma} = \{ f\in \mathcal{F}^h : f\subset \Gamma\}$. The set of facets of $K\in \Omega^h$ is denoted as $\mathcal{F}^h(K)$, who in turn are classified as interior and boundary facets, denoted by $\mathcal{F}^h_\Omega(K)$ and $\mathcal{F}^h_{\Gamma}(K)$, respectively. The set of elements $K$ that share $f$ as a facet is denoted by $\Omega^h_f$ and the normal jump operator is defined as usual by $[\![\nabla u \cdot \boldsymbol{n}_f]\!]:= (\nabla u|_{K} - \nabla u|_{K'})|_f \cdot \boldsymbol{n}_f$, where $K,K' \in \Omega^h_f$, and $\boldsymbol{n}_f$ is the outward normal vectors of $f$ with respect to $\partial K$. 
In addition, we denote by $\mathcal{V}_{\Gamma}^h$ and $\mathcal{V}_{\Omega^o}^h$ the sets of boundary and interior vertices in $\overline{\Omega^h}$, respectively. 
On the other hand, $\Delta_{f}$ and $\nabla_{f}$ denote the respective Laplace and gradient operators in the local facet coordinates and $\partial_{{\boldsymbol{n}}_f^e}$ denotes the normal derivative on each ridge $e\in \partial f$ with respect to $f$ (similarly for $\partial_{{\boldsymbol{n}}_K^f}$). 

\subsection{Polynomial Projection operators}
The space of polynomials of total degree at most $k$ defined locally on $K\in \Omega^h$ (or facet $f\in \mathcal{F}^h$) is denoted by $\rP_k (K)$, and its vector counterparts is denoted by $\mathbf{P}_k(K)$. We also consider the standard notation $\rP_{-1}(K)=\{0\}$. In this paper, we consider the lowest-case order VEM discretisation where $k=2$. Then, we define local polynomial projection operators in an element $K\in \Omega^h$, as follows:
\begin{itemize}
        \item The vector $\bL^2$ projection $\bPi_{K}^{0}:\bL^2(K)\rightarrow \bP_{1}(K)$ with
    \begin{align}\label{PiVector}
        \int_K \left(\bv-\bPi_{K}^{0}\bv\right)\cdot \bp_{1} = 0, \ \forall \bp_{1} \in \bP_{1}(K).
    \end{align}
    \item The $\rH^1$-energy projection $\Pi_{K}^{\nabla}:\rH^{1}(K)\to\rP_{2}(K)$ is defined as
    \begin{subequations}
        \begin{align*}
            \int_K \nabla (\Pi_{K}^{\nabla} v_h -v_h)\cdot \nabla p_2  &= 0,\quad \forall p_2\in \rP_2(K),\\
            \int_{\partial K} (\Pi_{K}^{\nabla} v_h -v_h) &= 0.
        \end{align*}
    \end{subequations}
    \item The $\rH^2$-energy projection is defined as 
    $\Pi_{K}^{\nabla^2}:\rH^2(K)\to\rP_{2}(K)$ such that
    \begin{subequations}
        \begin{align*}
            \int_K \nabla^2 (\Pi_{K}^{\nabla^2} v_h -v_h): \nabla^2 p_2 &= 0, \quad \forall p_2\in \rP_2(K),\\
            \int_{\partial K} (\Pi_{K}^{\nabla^2} v_h -v_h)p_1&=0, \quad \forall p_1\in \rP_1(f).
        \end{align*}
    \end{subequations}
\end{itemize}
Notice that, the previous definitions are also valid for facets $f\in \mathcal{F}^h$, in this case we simply write $\Pi_{f}^{\nabla}$ and $\Pi_{f}^{\nabla^2}$. Moreover, the projections operators previously defined are computable directly from the degrees of freedom of the VEM discrete spaces (to be specified in Section~\ref{sec:discrete_spaces}), we refer to \cite{ABSV2016,BDR2019C1Polyhedral} for further details. 

We finalise by stating a result involving classical polynomial approximation theory, written for a general polynomial order of approximation in the elements $K\in \Omega^h$. Recall that similar estimates can be deduced for facets $f\in \mathcal{F}^h$ and vector valued functions (see e.g. \cite{BS-2008}). 
\begin{proposition}[polynomial approximation] \label{prop:est_poly}
    Given $K\in \Omega^h$, suppose that $v\in  \rH^{s}(K)$, with  $1 \leq s \leq k+1$. Then, there exist $v_\pi \in \rP_k(K)$ and a positive constant that depends only on $\rho$ ({cf.} \ref{A1}-\ref{A3}) such that for $0\leq r\leq s$ the following estimate holds
    \begin{align*}
        |v - v_\pi|_{r,K} &\lesssim  h_K^{s-r}|v|_{{s},K}.
    \end{align*} 
\end{proposition}

\subsection{Discrete spaces}\label{sec:discrete_spaces}
We start by considering the enhanced space on facets $f\in \partial K$ (see \cite{ABSV2016}), given by
\begin{align*}
    \rV^h_{2\mathrm{D}}(f):=\Big\{
    v_h\in \rH^2(f) : \,&\Delta^2_{f}v_h\in\rP_{2}(f), \\
    & v_h|_{\partial f}\in C^0(\partial f),\, v_h|_e\in\rP_3(e),\quad \forall e\in\partial f,\\
    & \nabla_{f} v_h|_{\partial f}\in [C^0(\partial f)]^2,\, \partial_{{\boldsymbol{n}}_f^e} v_h|_e\in\rP_1(e),\quad\forall e\in\partial f, \\
    &\int_f \Pi_{f}^{\nabla^2}v_h \,p_2=\int_f v_h\,p_{2}, \quad \forall p_{2}\in \rP_{2}(f) \Big\}.
\end{align*}
Notice that in a 2D setting, one can identify the facets $f$ and ridge $e\in \partial f$ in the previous space as polygonal elements $K$ and edges of the polygon $e$, respectively. Such identifications allow us to define directly the global 2D VEM space by 
\begin{align*}
    \rV^h_{2\mathrm{D}}  :=\left\{v_h \in  \rV  :\ v_h|_K\in  V^h_{2\mathrm{D}}(K)\right\}.
\end{align*}
We remark that the space $\rV^h_{2\mathrm{D}}(f)$ is richer to its 3D counterpart introduced in \cite{BDR2019C1Polyhedral}, where the condition $\Delta_f^2 v_h \in \rP_1(f)$ is imposed and the enhancement property is enforced only with respect to linear polynomials.

Next, following the hierarchical construction of VEM spaces in 3D, we introduce an additional enhanced space on facets in order to represent normal derivatives of virtual functions in facets, as follows
\begin{align*}
    \rV^h_{\partial}(f):=\Big\{
    v_h\in \rH^1(f):\, &\Delta_{f} v_h \in \rP_0(f), \\
    &v_h|_{\partial{f}}\in C^0(\partial f),\, v_h|_{e}\in \mathbb{P}_1(e), \quad \forall e\in \partial f,\\
    &\int_f \Pi_{f}^{\nabla}v_h=\int_f v_h \Big\}.
\end{align*}
With these building blocks, we are ready to define the local three-dimensional VEM space on a polyhedron $K\in\Omega^h$ as
\begin{align*}
    \rV^h_{\mathrm{3D}}(K):=\Bigg\{v_h\in \rH^2(K) :\, &\Delta^2v_h\in\rP_{2}(K),\\
    &v_h|_{e}\in C^0(e),\, \nabla v_h|_{e}\in [C^0(e)]^3,\quad \forall e\in\bigcup_{f\in\partial K}\partial f,\\
    &v_h|_f\in \rV^h_{2\mathrm{D}}(f),\, \partial_{{\boldsymbol{n}}_K^f} v_h|_f\in  \rV^h_{\partial}(f),\quad  \forall f\in\partial K,\\
    &\int_K \Pi_{K}^{\nabla^2}v_h\,p_2 = \int_K  v_h\,p_2, \quad \forall p_2\in \rP_2(K)\Bigg\},
\end{align*}
Similarly, the global space is obtained by gluing together the local spaces in the following way
\begin{equation*}
\rV^h_{3\mathrm{D}}:=\Big\{v_h\in \rV: v_h|_{K}\in V^h_{3\mathrm{D}}({K})\Big\}\,.
\end{equation*}

Following \cite{ABSV2016,BDR2019C1Polyhedral}, a unisolvent set of degrees of freedom for $\rV^h(K)$ can be chosen as
\begin{itemize}
    \item The values of $v_h$ at the vertices of $K$,
    \item The values of $\nabla v_h$ at the vertices of $K$.
\end{itemize}
Thanks to this selection, we have that
\begin{align*}
\mathrm{dim}(\rV^h_{2\mathrm{D}})
=\begin{cases}
3\#(\mathcal{V}_{\Omega^o}^h)&\mbox{for {\bf CP}},\\
3\#(\mathcal{V}_{\Omega^o}^h)-2\#(\mathcal{V}_{\Gamma}^h)&\mbox{for {\bf SSP}}.
\end{cases}
\end{align*}
and
\begin{align*}
\mathrm{dim}(\rV^h_{3\mathrm{D}})
=\begin{cases}
4\#(\mathcal{V}_{\Omega^o}^h)&\mbox{for {\bf CP}},\\
4\#(\mathcal{V}_{\Omega^o}^h)-3\#(\mathcal{V}_{\Gamma}^h)&\mbox{for {\bf SSP}}.
\end{cases}
\end{align*}

Finally, we introduce the notation $\rV_h(K)$ (resp. $\rV_h$) to denote simultaneously the local spaces $\rV^h_{\mathrm{2D}}(K)$ and $\rV^h_{3\mathrm{D}}(K)$ (resp. global spaces $\rV^h_{\mathrm{3D}}$ and $\rV^h_{2\mathrm{D}}$).  
Furthermore, the previous hierarchical construction is $C^1$ continuous along the elements' facets, which turns into discrete functions that are $C^1$ continuous globally.

\subsection{Discrete weak formulation}
We start by defining locally the computable discrete bilinear forms for all $u_h,v_h\in\rV_h$ as follows
\begin{align*}
    a^h_K(u_h,v_h)&:=\int_{K}\nabla^2 (\Pi_{K}^{\nabla^2} u_h): \nabla^2 (\Pi_{K}^{\nabla^2} v_h) + S_K^{\nabla^2}(u_h-\Pi_{K}^{\nabla^2} u_h,v_h-\Pi_{K}^{\nabla^2} v_h), \\
    b^h_K(u_h,v_h)&:=\int_{K}(\bkappa\bPi_{K}^{0}(\nabla u_h))\cdot \bPi_{K}^{0}(\nabla v_h),
\end{align*}
where the operator $S_{K}^{\nabla^2}
:\rV\times\rV\to \mathbb{R}$ is any symmetric positive bilinear form that satisfies  
\begin{alignat}{2}
    \alpha_{*}a_{K}(v_h,v_h)&\leq S_{K}^{\nabla^2}(v_h,v_h)\leq  \alpha^{*}a_{K}(v_h,v_h),\quad  &&\forall v_h\in \ker(\Pi_{K}^{\nabla^2}),\label{eq:stabA}
\end{alignat}
here, 
$\alpha_*,\alpha^*$
are positive constants independent of the element size $h_K$. Notice that the properties of the stabilization operator are given in \eqref{eq:stabA}
is crucial to prove the continuity and ellipticity of $a^h(\cdot,\cdot)$ , $b^h(\cdot,\cdot)$ 
(defined as the sum of the local contributions for all $K\in\Omega^h$) and ellipticity of $a^h(\cdot,\cdot)$, we refer to \cite{DV_camwa2022,MV2} for further details.

Now we are ready to introduce the VEM discrete weak formulation of \eqref{eq:MPr} read as follows: find $(\lambda_h,u_h)\in \mathbb{R}\times (\rV_h\setminus \{0\})$ such that
\begin{align}\label{WDPr}
    a^h(u_h,v_h) = \lambda_h b^h(u_h,v_h), \quad \forall v_h\in \rV_h.
\end{align}

\subsection{The a priori error analysis}
\label{Sec:aPriori}
We start recalling the orthogonal projection $\mathcal{P}_h^{\nabla}:H_0^1(\Omega)\to V_h$ defined as follows: For all $\phi\in H_0^1(\Omega)$ we set  $\mathcal{P}_h^{\nabla}\phi\in V_h$ as unique solution of 
$$ \int_\Omega \nabla (\mathcal{P}_h^{\nabla} \phi-\phi) \cdot \nabla v_h,\quad  \forall v_h\in V_h.$$
\begin{remark}\label{RemDefPos}
    If $\bkappa$  additionally satisfies the positive definiteness assumption, then the operator $\mathcal{P}_h^{\nabla }:H_0^1(\Omega)\to V_h$ given by     $$b(\mathcal{P}_h^{\nabla}\phi-\phi,v_h)=0,\quad \forall v_h \in V_h,$$
    is well defined.
\end{remark}

Next, we consider the discrete source operators 
\begin{equation*}
    \begin{array}{rl}
&T_h:V_h\longrightarrow V_h  \\
& \hspace*{0.725cm} f_h \longmapsto  T_hf_h      
    \end{array}
\qquad 
\mbox{and}
\qquad 
 \begin{array}{rl}
&\widehat{T}_h:H_0^1(\Omega)\longrightarrow V_h\subset H_0^1(\Omega)  \\
& \hspace*{1cm} f \hspace*{0.475cm}\longmapsto  \widehat{T}_hf:=T_h\mathcal{P}_h^{\nabla}f 
\end{array},
\end{equation*}
where $T_hf_h$ and $T_h\mathcal{P}_h^{\nabla}f$ are the unique solutions of the following discrete source problems
\begin{equation}\label{defTh_hat}
 a^h(T_hf_h,v_h)=b^h(f_h,v_h)
\quad
\mbox{and}
\quad
 a^h(T_h\mathcal{P}_h^{\nabla}f,v_h)=b^h(\mathcal{P}_h^{\nabla}f,v_h),
\end{equation}
for all $v_h\in V_h$, respectively. It is well known that the operators $\widehat{T}_h$ and $T_h$ satisfy  $\spec(\widehat{T}_h)=\spec(T_h)$ and they have the same eigenfunctions (see for instance \cite{DV_camwa2022}). In addition, if $Z_h:=\{u_h\in V_h: b^h(u_h,v_h)=0 \quad \forall v_h\in V_h\}$, we readily see that the spectrum $\mathrm{sp}(T_h)$ consists of $\mathrm{dim}(V_h)-\mathrm{dim}(Z_h)=:M$ eigenvalues repeated according to their respective multiplicity. Moreover, it holds that $\sp(T_h)=\{\mu_h\}_{k=1}^{M}\cup \{ 0 \}$.  

On the other hand, for $(T,\mu)$ and $(T_h,\mu_h)$, we recall the  definitions of their spectral projections $E$ and $E_h$ as
\begin{align}
    &E:=E(\mu)=(2\pi)^{-1}\int_\Omega (z-T)^{-1}dz,\label{defOpE}\\
    &E_h:=E_h(\mu)=(2\pi)^{-1}\int_\Omega (z-T_h)^{-1}dz\nonumber,
\end{align}
where $\mu_h:=1/\lambda_h, (\lambda_h\neq 0)$. Finally, we establish error estimates for the solution of the source problems in $H^2$ and $H^1$ norms (see for instance \cite{dassi2025posteriorierrorestimatesc1} and \cite{MV2}).
\begin{lemma}\label{Lema8ApVEP}
    For $s\in (1/2,1]$ it holds
    \begin{align}
        &||(T-\widehat{T}_h)(f)||_{2,\Omega}\lesssim h^s||f||_{2,\Omega},\quad \forall f\in V,\nonumber\\
        &||(\widetilde{T}-\widehat{T}_h)(f)||_{1,\Omega}\lesssim h^s||f||_{1,\Omega},\quad \forall f\in H_0^1(\Omega).\nonumber
    \end{align}
\end{lemma}

We now prove an useful relation between the continuous and discrete bilinear forms $b(\cdot,\cdot)$ and~$b^h(\cdot,\cdot)$.
\begin{lemma}\label{LemAuxiliar}
    For all $v_h,w_h\in V_h$ it holds
    \begin{equation*}
        b(v_h,w_h)-b^h(v_h,w_h)=\sum\limits_{K\in \Omega^h} \int_K \Big(\bkappa \nabla v_h-\bPi_K^0(\bkappa \nabla v_h) \Big)\cdot\Big( \nabla w_h - \bPi_K^0 \nabla w_h \Big) 
    \end{equation*}
\end{lemma}
\begin{proof}
Let  $v_h,w_h\in V_h$ and  $K\in \Omega^h$. From the definition of the projector $\bPi_K^0$ (cf. \eqref{PiVector}) and some algebraic manipulations, we readily see that
\begin{align*}
b_K(v_h,w_h)-b_K^h(v_h,w_h)
& =\int_K  \Big\{ \bkappa \nabla v_h\cdot \nabla w_h - \bkappa \bPi_K^0 \nabla v_h \cdot \bPi_K^0 \nabla w_h  \Big\}  \\
&=\int_K  \Big\{ \bkappa \nabla v_h\cdot \nabla w_h - \bPi_K^0(\bkappa \bPi_K^0 \nabla v_h) \cdot  \bPi_K^0 \nabla w_h  \Big\}\nonumber\\
&=\int_K \Big\{ [\bkappa \nabla v_h - \bPi_K^0(\bkappa \nabla u_h) ]\cdot \nabla w_h + \bPi_K^0(\bkappa \nabla v_h )\cdot [\nabla w_h-\bPi_K^0 \nabla w_h]\nonumber\\
&\quad + \bPi_K^0 \nabla w_h\cdot [\bPi_K^0(\bkappa \nabla v_h)- \bkappa \nabla v_h] + \bPi_K^0 \nabla w_h\cdot \bkappa[ \nabla v_h- \bPi_K^0 \nabla v_h] \\
&\quad + \bPi_K^0 \nabla w_h\cdot[\bkappa \bPi_K^0 \nabla v_h - \bPi_K^0(\bkappa \bPi_K^0 \nabla v_h) ] 
\Big\} \nonumber\\
&= \int_K  [\bkappa \nabla v_h - \bPi_K^0(\bkappa \nabla u_h) ]\cdot \nabla w_h\\
&= \int_K  [\bkappa \nabla v_h - \bPi_K^0(\bkappa \nabla u_h) ]\cdot [\nabla w_h -\bPi_K^0 \nabla w_h]\nonumber.
 \end{align*}
By taking sum over all $K\in \Omega^h$, the result is obtained.
\end{proof}
We now introduce an error estimate for $\lambda_h$ as a consequence of Lemmas~\ref{Lema8ApVEP} and \ref{LemAuxiliar}.
\begin{theorem}\label{th:convergenceRates}
For all space $R(E)\subseteq H^{2+s}(\Omega)$ with $s \in (1/2,1]$ there exists $h_0>0$ independent of $h$ such that for all $h<h_0$ it hold
\begin{align}
    \left|\l-\l_h^{(j)}\right|&\lesssim  h^{2s},\quad j=1,\ldots,m, \quad \text{and} \nonumber\\
    |\lambda - \lambda_h^{(j)}|&\lesssim \Bigg\{|u-u_h|_{2,\Omega}^2+|u-\Pi_K^{\nabla^2} u_h|_{2,h}^2 +||\bkappa \nabla u_h - \bPi_K^0(\bkappa \nabla u)   ||_{0,\Omega} ||\nabla u_h -\bPi_K^0( \nabla u_h)  ||_{0,\Omega} \Bigg\}.\nonumber
\end{align}
\end{theorem}
\begin{proof}
The first inequality follows as a direct consequence of Lemma~\ref{Lema8ApVEP}. Regarding the second inequality, following the arguments used in the proof of the  \cite[Theorem~4.4]{MoraVelasquez2020} and applying Lemma~\ref{LemAuxiliar}, we readily see that 
\begin{align*}
    (\lambda_h^{(j)}-\lambda)b^h(u_h,u_h)&= a(u-u_h,u-u_h) - \lambda b(u-u_h,u-u_h) + \big\{  a^h(u_h,u_h)-a(u_h,u_h) \big\} \nonumber\\
    & \quad + \lambda \big\{ b(u_h,u_h)-b^h(u_h,u_h)\big\} \nonumber\\
    &\lesssim |u-u_h|^2_{2,\Omega} + ||\bPi_K^0\nabla(u-u_h)||_{0,\Omega}^2 + |u-u_h|_{2,\Omega}^2 + |u-\Pi_K^{\nabla^2}u_h|_{2,h}^2\nonumber\\ 
    &\quad +||\bkappa \nabla u_h - \bPi_K^0(\bkappa \nabla u)   ||_{0,\Omega} ||\nabla u_h -\bPi_K^0(\nabla u_h)  ||_{0,\Omega}\nonumber\\
    &\lesssim |u-u_h|^2_{2,\Omega}  + |u-\Pi_K^{\nabla^2}u_h|_{2,h}^2\\
    &\quad +||\bkappa \nabla u_h - \bPi_K^0(\bkappa \nabla u)   ||_{0,\Omega} ||\nabla u_h -\bPi_K^0(\nabla u_h)  ||_{0,\Omega}\nonumber,
\end{align*}
where, in the last step we have used the boundedness of $\bPi_K^0$ and  Poincar\'e inequality (cf. $u-u_h\in V$). Moreover, the ellipticity of $ a^h(\cdot,\cdot)$ implies that
\begin{equation*}
    |b^h(u_h,u_h)|=|\lambda_h^{-1} a^h(u_h,u_h)|\gtrsim |\lambda_h^{-1}| > 0,
\end{equation*}
and the result follows.
\end{proof}

An error estimate for the solutions of the source problems \eqref{defOfT} and \eqref{defTh_hat} is given below.
\begin{lemma}\label{Lem10APVEP}
    Given any $f\in R(E)$ set $w:=Tf$ and $w_h:=\widehat{T}_hf$, then,  the following estimate holds
\begin{align}
|w-w_h|_{2,\Omega}
&\lesssim h^s\Big\{|w-w_h|_{2,\Omega} + |w-\Pi^{\nabla^2}w_h|_{2,h} \Big\} \nonumber\\
&\quad + h^{1+s}  ||\bkappa \nabla w - \bPi_K^0(\bkappa \nabla w)  ||_{0,\Omega} v  + |w-\mathcal{P}_h^\nabla w|_{1,\Omega}, \nonumber
\end{align}
where $R(E)$ denotes the range of the operator $E$ (cf. \eqref{defOpE}).
\end{lemma}
\begin{proof}
Given $f\in R(E)$, set $w:=Tf$ and $w_h:=\widehat{T}_hf$ as the solutions of problems \eqref{defOfT} and  \eqref{defTh_hat} (respectively), let $v\in V$ the unique solution of the following weak  variational formulation: find $v\in V$, such that \begin{equation}\label{AuxProbDual}
    a(v,z)=\int_\Omega  \nabla(w-w_h)\cdot \nabla z\quad \forall z\in V. 
\end{equation} 
It is easy to check that $T(w-w_h)=v$ (cf. \eqref{defOfT}), and the Lax-Milgram Theorem implies $|v|_{2,\Omega}\lesssim ||\bkappa||_{\infty,\Omega}|w-w_h|_{1,\Omega}$. In addition, since $v\in H^{2+s}(\Omega)$, we have that there exists $v_I\in V_h$ such $|v-v_I|_{2,\Omega}\lesssim h^s |v|_{2+s,\Omega}$ (cf. \cite{G}). Then, testing \eqref{AuxProbDual} with $w-w_h\in V$, we get
\begin{align}
 |w-w_h|_{1,\Omega}^2&=   \int_\Omega  \nabla(w-w_h)\cdot \nabla (w-w_h)\nonumber\\
 &=a(v-v_I,w-w_h)+a(v_I,w-w_h)\nonumber\\
 &=  a(v-v_I,w-w_h)+a(w,v_I) - a(w_h,v_I) +  a^h(w_h,v_I)- a^h(w_h,v_I)\nonumber\\
 &=  a(v-v_I,w-w_h) + \{ a^h(w_h,v_I) -  a(w_h,v_I)\} +\{ b(f,v_I) -b^h(\mathcal{P}_h^\nabla f,v_I)\} .\nonumber 
\end{align}
From \cite[Lemma~10]{dassi2025posteriorierrorestimatesc1}, we have that
\begin{align}
 |w-w_h|_{1,\Omega}^2 &\leq h^s\Big\{|w-w_h|_{2,\Omega} + |w-\Pi^{\nabla^2}w_h|_{2,h}\Big\}|w-w_h|_{1,\Omega}  + \big\{ b(f,v_I) - b^h(\mathcal{P}_h^{\nabla} f,v_I)\big\} .\label{aux0_teo10}
\end{align}

Now, by adding and subtracting the term $b(\mathcal{P}_h^\nabla f,v_I)$,  applying  Lemma~\ref{LemAuxiliar} and $f=\mu^{-1}w$ in the last term of \eqref{aux0_teo10}, we arrive at
\begin{align}
b(f,v_I) - b^h(\mathcal{P}_h^{\nabla} f,v_I)  &=   b(f-\mathcal{P}_h^\nabla f,v_I) + \{ b(\mathcal{P}_h^{\nabla} f,v_I)-b^h(\mathcal{P}_h^{\nabla} f,v_I) \}\nonumber\\
&= b(f-\mathcal{P}_h^\nabla f,v_I) \nonumber\\
&\quad + \sum\limits_{K\in \Omega^h} \int_K [\bkappa \nabla(\mathcal{P}_hf)-\bPi_K^0(\bkappa \nabla(\mathcal{P}_hf))]\cdot [\nabla v_I - \bPi_K^0(\nabla v_I)]\nonumber\\
&= \mu^{-1}\Bigg\{ b(w-\mathcal{P}_h^\nabla w,v_I) \nonumber\\
&\quad + \sum\limits_{K\in \Omega^h} \int_K [\bkappa \nabla(\mathcal{P}_h^\nabla w)-\bPi_K^0(\bkappa \nabla(\mathcal{P}_h^\nabla w))]\cdot [\nabla v_I - \bPi_K^0(\nabla v_I)]\Bigg\}.\label{aux1_teo10}
\end{align}
Next, by adding and subtracting the terms $\nabla v$ and $\bPi_K^0 \nabla v$, applying the Cauchy-Scwharz inequality and the additional regularity of $v$ in  \eqref{aux1_teo10}, we obtain
\begin{align}
b(f,v_I) - b^h(\mathcal{P}_h^{\nabla} f,v_I)  \lesssim |b(w-\mathcal{P}_h^\nabla w,v_I)| + h^{1+s}||\bkappa \nabla(\mathcal{P}_hw)-\bPi_K^0(\bkappa \nabla(\mathcal{P}_hw)) ||_{0,\Omega}|w-w_h|_{1,\Omega}\nonumber.
\end{align}
Note that the addition and subtraction of the terms $\bkappa \nabla w$ and $\bPi_K^0 (\bkappa\nabla w)$ together with the triangle inequality, the assumption $\bkappa \in L^{\infty}(\Omega)
$ and the fact that $|w-\mathcal{P}_hw|_{1,\Omega}=\inf\limits_{z_h\in V_h}|w-z_h|_{1,\Omega}$ applied to the last term of the above expression, we obtain
\begin{align}
    b(f,v_I) - b^h(\mathcal{P}_h^{\nabla} f,v_I) 
    &\lesssim |b(w-\mathcal{P}_h^\nabla w,v_I)|+ h^{1+s}\Big\{ |\mathcal{P}_h^\nabla w-w|_{1,\Omega} \nonumber\\
    &\quad + ||\bkappa \nabla  w - \bPi_K^0(\bkappa \nabla w)  ||_{0,\Omega}+ |w-\mathcal{P}_h^\nabla w|_{1,\Omega} \Big\}|w-w_h|_{1,\Omega}\nonumber\\
        &\lesssim |b(w-\mathcal{P}_h^\nabla w,v_I)| \nonumber \\
        &\quad + h^{1+s}\Big\{ |w-w_h|_{1,\Omega} + ||\bkappa \nabla w - \bPi_K^0(\bkappa \nabla w)  ||_{0,\Omega}\Big\}|w-w_h|_{1,\Omega}\label{aux3cotateo10}.
\end{align}

Finally, by inserting the estimate \eqref{aux3cotateo10} in the right-hand side of \eqref{aux0_teo10} and multiplying by $|w-w_h|_{1,\Omega}^{-1}$ in \eqref{aux0_teo10} we have
\begin{align}
|w-w_h|_{2,\Omega}
&\lesssim h^s\Big\{|w-w_h|_{2,\Omega} + |w-\Pi^{\nabla^2}w_h|_{2,h}\Big\} \nonumber \\
& \quad + h^{1+s}\Big\{ |w-w_h|_{1,\Omega} + ||\bkappa \nabla w - \bPi_K^0(\bkappa \nabla w)  ||_{0,\Omega}\Big\} + |b(w-\mathcal{P}_h^\nabla w,v_I)| \nonumber\\
&\lesssim h^s\Big\{|w-w_h|_{2,\Omega} + |w-\Pi^{\nabla^2}w_h|_{2,h} \Big\}\nonumber\\
& \quad + h^{1+s}  ||\bkappa \nabla w - \bPi_K^0(\bkappa \nabla w)  ||_{0,\Omega} v + |b(w-\mathcal{P}_h^\nabla w,v_I)| \nonumber\\
&\lesssim h^s\Big\{|w-w_h|_{2,\Omega} + |w-\Pi^{\nabla^2}w_h|_{2,h} \Big\} \nonumber\\
&\quad + h^{1+s}  ||\bkappa \nabla w - \bPi_K^0(\bkappa \nabla w)  ||_{0,\Omega} v  + |w-\mathcal{P}_h^\nabla w|_{1,\Omega}. \nonumber
\end{align}
\end{proof}
As a consequence of the previous result, we have
\begin{corollary}
If the assumption in Remark~\ref{RemDefPos} is satisfied, then
\begin{align}
|w-w_h|_{2,\Omega}&\lesssim h^s\Big\{|w-w_h|_{2,\Omega} + |w-\Pi^{\nabla^2}w_h|_{2,h} \Big\}+ h^{1+s} ||\bkappa \nabla w - \bPi_K^0(\bkappa \nabla w)  ||_{0,\Omega},
\nonumber
\end{align}
for all $f\in R(E)\subset H^{2+s}(\Omega)$ such that $w:=Tf$ and $w_h:=\widehat{T}_hf$.
\end{corollary}

We finalise by introducing an error estimate in the $H^1$ semi-norm for the spectral problem~\eqref{WDPr}.
\begin{lemma}\label{Lem11ApVEP}
   For all eigenpairs $(\mu_h^{(j)},u_h)$ of $\widehat{T}_h$ such that $||u_h||_{1,\Omega}=1$ $(j=1,2,...,m)$, there exists an eigenfunction $u\in H_0^1(\Omega)$ of $T$ associated with $\mu$ such that
\begin{align*}
|u-u_h|_{1,\Omega}\lesssim h^s\Big\{|w-w_h|_{2,\Omega} + |w-\Pi_K^{\nabla^2}w_h|_{2,h}\Big\}  + h^{1+s}||\bkappa \nabla w - \bPi_K^0(\bkappa \nabla w)  ||_{0,\Omega}   + |w-\mathcal{P}_h^{\nabla} w|_{1,\Omega}.
\end{align*}
\end{lemma}

\begin{proof}
The proof follows with the same arguments presented in \cite[Lemma~11]{dassi2025posteriorierrorestimatesc1}.
\end{proof}

\begin{theorem}\label{ApTeo12VEP}
If $u$ and $u_h$ solve the continuous and discrete spectral problems \eqref{WMPr} and \eqref{WDPr}, respectively. Then, there exists $s\in (1/2,1]$ such that 
\begin{align}
&|u-u_h|_{1,\Omega}\lesssim h^s\Big\{|u-u_h|_{2,\Omega} + |u-\Pi_K^{\nabla^2}u_h|_{2,h} \Big\}  + h^{1+s}||\bkappa \nabla u - \bPi_K^0(\bkappa \nabla u)  ||_{0,\Omega}  + |u-\mathcal{P}_h^{\nabla} u|_{1,\Omega}. \nonumber
\end{align}   
\end{theorem}
\begin{proof}
The proof can be obtained from Lemma~\ref{Lem11ApVEP} with the same arguments as used in the proof of \cite[Theorem~12]{dassi2025posteriorierrorestimatesc1} and 
    \begin{align}
||\bkappa \nabla w - \bPi_K^0(\bkappa \nabla w)  ||_{0,\Omega}\lesssim |\mu^{-1}|\ 
||\bkappa \nabla u - \bPi_K^0(\bkappa \nabla u)  ||_{0,\Omega}.\nonumber
\end{align}
\end{proof}



\setcounter{equation}{0}
\section{Residual-based a posteriori error analysis} \label{sec:apost}
\label{SEC:EST_A_POST}
This section aims to define a residual-type estimator based on computable quantities of the VEM solution. In addition, we establish the reliability and efficiency of the proposed estimator.The upper bound follows from the residual equation combined with polynomial projections, stabilization terms, and interpolation estimates, while the lower bound is guaranteed through properties of bubble functions. 

\subsection{Error estimators} 
For each $K\in \Omega^h$ and $f\in \partial K$, the local estimators corresponding to the volume residual, the jump residual, the data oscillation, and the stabilisation term are defined by
\begin{subequations}\label{terms_estimator}
    \begin{align}
        \Xi_{K}^2 &:= h_K^4 \norm{\lambda_h \vdiv\left(\bPi_K^0(\bkappa \bPi_K^0(\nabla u_h))\right)}_{0,K}^2,\label{vol_est}\\
        \cJ_{f}^2 &:= h_f \norm{\jump{(\nabla^2\Pi_K^{\nabla^2} u_h )\bn_K^f}}_{0,f}^2 + h_f^{3}\norm{\jump{\lambda_h\bPi_K^0(\bkappa \bPi_K^0(\nabla u_h ))\cdot \bn_K^f}}_{0,f}^2,\label{jump_est}\\
        \Lambda_K^2 & := h_K^2\norm{\lambda_h \left(\bkappa \bPi_K^0(\nabla u_h)-\bPi_K^0(\bkappa \bPi_K^0(\nabla u_h))\right)}_{0,K}^2, \label{data_est}\\
        S_{K}^2 &:= S_K^{\nabla^2}(u_h-\Pi_K^{\nabla^2} u_h,u_h-\Pi_K^{\nabla^2} u_h) 
        .\label{stab_est}
    \end{align}
\end{subequations}
Notice that, the volume residual $\Xi^2 := \sum_{K\in \Omega^h} \Xi_K^2$ contains a polynomial approximation of the right-hand related to the buckling eigenvalue problem (cf. \eqref{eq:balance}); recovering the full residual requires the use of higher-order approximations. In addition, the jump residual $\cJ^2 := \sum_{f\in \cF^h_\Omega} \cJ_f^2$, and the stabilisation term $S^2 := \sum_{K\in \Omega^h} S_K^2$ are composed by computable polynomial projections of the discrete solution $u_h$. On the other hand, the data oscillation $\Lambda^2 := \sum_{K\in \Omega^h} \Lambda_K^2$ measures the polynomial approximation error of $\bkappa \nabla \Pi_K^{\nabla} u_h$. Finally, the respective global total error estimator is defined as follows:
\begin{align*}
\eta^2:= \sum_{K\in \Omega^h} \eta_K^2 = \sum_{K\in \Omega^h} \left[\Xi_K^2 + \Lambda_K^2 + S_K^2\right] + \sum_{f\in \cF^h_\Omega} \cJ_f^2 =  \Xi^2 + \cJ^2 + \Lambda^2 + S^2. 
\end{align*}

\subsection{Reliability}\label{Reliablity}
This section aims to prove an upper bound of the error in terms of the a posteriori error estimator. We start by rewriting the term $a(u-u_h,v)$ in terms of residual equations, from which the local error estimators emerge in a natural way during the subsequent analysis.

\begin{lemma}\label{Lem13ApVEP}
    If  $u$ and $u_h$ solve the spectral problems \eqref{WMPr} and \eqref{WDPr}, then for every $v\in  \rV$, the following identity holds:
    \begin{align*}
        a(u-u_h,v) &= \lambda b(u,v)-\lambda_h b( u_h,v) + \sum_{K\in \Omega^h} a_K(\Pi_K^{\nabla^2} u_h-u_h,v) \\
        &\quad + \lambda_h \sum_{K\in \Omega^h} \int_K (\bkappa(\nabla u_h - \bPi_K^0 (\nabla u_h)))\cdot \nabla v - \sum_{K\in \Omega^h} \int_{K}\lambda_h \vdiv\left( \bPi_K^0 (\bkappa \bPi_K^0(\nabla u_h))\right) v \\
        &\quad - \sum_{f\in \cF^h_\Omega} \int_f \jump{(\nabla^2\Pi_K^{\nabla^2} u_h )\bn_K^f}\cdot \nabla v + \sum_{f\in \mathcal{F}^h_\Omega} \int_f
        \jump{\lambda_h\bPi_K^0(\bkappa \bPi_K^0(\nabla u_h ))\cdot \bn_K^{f}}v \\
        &\quad + \sum_{K\in \Omega^h}\int_K \lambda_h \left(\bkappa \bPi_K^0(\nabla u_h)-\bPi_K^0(\bkappa \bPi_K^0(\nabla u_h))\right) \cdot \nabla v.
    \end{align*}
\end{lemma}
\begin{proof}
Fix $v\in \rV$. Since $(\lambda,u)$ solves the weak buckling eigenvalue problem (cf. \eqref{WMPr}), basic algebraic manipulations lead to
\begin{align} \label{first_part_residual}
a(u-u_h,v) &= \lambda b(u,v)-\lambda_h b(u_h,v) + \lambda_h b(u_h,v)-a(u_h,v) \notag\\
&= \lambda b(u,v)-\lambda_h b(u_h,v)  + \sum_{K\in \Omega^h} \left[-a_K (u_h-\Pi_K^{\nabla^2} u_h,v) -  a_K(\Pi_K^{\nabla^2} u_h,v)\right] \notag\\
& \quad + \lambda_h\sum_{K\in \Omega^h} \int_K (\bkappa(\nabla u_h - \bPi_K^0 (\nabla u_h))\cdot \nabla v + \lambda_h\int_K \bPi_K^0(\bkappa \bPi_K^0 (\nabla u_h))\cdot \nabla v \notag\\
& \quad +\lambda_h \int_K \left[\bkappa \bPi_K^0 (\nabla u_h)-\bPi_K^0(\bkappa \bPi_K^0 (\nabla u_h))\right]\cdot \nabla v.
\end{align}
Next, integration by parts and using that $\Pi_K^{\nabla^2} u_h\in \rP_2(K)$ imply that 
\begin{align}
    \sum_{K\in\Omega^h} a_K(\Pi_K^{\nabla^2} u_h,v) &= \sum\limits_{K\in \Omega^h} \int_{\partial K} \left(\nabla^2(\Pi_K^{\nabla^2} u_h)\bn_{K}\right)\cdot \nabla v. \label{second_part_residual}
    \end{align}
    Whereas, integration by parts lead to 
    \begin{align}        
    \sum_{K\in \Omega^h} \int_K \bPi_K^0(\bkappa \bPi_K^0 (\nabla u_h))\cdot \nabla v &= - \sum_{K\in \Omega^h} \int_K \vdiv\left(\bPi_K^0(\bkappa \bPi_K^0(\nabla u_h))\right) v \notag \\
    & \quad + \sum_{K\in\Omega^h} \int_{\partial K} \left(\bPi_K^0(\bkappa \bPi_K^0(\nabla u_h)) \cdot \bn_K\right) v. \label{third_part_residual}
\end{align}
The proof finishes by applying the boundary conditions \eqref{eq:SSP}-\eqref{eq:CP} together with the discrete jump operator in \eqref{second_part_residual}- \eqref{third_part_residual}, and subsequently substituting the result into \eqref{first_part_residual}.
\end{proof}

In what follows, we will construct the reliability estimate for $\eta$.

\begin{theorem}\label{teo14ApVEP}
    If  $u$ and $u_h$ solve the spectral problems \eqref{WMPr} and \eqref{WDPr}, respectively, then the  following result holds true
\begin{align}
    |u-u_h|_{2,\Omega}\lesssim \eta +  \frac{\lambda+\lambda_h}{2}|u-u_h|_{1,\Omega} +||\bkappa \nabla u_h-\bPi_K^0(\bkappa \nabla u_h)||_{0,\Omega} +||\nabla u_h -\bPi_K^0(\nabla u_h) ||_{0,\Omega}.\nonumber
\end{align}
\end{theorem}
\begin{proof}
Let us define $e := u - u_h \in V$ and $e_I \in V_h$. Then, for any $s\in\{2,3\}$ and $t\in\{0,1,\ldots,s\}$, the following estimate holds
\begin{equation}\label{errorInterpolated}
|e - e_I|_{t,\Omega} \lesssim h^{\,s-t}\, |e|_{s,\Omega}.
\end{equation}
Testing the error equation in Lemma~\ref{Lem13ApVEP} by $v \equiv e - e_I \in V$ yields
\begin{align}
|e|_{2,\Omega}^2
& = a(e,e-e_I)+a(u,e_I)-a(u_h,e_I) +  a^h(u_h,e_I)- a^h(u_h,e_I)\nonumber\\
   & = \left[\lambda b(u,e) - \lambda_h b(u_h,e)\right] + \left[ \lambda_h \big\{b(u_h,e_I)-b^h(u_h,e_I) \big\} \right] + \left[ a^h(u_h,e_I)-a(u_h,e_I) \right] \nonumber\\
   &\quad + \sum_{K\in \Omega^h} a_K(\Pi_K^{\nabla^2} u_h-u_h,e-e_I) \nonumber\\
   &\quad + \lambda_h \sum_{K\in \Omega^h} \int_K (\bkappa(\nabla u_h - \bPi_K^0 (\nabla u_h)))\cdot \nabla (e-e_I) \nonumber\\
   &\quad - \sum_{K\in \Omega^h} \int_{K}\lambda_h \vdiv\left( \bPi_K^0 (\bkappa \bPi_K^0(\nabla u_h))\right) (e-e_I) \nonumber\\ 
   &\quad - \sum_{f\in \cF^h_\Omega} \int_f \jump{(\nabla^2\Pi_K^{\nabla^2} u_h )\bn_K^f}\cdot \nabla (e-e_I)  \nonumber\\
   &\quad + \sum_{f\in \mathcal{F}^h_\Omega} \int_f \jump{\lambda_h\bPi_K^0(\bkappa \bPi_K^0(\nabla u_h ))\cdot \bn_K^{f}}(e-e_I)\nonumber \\
   &\quad + \sum_{K\in \Omega^h}\int_K \lambda_h \left(\bkappa \bPi_K^0(\nabla u_h)-\bPi_K^0(\bkappa \bPi_K^0(\nabla u_h))\right) \cdot \nabla (e-e_I)\nonumber\\
   & =:\sum\limits_{j=1}^9 E_j, \label{auxboundrelab}
\end{align}
where the error terms are defined by
\begin{align*}
      E_1&:= \lambda b(u,e) - \lambda_h b(u_h,e),\\  
      E_2&:=  \lambda_h \big\{b(u_h,e_I)-b^h(u_h,e_I) \big\},\\
      E_3&:=    a^h(u_h,e_I)-a(u_h,e_I), \nonumber \\
      E_4&:=\sum_{K\in \Omega^h} a_K(\Pi_K^{\nabla^2} u_h-u_h,e-e_I),\\ E_5&:= \lambda_h \sum_{K\in \Omega^h} \int_K (\bkappa(\nabla u_h - \bPi_K^0 (\nabla u_h)))\cdot \nabla (e-e_I), \nonumber\\
      E_6&:= - \sum_{K\in \Omega^h} \int_{K}\lambda_h \vdiv\left( \bPi_K^0 (\bkappa \bPi_K^0(\nabla u_h))\right) (e-e_I),\\
      E_7&:=- \sum_{f\in \cF^h_\Omega} \int_f \jump{(\nabla^2\Pi_K^{\nabla^2} u_h)\bn_K^f}\cdot \nabla (e-e_I),\nonumber\\
      E_8&:= \sum_{f\in \mathcal{F}^h_\Omega} \int_f
        \jump{\lambda_h\bPi_K^0(\bkappa \bPi_K^0(\nabla u_h ))\cdot \bn_K^{f}}(e-e_I),\nonumber\\
      E_9&:=\sum_{K\in \Omega^h}\int_K \lambda_h \left(\bkappa \bPi_K^0(\nabla u_h)-\bPi_K^0(\bkappa \bPi_K^0(\nabla u_h))\right) \cdot \nabla (e-e_I).\nonumber
\end{align*}
Notice that the terms $E_3$, $E_4$, and $E_7$ were estimated in \cite[Theorem~14]{dassi2025posteriorierrorestimatesc1}. 

We now proceed to bound the remaining terms $E_j$, let $(\lambda,u)\neq (0,\mathbf{0})$ be an eigenpair of $T$, with a simple eigenvalue $\lambda$. For each mesh $\Omega^h$, let $(\lambda_h,u_h)$ be the corresponding discrete solution of the spectral problem, such that $|u_h|_{1,\Omega}=1$ and satisfying
\[
\lim_{h\to 0} |\lambda-\lambda_h| = 0
\quad\text{and}\quad
\lim_{h\to 0} |u-u_h|_{2,\Omega} = 0.
\]
Following the arguments presented in \cite{dassi2025posteriorierrorestimatesc1} and using the fact that $b(u,u)\lesssim 1$ in $E_1$, we obtain
\begin{align}
    E_1\lesssim \frac{\lambda+\lambda_h}{2}b(e,e)\lesssim \frac{\lambda+\lambda_h}{2} ||\bPi_K^0 \nabla (u-u_h) ||_{0,\Omega}^2\lesssim  \frac{\lambda+\lambda_h}{2} |u-u_h|_{1,\Omega}|u-u_h|_{2,\Omega}\label{cotaE1}.
\end{align}
Next, for the term $E_2$, we apply the Lemma~\ref{LemAuxiliar} to deduce
\begin{align}
    E_2 & = \lambda_h \sum\limits_{K\in \Omega^h} \int_K  \Big(\bkappa \nabla u_h-\bPi_K^0(\bkappa \nabla u_h) \Big)\cdot\Big( \nabla e_I - \bPi_K^0 \nabla e_I \Big)    \nonumber\\
    &\lesssim \sum\limits_{K\in \Omega^h} ||\bkappa \nabla u_h-\bPi_K^0(\bkappa \nabla u_h)||_{0,K} ||  \nabla e_I - \bPi_K^0 \nabla e_I||_{0,K}\nonumber\\
    &\lesssim   ||\bkappa \nabla u_h-\bPi_K^0(\bkappa \nabla u_h)||_{0,\Omega} |e-e_I|_{1,\Omega} \nonumber \\
    &\lesssim  ||\bkappa \nabla u_h-\bPi_K^0(\bkappa \nabla u_h)||_{0,\Omega} |e-e_I|_{2,\Omega} \nonumber\\
    &\lesssim 
    ||\bkappa \nabla u_h-\bPi_K^0(\bkappa \nabla u_h)||_{0,\Omega} |u-u_h|_{2,\Omega}.
    \label{cotaE2}
\end{align}

In turn, for the terms $E_5, E_6$ y $E_9$ we apply the Cauchy-Schwarz inequality, estimate~\eqref{errorInterpolated} and definitions \eqref{vol_est}, \eqref{jump_est} and  \eqref{stab_est}  to obtain 
\begin{align}
   E_5+E_6+E_9
&\lesssim \sum\limits_{K\in \Omega^h}\Big\{ ||\nabla u_h -\bPi_K^0(\nabla u_h) ||_{0,K} + h^2\norm{\lambda_h \vdiv\left(\bPi_K^0(\bkappa \bPi_K^0(\nabla u_h))\right)}_{0,K}\nonumber\\
   &\quad + h_K \norm{\lambda_h \left(\bkappa \bPi_K^0(\nabla u_h)-\bPi_K^0(\bkappa \bPi_K^0(\nabla u_h))\right)}_{0,K} \Big\}|e|_{2,K}\nonumber\\
   &\lesssim \Big\{ ||\nabla u_h -\bPi_K^0(\nabla u_h) ||_{0,\Omega}  + \Xi + \Lambda  \Big\}|u-u_h|_{2,\Omega}.\label{cotasE5E6E9}
\end{align}

Lastly, for term $E_8$ we apply the standard trace inequality in the Sobolev space $H^1$ (see for instance \cite{G}) and Cauchy-Schwarz inequality to obtain 
\begin{align}
  E_8\lesssim  \cJ |u-u_h|_{2,\Omega}. \label{cotaE8}
\end{align}

The result follows by inserting the estimates \eqref{cotaE1}, \eqref{cotaE2}, \eqref{cotasE5E6E9}, \eqref{cotaE8}, the inequalities~\cite[(39), (40) and (43)]{dassi2025posteriorierrorestimatesc1} into \eqref{auxboundrelab}, and multiplying by $|u-u_h|_{2,\Omega}^{-1}$, i.e.,
\begin{align*}
    |u-u_h|_{2,\Omega}\lesssim \eta +  \frac{\lambda+\lambda_h}{2}|u-u_h|_{1,\Omega} +||\bkappa \nabla u_h-\bPi_K^0(\bkappa \nabla u_h)||_{0,\Omega} +||\nabla u_h -\bPi_K^0(\nabla u_h) ||_{0,\Omega}.  
\end{align*}
\end{proof}

\begin{corollary}\label{cor15ApVEP}
    If  $u$ and $u_h$ solve the spectral problems \eqref{WMPr} and \eqref{WDPr}, respectively. Then, the inequality holds
     \begin{align}
        |u-\Pi_K^{\nabla^2}u_h|_{2,h}&\lesssim \eta + \frac{\lambda + \lambda_h}{2}|u-u_h|_{1,\Omega} +||\bkappa \nabla u_h - \bPi_K^0(\bkappa \nabla u)_{0,\Omega} + ||\nabla u_h - \bPi_K^0(\nabla u_h)||_{0,\Omega}.\nonumber
    \end{align}   
\end{corollary}
\begin{proof}
The proof follows by applying the triangle inequality, Theorem~\ref{teo14ApVEP}  and \eqref{stab_est}. Indded,  
\begin{align}
    |u-\Pi_K^{\nabla^2}u_h|_{2,h}&\leq |u-u_h|_{2,\Omega} + |u_h-\Pi_K^{\nabla^2}u_h|_{2,h}\nonumber\\
    &\lesssim  \eta +  \frac{\lambda+\lambda_h}{2}|u-u_h|_{1,\Omega} +||\bkappa \nabla u_h-\bPi_K^0(\bkappa \nabla u_h)||_{0,\Omega} \nonumber\\
    & \quad +||\nabla u_h -\bPi_K^0(\nabla u_h) ||_{0,\Omega} + \Big\{\sum\limits_{K\in \Omega^h} S_K^2\Big\}^{1/2}\nonumber\\
    &\lesssim  \eta +  \frac{\lambda+\lambda_h}{2}|u-u_h|_{1,\Omega} +||\bkappa \nabla u_h-\bPi_K^0(\bkappa \nabla u_h)||_{0,\Omega} +||\nabla u_h -\bPi_K^0(\nabla u_h) ||_{0,\Omega}.\nonumber 
\end{align}
\end{proof}

\begin{corollary}\label{cor16ApVEP}
      If  $u$ and $u_h$ solve the spectral problems \eqref{WMPr} and \eqref{WDPr}, respectively. Then, the following estimate holds   
         \begin{align}
        |\lambda-\lambda_h|&\lesssim \Big\{\eta + \frac{\lambda + \lambda_h}{2}|u-u_h|_{1,\Omega} +||\bkappa \nabla u_h - \bPi_K^0(\bkappa \nabla u)_{0,\Omega} + ||\nabla u_h - \bPi_K^0(\nabla u_h)||_{0,\Omega}\Big\}^2\nonumber\\
        &\quad +||\bkappa \nabla u_h-\bPi_K^0(\bkappa \nabla u_h)||_{0,\Omega} ||\nabla u_h -\bPi_K^0(\nabla u_h) ||_{0,\Omega}.\nonumber
    \end{align} 
\end{corollary}
\begin{proof}
    The result follows by adding the term $|u-u_h|_{2,\Omega}^2+||\bkappa \nabla u_h-\bPi_K^0(\bkappa \nabla u_h)||_{0,\Omega} ||\nabla u_h -\bPi_K^0(\nabla u_h) ||_{0,\Omega}$ to the square of the estimate in Corollary~\ref{cor15ApVEP}, and subsequently applying Theorem~\ref{th:convergenceRates} together with Corollary~\ref{cor15ApVEP}.
\end{proof}

\begin{theorem}\label{th:reliability}
     If  $u$ and $u_h$ solve the spectral problems \eqref{WMPr} and \eqref{WDPr}, respectively, then there exists $h_0>0$ such that for all $h<h_0$ the  following estimates hold true
     \begin{align}\label{cota_conf_for_u}
         |u-u_h|_{2,\Omega}+|u-\Pi_K^{\nabla^2}u_h|_{2,h} +||\bkappa \nabla u_h-\bPi_K^0(\bkappa \nabla u_h)||_{0,\Omega} ||\nabla u_h -\bPi_K^0(\nabla u_h) ||_{0,\Omega} &\lesssim \eta, \\
            \label{cota_conf_for_lambda}
        |\lambda - \lambda_h|&\lesssim \eta^2.    
     \end{align}
\end{theorem}
\begin{proof}
    From Theorem~\ref{teo14ApVEP}, the definition of $S_K$ (cf. \eqref{stab_est}) and Theorem~\ref{ApTeo12VEP} we have
    \begin{align}
       &|u-u_h|_{2,\Omega}+|u-\Pi_K^{\nabla^2}u_h|_{2,h} +||\bkappa \nabla u_h-\bPi_K^0(\bkappa \nabla u_h)||_{0,\Omega} ||\nabla u_h -\bPi_K^0(\nabla u_h) ||_{0,\Omega}\nonumber\\
       &\lesssim  \eta + \frac{\lambda + \lambda_h}{2}|u-u_h|_{1,\Omega} +||\bkappa \nabla u_h - \bPi_K^0(\bkappa \nabla u)||_{0,\Omega} + ||\nabla u_h - \bPi_K^0(\nabla u_h)||_{0,\Omega}\nonumber\\
       &\quad +||\bkappa \nabla u_h-\bPi_K^0(\bkappa \nabla u_h)||_{0,\Omega} ||\nabla u_h -\bPi_K^0(\nabla u_h) ||_{0,\Omega}\nonumber\\
       &\lesssim \eta + \frac{\lambda + \lambda_h}{2}\Bigg\{ h^s\Big\{|u-u_h|_{2,\Omega} + |u-\Pi_K^{\nabla^2}u_h|_{2,h} \Big\}  + h^{1+s}||\bkappa \nabla u - \bPi_K^0(\bkappa \nabla u)  ||_{0,\Omega}  + |u-\mathcal{P}_h^{\nabla} u|_{1,\Omega}\Bigg\}\nonumber\\
       &\quad +||\bkappa \nabla u_h - \bPi_K^0(\bkappa \nabla u)||_{0,\Omega} + ||\nabla u_h - \bPi_K^0(\nabla u_h)||_{0,\Omega} \nonumber\\
       &\quad +||\bkappa \nabla u_h-\bPi_K^0(\bkappa \nabla u_h)||_{0,\Omega} ||\nabla u_h -\bPi_K^0(\nabla u_h) ||_{0,\Omega},\nonumber
    \end{align}
where, in the last line of the above inequality, we have proceeded  as in \cite[Theorem~4.4]{MoraVelasquez2020}. Thus, there exists $h_0>0$ such that for all $h<h_0$, the estimate \eqref{cota_conf_for_u} holds true. 

Moreover, by applying Theorem~\ref{ApTeo12VEP}, the estimate   \eqref{cota_conf_for_u} and Corollary~\ref{cor16ApVEP}, there exists $h_0>0$ such that for all $0<h<h_0$ the estimate \eqref{cota_conf_for_lambda} holds.
\end{proof}

\subsection{Efficiency}\label{Efficiency}
This subsection aims to prove the efficiency (lower bound for the error) of the global volume and jump estimators $\Xi$ and $\mathcal{J}$ up to the data oscillation $\Lambda$, the stabilisation estimator $S$ and higher-order terms. We begin by recalling properties of element $\psi_k$ and facet $\psi_f$ bubble functions in the space $H^2_0(K)$ (see \cite[Section~3.7]{verfurth1996review} for construction details in two-dimensions and \cite[Remark 6]{dassi2025posteriorierrorestimatesc1} for a discussion regarding the three-dimensional case).
\begin{subequations}
\begin{align}
||q_K||_{0,K}^2&\lesssim \int_K \psi_K q_K^2, &\quad \forall q_K\in \mathbb{P}_2(K),\label{bubb1}\\
||\psi_Kq_K||_{0,K} &\leq ||q_K||_{0,K}, &\quad \forall q_K\in \mathbb{P}_2(K),\label{bubb2}\\
h_f^{-1}||q_f||_{0,f}^2&\lesssim \int_f q_f \mathbf{n}^f_{K} \cdot \nabla(\psi_f q_f), &\quad \forall q_f\in \mathbb{P}_0(f),\label{bubb1Edge}\\
||\mathbf{n}^f_{K} \cdot \nabla(\psi_f q_f)||_{0,f} &\lesssim h^{-1}_f||q_f||_{0,f}, &\quad \forall q_f\in \mathbb{P}_0(f),\label{bubb2Edge}\\
\sum_{K=K^+,K^-} h_f^{-2}||\psi_f q_f||_{0,K}&\lesssim \sum_{K=K^+,K^-} |\psi_f q_f|_{2,K}, &\quad \forall q_f\in \mathbb{P}_0(f),\label{bubb3Edge} \\
\sum_{K=K^+,K^-} |\psi_f q_f|_{2,K}&\lesssim \sum_{K=K^+,K^-} h_f^{-2}||\psi_f q_f||_{0,K}, &\quad \forall q_f\in \mathbb{P}_0(f),\label{bubb3_5Edge} \\
\sum_{K=K^+,K^-} ||\psi_f q_f||_{0,K}&\lesssim h_f^{1/2}||q_f||_{0,f}, &\quad\forall q_f\in \mathbb{P}_0(f).\label{bubb4Edge}
\end{align}
\end{subequations}
Moreover, the following  local inverse estimates were  established in \cite{zhao2023interior} 
\begin{align}
 | v|_{1,K}\lesssim h_K^{-1}||v||_{0,K} 
\quad
\mbox{and}
\quad 
 | v|_{2,K}\lesssim h_K^{-2}||v||_{0,K},
 &\quad \forall v\in  V_h^K .\label{bubb3}
\end{align}

The following result establishes an estimate for the volume estimator $\Xi$.
\begin{lemma}\label{ApVEPLem18}
    If  $u$ and $u_h$ solve the spectral problems \eqref{WMPr} and \eqref{WDPr}, respectively,  then 
\begin{align}    
\Xi_K&\lesssim  |u-u_h|_{2,K} + S_K + h_K|\lambda u-\lambda_h u_h|_{1,K} + h_K\Lambda_K + h_K ||\nabla u_h - \bPi_K^0 (\nabla u_h)||_{0,K}.\nonumber
\end{align}
\end{lemma}
\begin{proof}
Let $K\in\Omega^h$. Consider $\psi_K$ as the interior bubble function defined in $K$ and  $v_K:=\psi_K\lambda_h\vdiv\left( \bPi_K^0 (\bkappa \bPi_K^0(\nabla u_h))\right)$. Since $v_K\in H_0^2(K)$, from now on we will denote its extension by $v$ by zero in $\Omega$. Thus, from the residual error equation in Lemma~\eqref{Lem13ApVEP} we readily see that
    \begin{align*} 
       \int_{K}\lambda_h \vdiv\left( \bPi_K^0 (\bkappa \bPi_K^0(\nabla u_h))\right) v  &= \lambda b_K(u,v)-\lambda_h b_K( u_h,v) + a_K(\Pi_K^{\nabla^2} u_h-u_h,v) -a_K(u-u_h,v)\\
        &\quad+ \lambda_h  \int_K (\bkappa(\nabla u_h - \bPi_K^0 (\nabla u_h)))\cdot \nabla v  \nonumber\\
        &\quad   + \int_K \lambda_h \left(\bkappa \bPi_K^0(\nabla u_h)-\bPi_K^0(\bkappa \bPi_K^0(\nabla u_h))\right) \cdot \nabla v
    \end{align*}
Next, we apply \eqref{bubb1} in the previous identity to obtain 
    \begin{align}
    ||\lambda_h \vdiv\left( \bPi_K^0 (\bkappa \bPi_K^0(\nabla u_h))\right) ||_{0,K}^2  &\lesssim \int_K \psi_K [\lambda_h \vdiv\left( \bPi_K^0 (\bkappa \bPi_K^0(\nabla u_h))\right)]^2 \nonumber \\ 
        &   = b(\lambda u- \lambda_h u_h,v) +  a_K(\Pi_K^{\nabla^2} u_h-u_h,v)-a(u-u_h,v) \nonumber\\
        &\quad + \lambda_h  \int_K (\bkappa(\nabla u_h - \bPi_K^0 (\nabla u_h)))\cdot \nabla v   \nonumber\\
        &\quad + \int_K \lambda_h \left(\bkappa \bPi_K^0(\nabla u_h)-\bPi_K^0(\bkappa \bPi_K^0(\nabla u_h))\right) \cdot \nabla v\nonumber\\
       & \lesssim |\lambda u-\lambda_h u_h|_{1,K}|v|_{1,K} + |\Pi_K^{\nabla^2} u_h-u_h|_{2,K} |v|_{2,K} \nonumber\\ 
       &\quad + |u-u_h|_{2,K}|v|_{2,K} + ||\nabla u_h - \bPi_K^0 (\nabla u_h)||_{0,K}|v|_{1,K}\nonumber\\
       &\quad + ||\lambda_h \left(\bkappa \bPi_K^0(\nabla u_h)-\bPi_K^0(\bkappa \bPi_K^0(\nabla u_h))\right) ||_{0,K}|v|_{1,K}.\label{cota_VOl_Eff}
    \end{align}
Since $v\in H_0^2(K)$, $0\leq \psi_K\leq 1$ and  $\lambda_h \vdiv\left( \bPi_K^0 (\bkappa \bPi_K^0(\nabla u_h))\right) \in \mathbb{P}_0(K)\in \rV_h$, the estimates \eqref{bubb2} and \eqref{bubb3} lead to 
\begin{subequations}
    \begin{align}
        |v|_{1,K} &\lesssim h_K^{-1}||\lambda_h \vdiv\left( \bPi_K^0 (\bkappa \bPi_K^0(\nabla u_h))\right) ||_{0,K}, \label{aux2-3_effa}\\ 
        |v|_{2,K} &\lesssim  h_K^{-2} ||\lambda_h \vdiv\left( \bPi_K^0 (\bkappa \bPi_K^0(\nabla u_h))\right) ||_{0,K}.\label{aux2-3_effb}
    \end{align}
\end{subequations}
Thus, by replacing the estimates \eqref{aux2-3_effa} and \eqref{aux2-3_effb} on the right-hand side of \eqref{cota_VOl_Eff}, and applying the definitions of $\Lambda_K$ and $S_K$ we have
\begin{align}    
||\lambda_h \vdiv\left( \bPi_K^0 (\bkappa \bPi_K^0(\nabla u_h))\right) ||_{0,K}
&\lesssim h_K^{-1}|\lambda u-\lambda_h u_h|_{1,K} + h_K^{-2}|\Pi_K^{\nabla^2} u_h-u_h|_{2,K} \nonumber\\
&\quad + h_K^{-2}|u-u_h|_{2,K} + h_K^{-1} ||\nabla u_h - \bPi_K^0 (\nabla u_h)||_{0,K} \nonumber\\
&\quad + h_K^{-1}||\lambda_h \left(\bkappa \bPi_K^0(\nabla u_h)-\bPi_K^0(\bkappa \bPi_K^0(\nabla u_h))\right) ||_{0,K} \nonumber\\
&\lesssim h_K^{-1}|\lambda u-\lambda_h u_h|_{1,K} + h_K^{-2} S_K + h_K^{-1} ||\nabla u_h - \bPi_K^0 (\nabla u_h)||_{0,K} \nonumber\\
&\quad + h_K^{-2}|u-u_h|_{2,K}  + h_K^{-1}\Lambda_K. \nonumber
\end{align}
Therefore, by using the definition of the local volume estimator $\Xi_K$ and multiplying by $h_K^{2}$ in the above inequality, we obtain the result.
%
\end{proof}

Now, the following result can be obtained following the arguments as those applied in \cite[Lemma~19]{dassi2025posteriorierrorestimatesc1}.
\begin{lemma}\label{ApVEPLem19}
    If  $u$ and $u_h$ solve the spectral problems \eqref{WMPr} and \eqref{WDPr}, respectively,  then
	\begin{align*}
		S_K\lesssim |u-u_h|_{2,K}+|u-\Pi_K^{\nabla^2}  u_h|_{2,K} + |u-u_h|_{1,K} + |u-\Pi_K^{\nabla}u_h|_{1,K}.
	\end{align*}
\end{lemma}

The following result establishes an upper bound for the local jump estimator $\cJ_{f}$ (cf. \eqref{jump_est}).
\begin{lemma}\label{ApVEPLem20}
	If  $u$ and $u_h$ solve the spectral problems \eqref{WMPr} and \eqref{WDPr}, respectively,  then 
	\begin{align*}
		\cJ_{f}&\lesssim \sum\limits_{K\in \Omega^h_f} \Big\{ |u-u_h|_{2,K} + S_K + h_K |u-u_h|_{2,K} + (h_K + h_K^2)|\lambda u -\lambda_h u_h |_{1,K} \nonumber\\
    &\quad + (h_K + h_K^2)||\nabla u_h - \bPi_K^0 (\nabla u_h)||_{0,K} + h_K\Xi_K + (h_K+h_K^2)\Lambda_K + h_K S_K \bigg\} \nonumber\\
    &\quad + h_f^{5/2}||[\![(\nabla^2\Pi_K^{\nabla^2} u_h )\boldsymbol{n}_K^{f}]\!]||_{0,f}.
	\end{align*}
\end{lemma}
\begin{proof}
Given any facet $f\in \mathcal{F}^h$. Let us define $\psi_l$ as the corresponding facet bubble function, and define the test functions $v_{1,f}:=\psi_l\jump{((\nabla^2\PiK u_h )\boldsymbol{n}_K^{l}) \cdot \boldsymbol{n}_K^{l}},v_{2,f}:=\psi_l\jump{\lambda_h\bPi_K^0(\bkappa \bPi_K^0(\nabla u_h ))\cdot \bn_K^f}\in H_0^2(T^+\cup T^-)\subset V$. Since $v_{i,f}$ ($i=1,2$) is a polynomial function, $v_{i,f}$ can be extended to $K^+\cup K^-$, where $K^+,K^-\in\Omega_h$ have $f$ as a common facet. Without losing generality, this extension is denoted by $v_i$. 

First, note that the estimates \eqref{bubb3} and \eqref{bubb4Edge}, together with \eqref{A3} lead to 
\begin{subequations}
    \begin{align}
        |v_i|_{1,K} &\lesssim h_f^{-1/2}||q_i||_{0,K}, \label{aux2-3_effa_edge}\\ 
        |v_i|_{2,K} &\lesssim  h_f^{-3/2} ||q_i||_{0,K}.\label{aux2-3_effb_edge}
    \end{align}
\end{subequations}
where $i=1,2$, $q_1=\jump{((\nabla^2\PiK u_h )\boldsymbol{n}_K^{l}) \cdot \boldsymbol{n}_K^{l}}$, and $q_2=\jump{\lambda_h\bPi_K^0(\bkappa \bPi_K^0(\nabla u_h ))\cdot \bn_K^f}.$

Next, the identities in \cite[(22)-(24)]{dassi2025posteriorierrorestimatesc1}, together with Lemma~\ref{Lem13ApVEP} and \eqref{bubb1Edge} imply that
\begin{align*}
h^{-1}_f||[\![(\nabla^2\Pi_K^{\nabla^2} u_h )\boldsymbol{n}_K^{f}]\!]||_{0,f}^2 &\lesssim 
\int_f [\![(\nabla^2\Pi_K^{\nabla^2} u_h )\boldsymbol{n}_K^{f}]\!]\cdot \nabla v_1\nonumber\\
&= \sum_{K\in \Omega^h_f} \bigg\{ a_K(u_h-u,v_1) + \lambda b_K(u,v_1)-\lambda_h b_K(u_h,v_1) \nonumber\\
&\quad + a_K(\Pi_K^{\nabla^2} u_h-u_h,v_1) \\
&\quad + \lambda_h \int_K (\bkappa(\nabla u_h - \bPi_K^0 (\nabla u_h)))\cdot \nabla v_1 \nonumber\\
&\quad - \int_{K}\lambda_h \vdiv\left( \bPi_K^0 (\bkappa \bPi_K^0(\nabla u_h))\right) v_1 \\
&\quad + \int_K \lambda_h \left(\bkappa \bPi_K^0(\nabla u_h)-\bPi_K^0(\bkappa \bPi_K^0(\nabla u_h))\right) \cdot \nabla v_1\bigg\} \nonumber \\
&\quad + \int_f \jump{\lambda_h\bPi_K^0(\bkappa \bPi_K^0(\nabla u_h ))\cdot \bn_K^{f}}v_1 \\
&\lesssim  \sum\limits_{K\in \Omega^h_f}\Bigg\{ |u-u_h|_{2,K}|v_1|_{2,K}+ |\lambda u -\lambda_h u_h |_{1,K}|v_1|_{1,K} \\
&\quad + |u_h-\Pi_K^{\nabla^2} u_h|_{2,K}|v_1|_{2,K} + ||\nabla u_h - \bPi_K^0 (\nabla u_h)||_{0,K}|v_1|_{1,K}\nonumber \\
&\quad +|| \lambda_h \vdiv\left( \bPi_K^0 (\bkappa \bPi_K^0(\nabla u_h))\right) ||_{0,K}|| v_1||_{0,K} \nonumber\\
&\quad + ||\lambda_h \left(\bkappa \bPi_K^0(\nabla u_h)-\bPi_K^0(\bkappa \bPi_K^0(\nabla u_h))\right) ||_{0,K}|v|_{1,K} \bigg\} \nonumber\\
&\quad + ||\jump{\lambda_h\bPi_K^0(\bkappa \bPi_K^0(\nabla u_h ))\cdot \bn_K^{f}}||_{0,f} ||v_1||_{0,f}.
\end{align*}
Inserting \eqref{aux2-3_effa_edge} and \eqref{aux2-3_effa_edge} in the previous inequality and multiplying by $h^{3/2}$, lead to
\begin{align}
    h^{1/2}_f||[\![(\nabla^2\Pi_K^{\nabla^2} u_h )\boldsymbol{n}_K^{f}]\!]||_{0,f} &\lesssim \sum\limits_{K\in \Omega^h_f}\Bigg\{ |u-u_h|_{2,K} + h_f|\lambda u -\lambda_h u_h |_{1,K} \nonumber\\
    &\quad + h_f||\nabla u_h - \bPi_K^0 (\nabla u_h)||_{0,K} + \Xi_K + h_f\Lambda_K + S_K \bigg\} \nonumber\\
    &\quad + h_f^{3/2}||\jump{\lambda_h\bPi_K^0(\bkappa \bPi_K^0(\nabla u_h ))\cdot \bn_K^{f}}||_{0,f}. \label{first_part_jump_eff}
\end{align}

Similarly, we use Lemma~\ref{Lem13ApVEP} and \eqref{bubb1Edge} to arrive at
\begin{align*}
h^{-1}_f|| \jump{\lambda_h\bPi_K^0(\bkappa \bPi_K^0(\nabla u_h ))\cdot \bn_K^f}||_{0,f}^2 &\lesssim 
\int_f \jump{\lambda_h\bPi_K^0(\bkappa \bPi_K^0(\nabla u_h ))\cdot \bn_K^f} v_2\nonumber\\
&= \sum_{K\in \Omega^h_f} \bigg\{ a_K(u_h-u,v_2) + \lambda b_K(u,v_2)-\lambda_h b_K(u_h,v_2) \nonumber\\
&\quad + a_K(\Pi_K^{\nabla^2} u_h-u_h,v_2) \\
&\quad + \lambda_h \int_K (\bkappa(\nabla u_h - \bPi_K^0 (\nabla u_h)))\cdot \nabla v_2 \nonumber\\
&\quad - \int_{K}\lambda_h \vdiv\left( \bPi_K^0 (\bkappa \bPi_K^0(\nabla u_h))\right) v_2 \\
&\quad + \int_K \lambda_h \left(\bkappa \bPi_K^0(\nabla u_h)-\bPi_K^0(\bkappa \bPi_K^0(\nabla u_h))\right) \cdot \nabla v_2\bigg\} \nonumber \\
&\quad + \int_f \jump{(\nabla^2\Pi_K^{\nabla^2} u_h )\boldsymbol{n}_K^{f}}\cdot \nabla v_2 \\
&\lesssim  \sum\limits_{K\in \Omega^h_f}\Bigg\{ |u-u_h|_{2,K}|v_2|_{2,K}+ |\lambda u -\lambda_h u_h |_{1,K}|v_2|_{1,K} \\
&\quad + |u_h-\Pi_K^{\nabla^2} u_h|_{2,K}|v_2|_{2,K} + ||\nabla u_h - \bPi_K^0 (\nabla u_h)||_{0,K}|v_2|_{1,K}\nonumber \\
&\quad +|| \lambda_h \vdiv\left( \bPi_K^0 (\bkappa \bPi_K^0(\nabla u_h))\right) ||_{0,K}|| v_2||_{0,K} \nonumber\\
&\quad + ||\lambda_h \left(\bkappa \bPi_K^0(\nabla u_h)-\bPi_K^0(\bkappa \bPi_K^0(\nabla u_h))\right) ||_{0,K}|v|_{1,K} \nonumber\\
&\quad + ||\jump{\lambda_h\bPi_K^0(\bkappa \bPi_K^0(\nabla u_h ))\cdot \bn_K^{f}}||_{0,f} ||v_2||_{0,f}.
\end{align*}
Thus,
\begin{align}
    h^{3/2}_f||\jump{\lambda_h\bPi_K^0(\bkappa \bPi_K^0(\nabla u_h ))\cdot \bn_K^f}||_{0,f} &\lesssim \sum\limits_{K\in \Omega^h_f}\Bigg\{ h_f |u-u_h|_{2,K} + h_f^2|\lambda u -\lambda_h u_h |_{1,K} \nonumber\\
    &\quad + h_f^2||\nabla u_h - \bPi_K^0 (\nabla u_h)||_{0,K} + h_f\Xi_K + h_f^2\Lambda_K + h_f S_K \bigg\} \nonumber\\
    &\quad + h_f^{5/2}||[\![(\nabla^2\Pi_K^{\nabla^2} u_h )\boldsymbol{n}_K^{f}]\!]||_{0,f}. \label{second_part_jump_eff}
\end{align}

The result follows from the addition of \eqref{first_part_jump_eff} and \eqref{second_part_jump_eff}, together with Lemma~\ref{ApVEPLem18} and the inequality $h_f\leq h_K$.
\end{proof}

Finally, we introduce the efficiency result of the global volume and jump estimators (up to data oscillation, stabilisation, and higher-order terms) as a direct consequence of Lemmas~\ref{ApVEPLem18}-\ref{ApVEPLem20} and summation for all $K\in\Omega^h$.
\begin{theorem}\label{th:efficiency}
	If  $u$ and $u_h$ solve the spectral problems \eqref{WMPr} and \eqref{WDPr}, respectively. Then, 
	\begin{align*}
		\Xi + \cJ \lesssim |u-u_h|_{2,\Omega} + S + \mathrm{HOTs},
	\end{align*}
    where the higher-order terms are given by 
    \begin{align}
    \mathrm{HOTs} &:= (h+h^2)|\lambda u-\lambda_h u_h|_{1,\Omega} +  (h+h^2)||\nabla u_h - \bPi_K^0 (\nabla u_h)||_{0,\Omega} + h\Xi + (h+h^2)\Lambda + hS \nonumber\\
    &\quad  + \sum_{f\in\mathcal{F}^h_\Omega}h_f^{5/2}||[\![(\nabla^2\Pi_K^{\nabla^2} u_h )\boldsymbol{n}_K^{f}]\!]||_{0,f}. \nonumber
    \end{align}
\end{theorem}

\section{Numerical results}\label{sec:numerical_examples}
This section presents several numerical experiments showing the performance of the estimator introduced in Section~\ref{SEC:EST_A_POST}. We test the behaviour of the estimator under uniform and adaptive refinement for a variety of polytopal meshes.
Finally, we study the applicability of the adaptive routine with an application-oriented problem.

The method is implemented in the library \texttt{vem++} formally introduced in  \cite{dassi2025vem++} and the generalised eigenvalue problem arising from such discretisations is solved inside \texttt{vem++} with the library \texttt{SLEPC} (cf. \cite{slepc}). We follow the standard strategy
$$\textnormal{SOLVE} \rightarrow \textnormal{ESTIMATE} \rightarrow \textnormal{MARK} \rightarrow \textnormal{REFINE}.$$
We adopt two alternatives for the REFINE stage: the two dimensional adaptive refinement uses the Matlab-based method from \cite{yu2021implementationpolygonalmeshrefinement}, connecting each edge mid-point to the polygon barycenter. On the other hand, the three dimensional routine employs the library \texttt{p4est} \cite{p4est} through the \texttt{GridapP4est} module of the Julia package \texttt{Gridap} \cite{gridap}. While restricted to cubical meshes, \texttt{p4est} supports the generation of hanging nodes, which are naturally handled by the VEM. It is worth noting that the refinement strategy does not rely on the specific features of \texttt{vem++}. As a consequence, the proposed implementation can be applied to more general (possibly non-convex) polytopal meshes, provided that a suitable refinement routine is chosen. In addition, we employ a D\"orfler/Bulk marking strategy as follows: mark the subset of mesh elements $\Omega^h_* \subseteq \Omega^h$ with the largest estimated errors such that for $\delta_* = \frac{1}{2}\in [0,1]$, we have
$$\delta_* \sum_{K\in \Omega^h} \eta_K^2 \leq \sum_{K\in \Omega^h_*} \eta_K^2.$$ 

The experimental order of convergence $r_{j+1}(*)$ against the total number of degrees of freedom $\textnormal{\#DoFs}$ and the effectivity index $\textnormal{eff}$ are computed as follows
\begin{align*}
    r(*)_{j+1} = -d \frac{\log\left(\frac{*_{j+1}}{*_{j}}\right)}{\log\left(\frac{\textnormal{\#DoFs}_{j+1}}{\textnormal{\#DoFs}_{j}}\right)}, \quad 
    \textnormal{eff}_{j} = \frac{(\eta^{2}_{i})_{j}}{(\textnormal{e}_{i,h})_j},
\end{align*}
where $\eta_{i}$ and $\textnormal{e}_{i,h}=|\lambda_{i,h} - \lambda_i|$ are the error and global total error estimator associated to the $i$-th eigenvalue (same notation holds for all the contributing terms of the error estimator cf. \eqref{terms_estimator}). It is worth noting that the convergence rate holds for the errors and for the global error estimator, together with each of its constituent terms. Finally, we remark that the stabilisation operator can be modified by introduce a coefficient $\alpha$ 
as follows:
$$\alpha S_K^{\nabla^2}(u_h-\Pi_K^{\nabla^2} u_h,u_h-\Pi_K^{\nabla^2} u_h).$$
The influence of the parameters $\alpha$ 
on the accuracy of the scheme was previously addressed in \cite{MRV2018} through a sensitivity analysis (see also \cite{dassi2025posteriorierrorestimatesc1}). In particular, we select the \texttt{DOFI-DOFI} type of stabilisation as in \cite{MV2} and $\alpha = 1$ for both 2D and 3D case.
\begin{figure}[h!]
    \centering
    \subfigure[Voronoi. \label{fig:polygonal}]{\includegraphics[width=0.24\textwidth,trim={2.7cm 1cm 2.4cm 0.7cm},clip]{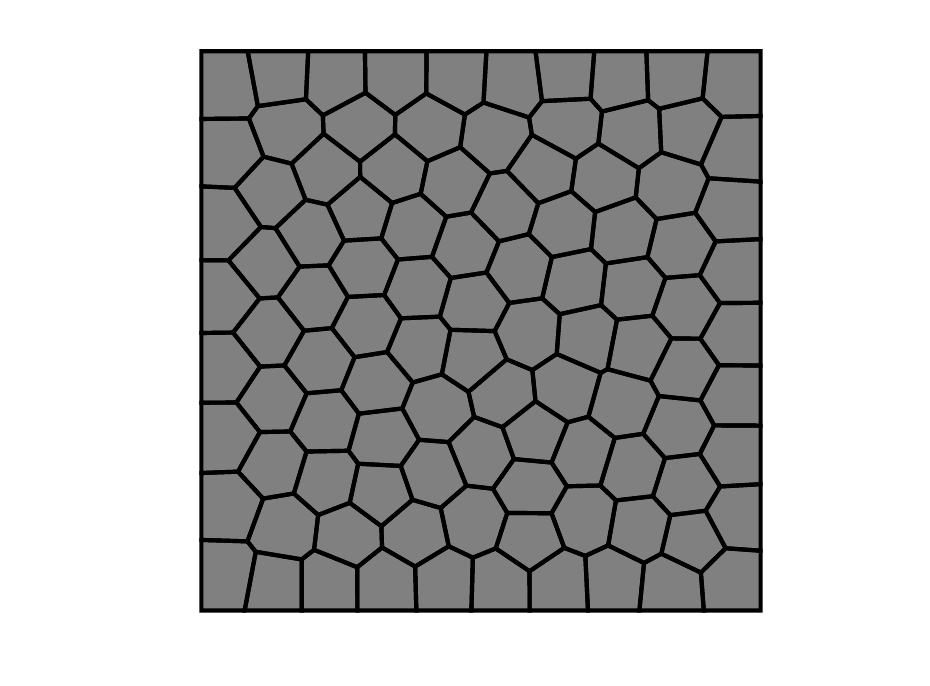}}  
    \subfigure[Perturbed Voronoi. \label{fig:distortionpolygonal}]{\includegraphics[width=0.24\textwidth,trim={2.7cm 1cm 2.4cm 0.7cm},clip]{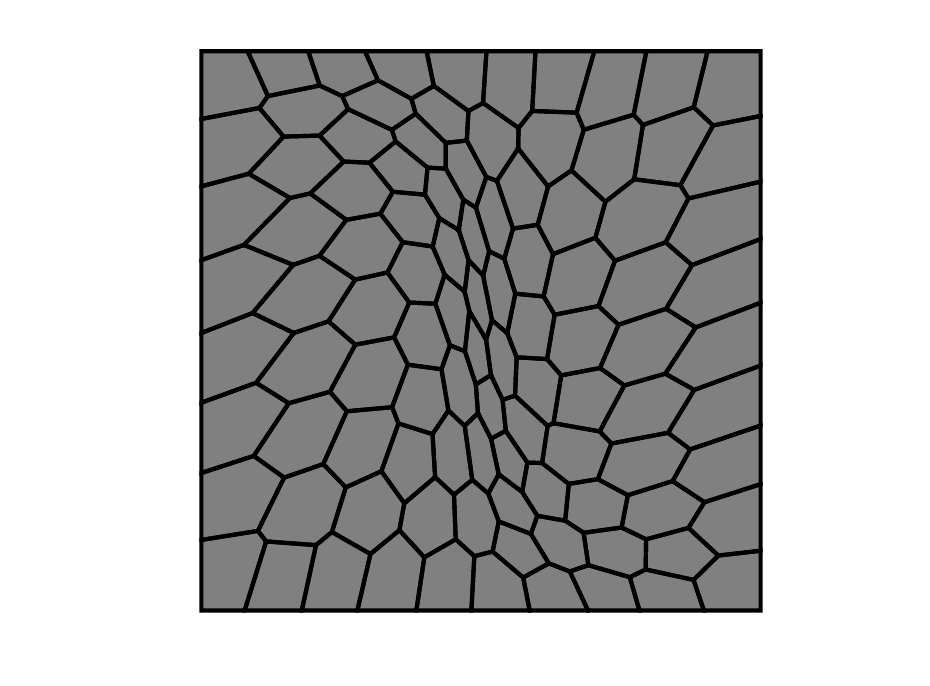}} 
    \subfigure[Square. \label{fig:square}]{\includegraphics[width=0.24\textwidth,trim={2.7cm 1cm 2.4cm 0.7cm},clip]{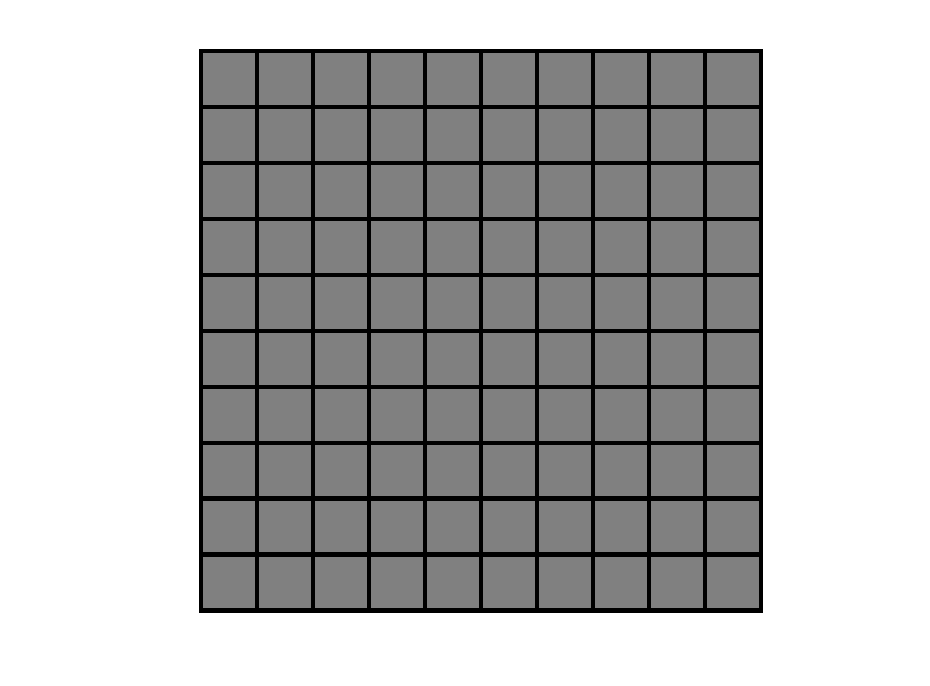}}
    \subfigure[Cube (cross-section). \label{fig:cube}]
    {\includegraphics[width=0.235\textwidth,trim={8.cm 4.5cm 8.cm 3.75cm},clip]{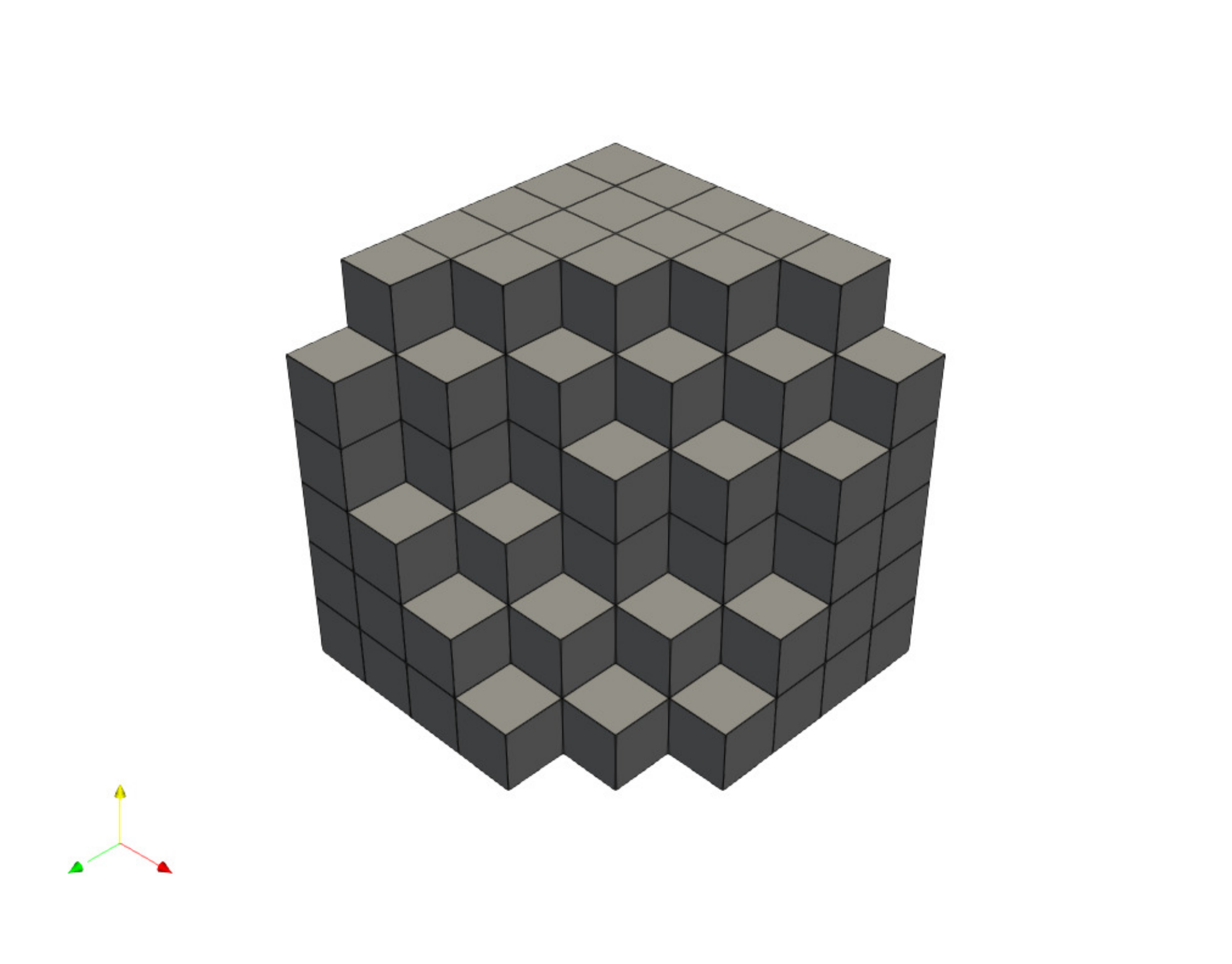}}
    \caption{Example 1. Polytopal discretisations used for the uniform refinement test.}\label{meshes_uniform}
\end{figure}

\subsection{Example 1. The estimator under uniform refinement} 
This test solves the buckling eigenvalue problem \eqref{eq:MPr} with the plane stress tensor taken as the identity ($\bkappa = \mathbb{I}_{d\times d}$), focusing on the first eigenvalue $\lambda_1$ in both two and three dimension. We consider the unit square $\Omega = (0,1)^2$ under a variety of discretisations (see Figure~\ref{fig:polygonal}-\ref{fig:square}) with boundary \textbf{CP} boundary conditions (cf. \eqref{eq:CP}) and the unit cube $\Omega = (0,1)^3$ with the discretisation shown in Figure~\ref{fig:cube} and \textbf{SSP} conditions as in \eqref{eq:SSP}, the stabilisation parameter is set to unity value.

It is well known that the lowest buckling coefficient of the proposed 2D problem is given by $\lambda_1 \approx 5.3037$ (see e.g. \cite{MR2009}) with $L=1$ and $D=\pi^2$. Whereas, the associated lowest eigenvalue for the 3D has an exact value of $\lambda_1 = 3\pi^2$ for $L,D=1$ (cf. \cite{DV_camwa2022}). Note that, the convexity of the considered domains ensures the smoothness of the associated eigenfunction $u_1$. Therefore, the optimal convergence rate is recovered for the error $\textnormal{e}_{1,h}$, and the global error estimator $\eta$ (see Table~\ref{tab:convergence2DExample1}-\ref{tab:convergence3DExample1}). Moreover, for any $K\in \Omega^h$ and constant $\bkappa$, the local data oscillation $\Lambda_K$ (cf. \eqref{data_est}) is exactly zero since $\bkappa \nabla \Pi_K^{\nabla} u_h=\Pi_K^{\nabla}(\bkappa \nabla \Pi_K^{\nabla} u_h)$, confirmed experimentally in both 2D and 3D cases.

The convergence history of the eigenvalue error $\textnormal{e}_{1,h}$ together with the global error estimator $\eta$ under (two-dimensional) uniform refinement is reported in Table~\ref{tab:convergence2DExample1}. Whereas, the three dimensional behaviour (associated with the eigenvalue error $\textnormal{e}_{1,h}$) is summarised in Table~\ref{tab:convergence3DExample1}, for this particular case, the dominant terms of the estimator ($\cJ$ and $S$) decay slower than $\Xi$, indicating optimal performance of the estimator once the mesh is sufficiently fine (results for very fine meshes are omitted due to the high computational cost). The reliability of the estimator is confirmed (see Theorem~\ref{th:reliability}) and the expected convergence rate of $O(h^2)$ is achieved (cf. Theorem~\ref{th:convergenceRates}) for both experiments. Finally, the effectivity index $\textnormal{eff}$ remains bounded for each test case as predicted by Theorem~\ref{th:efficiency}.

\begin{table}[!t]
    \setlength{\tabcolsep}{2pt}
    \begin{center}
        \resizebox{\textwidth}{!}{ 
            \begin{tabular}{| c | c | c | c | c | c | c | c | c | c | c | c | c | c | c |}
                \hline
                {$\Omega^h$} & 
                {$\textnormal{\#DoFs}$} & 
                {$\textnormal{e}_{1,h}$} & 
                {$r(\textnormal{e}_{1,h})$} & 
                {$\eta^2$} & 
                {$r(\eta^2)$} & 
                {$\Xi^2$} & 
                {$r(\Xi^2)$} & 
                {$\mathcal{J}^2$} & 
                ${r(\mathcal{J}^2)}$ & 
                {$S^2$} & 
                {$r(S^2)$} & 
                {$\Lambda^2$} & 
                {$r(\Lambda^2)$} & 
                {$\textnormal{eff}$} \\ [1pt]
                \hline 
                \hline
                \multirow{4}{*}{\rule{0pt}{19ex} \rotatebox{90}{Voronoi}} 
                & 606 & 3.11e-01 & * & 9.22e+00 & * & 7.28e-01 & * & 4.80e+00 & * & 3.69e+00 & * & 8.38e-33 & * & 2.97e+01 \\
                & 2406 & 1.23e-01 & 1.34e+00 & 2.61e+00 & 1.83e+00 & 5.28e-02 & 3.81e+00 & 1.22e+00 & 1.99e+00 & 1.34e+00 & 1.47e+00 & 2.04e-33 & * & 2.12e+01 \\
                & 5406 & 6.62e-02 & 1.53e+00 & 1.30e+00 & 1.72e+00 & 1.04e-02 & 4.02e+00 & 6.02e-01 & 1.75e+00 & 6.91e-01 & 1.63e+00 & 9.02e-34 & * & 1.97e+01 \\
                & 9597 & 3.94e-02 & 1.81e+00 & 7.59e-01 & 1.88e+00 & 3.31e-03 & 3.98e+00 & 3.49e-01 & 1.90e+00 & 4.06e-01 & 1.85e+00 & 5.10e-34 & * & 1.92e+01 \\
                & 14997 & 2.33e-02 & 2.37e+00 & 4.66e-01 & 2.19e+00 & 1.44e-03 & 3.74e+00 & 2.21e-01 & 2.05e+00 & 2.44e-01 & 2.29e+00 & 3.22e-34 & * & 2.00e+01 \\
                & 21594 & 1.63e-02 & 1.95e+00 & 3.21e-01 & 2.05e+00 & 6.80e-04 & 4.10e+00 & 1.54e-01 & 1.98e+00 & 1.66e-01 & 2.10e+00 & 2.16e-34 & * & 1.97e+01 \\
                & 29394 & 1.37e-02 & 1.11e+00 & 2.52e-01 & 1.58e+00 & 3.70e-04 & 3.95e+00 & 1.16e-01 & 1.84e+00 & 1.35e-01 & 1.33e+00 & 1.68e-34 & * & 1.83e+01 \\
                & 38364 & 1.12e-02 & 1.52e+00 & 1.95e-01 & 1.90e+00 & 2.21e-04 & 3.85e+00 & 8.80e-02 & 2.06e+00 & 1.07e-01 & 1.77e+00 & 1.26e-34 & * & 1.74e+01 \\
                & 48579 & 9.21e-03 & 1.66e+00 & 1.56e-01 & 1.89e+00 & 1.28e-04 & 4.63e+00 & 6.93e-02 & 2.03e+00 & 8.68e-02 & 1.77e+00 & 9.69e-35 & * & 1.70e+01 \\
                & 59976 & 7.40e-03 & 2.08e+00 & 1.25e-01 & 2.09e+00 & 8.75e-05 & 3.63e+00 & 5.71e-02 & 1.85e+00 & 6.82e-02 & 2.29e+00 & 8.21e-35 & * & 1.69e+01 \\
                \hline
                Avg. & * &  * & 1.71e+00 & * & 1.90e+00 & * & 3.97e+00 & * & 1.94e+00 & * & 1.83e+00 & * & * & 1.99e+01 \\
                \hline
                \multirow{4}{*}{\rotatebox{90}{Perturbed Voronoi}\rule{0pt}{25ex}}  
                & 606 & 1.63e+00 & * & 1.37e+01 & * & 6.97e-01 & * & 3.08e+00 & * & 9.88e+00 & * & 7.98e-33 & * & 8.39e+00 \\
                & 2406 & 6.66e-01 & 1.30e+00 & 6.48e+00 & 1.08e+00 & 1.77e-01 & 1.99e+00 & 1.11e+00 & 1.48e+00 & 5.19e+00 & 9.34e-01 & 2.63e-33 & * & 9.72e+00 \\
                & 5406 & 3.29e-01 & 1.74e+00 & 3.40e+00 & 1.59e+00 & 4.02e-02 & 3.66e+00 & 5.62e-01 & 1.69e+00 & 2.80e+00 & 1.53e+00 & 1.29e-33 & * & 1.03e+01 \\
                & 9597 & 2.04e-01 & 1.68e+00 & 2.15e+00 & 1.59e+00 & 1.59e-02 & 3.23e+00 & 3.50e-01 & 1.65e+00 & 1.79e+00 & 1.56e+00 & 6.88e-34 & * & 1.06e+01 \\
                & 14997 & 1.30e-01 & 2.03e+00 & 1.39e+00 & 1.96e+00 & 6.28e-03 & 4.16e+00 & 2.27e-01 & 1.94e+00 & 1.16e+00 & 1.94e+00 & 4.78e-34 & * & 1.07e+01 \\
                & 21594 & 9.27e-02 & 1.84e+00 & 1.00e+00 & 1.81e+00 & 3.38e-03 & 3.40e+00 & 1.66e-01 & 1.73e+00 & 8.31e-01 & 1.82e+00 & 3.31e-34 & * & 1.08e+01 \\
                & 29394 & 7.08e-02 & 1.75e+00 & 7.62e-01 & 1.76e+00 & 1.84e-03 & 3.96e+00 & 1.22e-01 & 1.99e+00 & 6.38e-01 & 1.71e+00 & 2.41e-34 & * & 1.08e+01 \\
                & 38364 & 5.51e-02 & 1.89e+00 & 5.89e-01 & 1.93e+00 & 1.19e-03 & 3.28e+00 & 9.47e-02 & 1.90e+00 & 4.93e-01 & 1.94e+00 & 1.93e-34 & * & 1.07e+01 \\
                & 48579 & 4.38e-02 & 1.94e+00 & 4.69e-01 & 1.93e+00 & 6.62e-04 & 4.94e+00 & 7.51e-02 & 1.96e+00 & 3.93e-01 & 1.92e+00 & 1.45e-34 & * & 1.07e+01 \\
                & 59976 & 3.64e-02 & 1.75e+00 & 3.89e-01 & 1.76e+00 & 4.37e-04 & 3.94e+00 & 6.27e-02 & 1.71e+00 & 3.26e-01 & 1.77e+00 & 1.22e-34 & * & 1.07e+01 \\
                \hline
                Avg. & * &  * & 1.77e+00 & * & 1.71e+00 & * & 3.62e+00 & * & 1.78e+00 & * & 1.68e+00 & * & * & 1.03e+01 \\
                \hline
                \multirow{4}{*}{\rule{0pt}{18ex} \rotatebox{90}{Square}} 
                & 363 & 4.56e-01 & * & 1.45e+01 & * & 1.54e+00 & * & 8.90e+00 & * & 4.04e+00 & * & 3.76e-33 & * & 3.17e+01 \\
                & 1323 & 1.34e-01 & 1.90e+00 & 3.92e+00 & 2.02e+00 & 1.31e-01 & 3.82e+00 & 2.54e+00 & 1.94e+00 & 1.25e+00 & 1.81e+00 & 1.34e-33 & * & 2.93e+01 \\
                & 2883 & 6.15e-02 & 2.00e+00 & 1.78e+00 & 2.03e+00 & 2.76e-02 & 4.00e+00 & 1.17e+00 & 2.00e+00 & 5.82e-01 & 1.96e+00 & 6.65e-34 & * & 2.89e+01 \\
                & 5043 & 3.50e-02 & 2.01e+00 & 1.01e+00 & 2.03e+00 & 8.93e-03 & 4.03e+00 & 6.65e-01 & 2.01e+00 & 3.33e-01 & 2.00e+00 & 2.89e-34 & * & 2.87e+01 \\
                & 7803 & 2.26e-02 & 2.01e+00 & 6.47e-01 & 2.03e+00 & 3.70e-03 & 4.04e+00 & 4.28e-01 & 2.01e+00 & 2.15e-01 & 2.01e+00 & 2.35e-34 & * & 2.87e+01 \\
                & 11163 & 1.57e-02 & 2.01e+00 & 4.50e-01 & 2.02e+00 & 1.79e-03 & 4.04e+00 & 2.99e-01 & 2.01e+00 & 1.50e-01 & 2.01e+00 & 1.55e-34 & * & 2.86e+01 \\
                & 15123 & 1.16e-02 & 2.01e+00 & 3.31e-01 & 2.02e+00 & 9.71e-04 & 4.04e+00 & 2.20e-01 & 2.01e+00 & 1.10e-01 & 2.01e+00 & 1.04e-34 & * & 2.85e+01 \\
                & 19683 & 8.92e-03 & 2.00e+00 & 2.54e-01 & 2.02e+00 & 5.71e-04 & 4.04e+00 & 1.69e-01 & 2.01e+00 & 8.46e-02 & 2.01e+00 & 7.36e-35 & * & 2.85e+01 \\
                & 24843 & 7.07e-03 & 2.00e+00 & 2.01e-01 & 2.02e+00 & 3.57e-04 & 4.03e+00 & 1.34e-01 & 2.01e+00 & 6.70e-02 & 2.01e+00 & 6.64e-35 & * & 2.84e+01 \\
                & 30603 & 5.74e-03 & 1.99e+00 & 1.63e-01 & 2.02e+00 & 2.34e-04 & 4.03e+00 & 1.08e-01 & 2.01e+00 & 5.43e-02 & 2.01e+00 & 6.04e-35 & * & 2.84e+01 \\
                \hline
                Avg. & * &  * & 1.99e+00 & * & 2.02e+00 & * & 4.01e+00 & * & 2.00e+00 & * & 1.98e+00 & * & * & 2.90e+01 \\
                \hline
            \end{tabular}
        }
    \end{center}
    \vspace{-0.5cm}
    \caption{Example 1. Convergence history for a variety of 2D meshes of the eigenvalue error $\textnormal{e}_{1,h}=|\lambda_{1,h} - \lambda_1|$, and the global error estimator $\eta^2$ together with  the volume residual, jump residual, stabilisation and data oscillation terms $\Xi^2$, $\mathcal{J}^2$, $S^2$, and $\Lambda^2$. The effectivity index $\textnormal{eff}$ of the estimator is also shown in the last column.}
    \label{tab:convergence2DExample1}
\end{table}

\begin{table}[!t]
    \setlength{\tabcolsep}{2pt}
    \begin{center}
        \resizebox{\textwidth}{!}{ 
            \begin{tabular}{| c | c | c | c | c | c | c | c | c | c | c | c | c | c | c |}
                \hline
                {$\Omega^h$} &
                {$\textnormal{\#DoFs}$} & 
                {$\textnormal{e}_{1,h}$} & 
                {$r(\textnormal{e}_{1,h})$} & 
                {$\eta^2$} & 
                {$r(\eta^2)$} & 
                {$\Xi^2$} & 
                {$r(\Xi^2)$} & 
                {$\mathcal{J}^2$} & 
                ${r(\mathcal{J}^2)}$ & 
                {$S^2$} & 
                {$r(S^2)$} & 
                {$\Lambda^2$} & 
                {$r(\Lambda^2)$} & 
                {$\textnormal{eff}$} \\ [1pt]
                \hline 
                \hline
                \multirow{4}{*}{\rule{0pt}{16ex} \rotatebox{90}{Cube}} 
                & 1372 & 3.82e+00 & * & 8.29e+01 & * & 7.41e+01 & * & 6.74e+00 & * & 2.03e+00 & * & 6.76e-31 & * & 2.17e+01 \\
                & 5324 & 1.92e+00 & 1.53e+00 & 1.79e+01 & 3.39e+00 & 1.49e+01 & 3.55e+00 & 1.79e+00 & 2.94e+00 & 1.21e+00 & 1.14e+00 & 1.97e-31 & * & 9.33e+00 \\
                & 13500 & 1.06e+00 & 1.90e+00 & 6.32e+00 & 3.36e+00 & 4.76e+00 & 3.68e+00 & 8.27e-01 & 2.48e+00 & 7.33e-01 & 1.62e+00 & 1.59e-31 & * & 5.93e+00 \\
                & 27436 & 6.70e-01 & 1.96e+00 & 2.87e+00 & 3.34e+00 & 1.91e+00 & 3.86e+00 & 4.78e-01 & 2.32e+00 & 4.79e-01 & 1.80e+00 & 9.04e-32 & * & 4.28e+00 \\
                & 48668 & 4.58e-01 & 1.99e+00 & 1.54e+00 & 3.24e+00 & 8.99e-01 & 3.95e+00 & 3.11e-01 & 2.24e+00 & 3.34e-01 & 1.88e+00 & 6.30e-32 & * & 3.38e+00 \\
                & 78732 & 3.32e-01 & 2.01e+00 & 9.38e-01 & 3.11e+00 & 4.74e-01 & 3.99e+00 & 2.19e-01 & 2.19e+00 & 2.45e-01 & 1.93e+00 & 4.87e-32 & * & 2.83e+00 \\
                & 119164 & 2.51e-01 & 2.02e+00 & 6.22e-01 & 2.98e+00 & 2.73e-01 & 4.01e+00 & 1.63e-01 & 2.16e+00 & 1.87e-01 & 1.96e+00 & 2.56e-32 & * & 2.48e+00 \\
                & 171500 & 1.96e-01 & 2.02e+00 & 4.40e-01 & 2.86e+00 & 1.67e-01 & 4.02e+00 & 1.26e-01 & 2.13e+00 & 1.47e-01 & 1.98e+00 & 2.46e-32 & * & 2.24e+00 \\
                \hline
                Avg. & * &  * & 1.92e+00 & * & 3.18e+00 & * & 3.87e+00 & * & 2.35e+00 & * & 1.76e+00 & * & * & 6.52e+00 \\
                \hline
            \end{tabular}
        }
    \end{center}
    \vspace{-0.5cm}
    \caption{Example 1. Convergence history for the 3D cube mesh of the eigenvalue error $\textnormal{e}_{1,h}=|\lambda_{1,h} - \lambda_1|$, and the global error estimator $\eta^2$ together with  the volume residual, jump residual, stabilisation and data oscillation terms $\Xi^2$, $\mathcal{J}^2$, $S^2$, and $\Lambda^2$. The effectivity index $\textnormal{eff}$ of the estimator is also shown in the last column.}
    \label{tab:convergence3DExample1}
\end{table}

\begin{figure}[h!]
    \centering
    \subfigure[L-shaped. \label{fig:LShape}]{\includegraphics[width=0.24\textwidth,trim={2.7cm 1cm 2.4cm 0.7cm},clip]{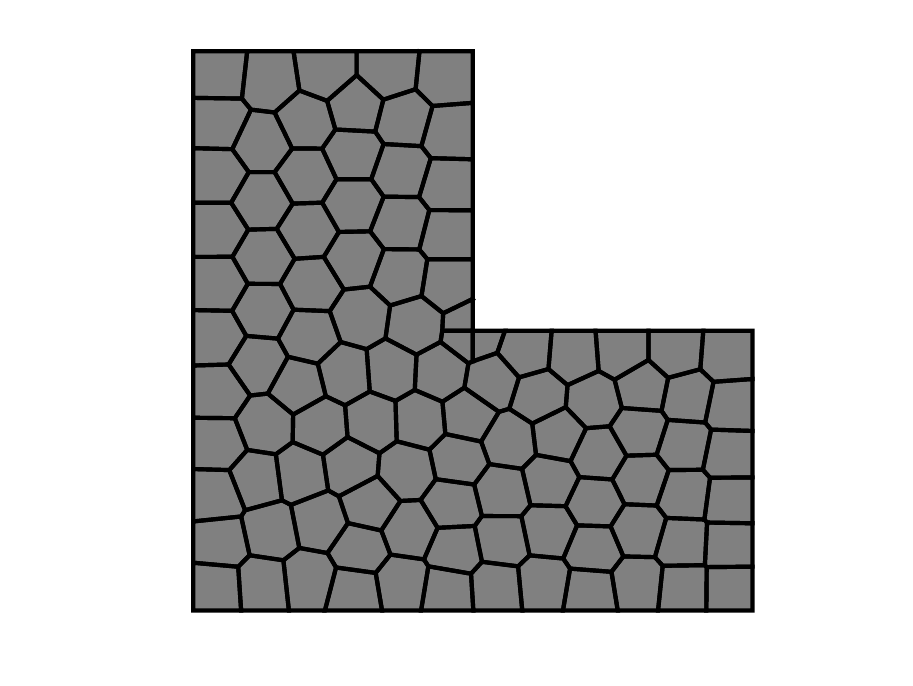}}  
    \subfigure[Fichera cube. \label{fig:Fichera}]{\includegraphics[width=0.24\textwidth,trim={5.5cm 0.05cm 4.cm 1.cm},clip]{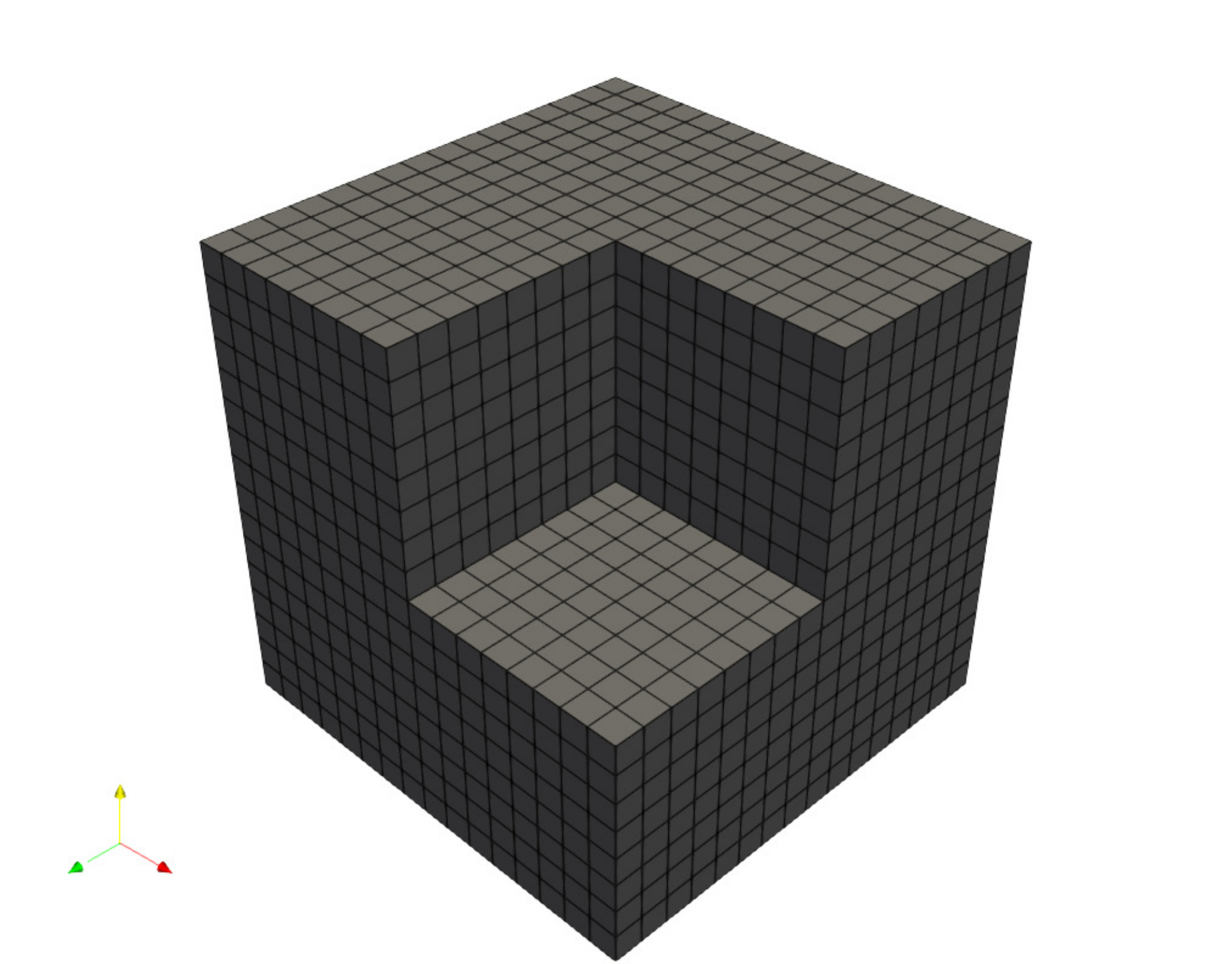}} 
    \subfigure[Perfored circle. \label{fig:CircleCircle}]{\includegraphics[width=0.24\textwidth,trim={2.7cm 1cm 2.4cm 0.7cm},clip]{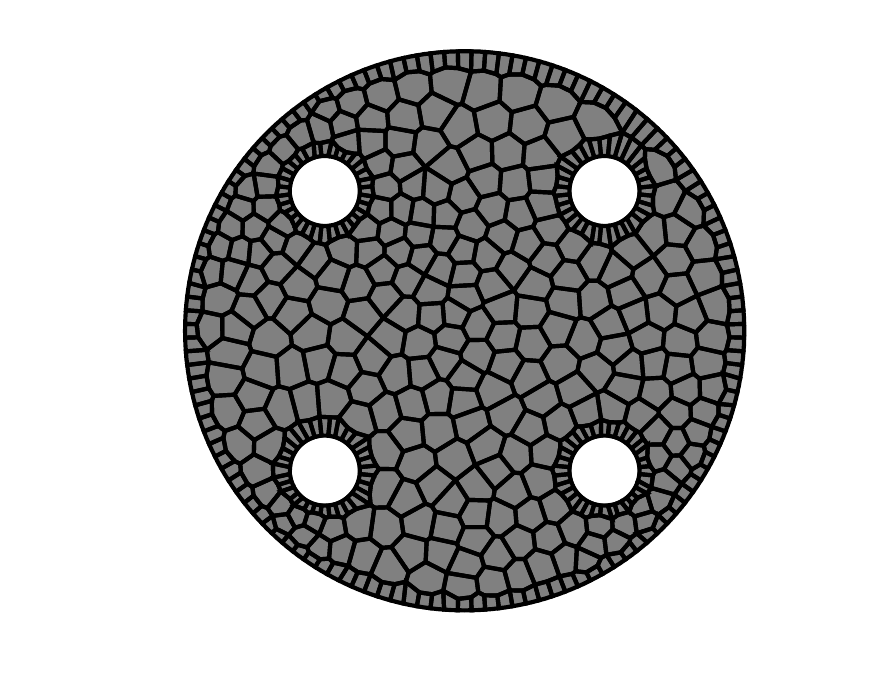}}
    \subfigure[Perfored cube (cross-section). \label{fig:CubeCube}]
    {\includegraphics[width=0.235\textwidth,trim={6.5cm 3.25cm 6.25cm 2.25cm},clip]{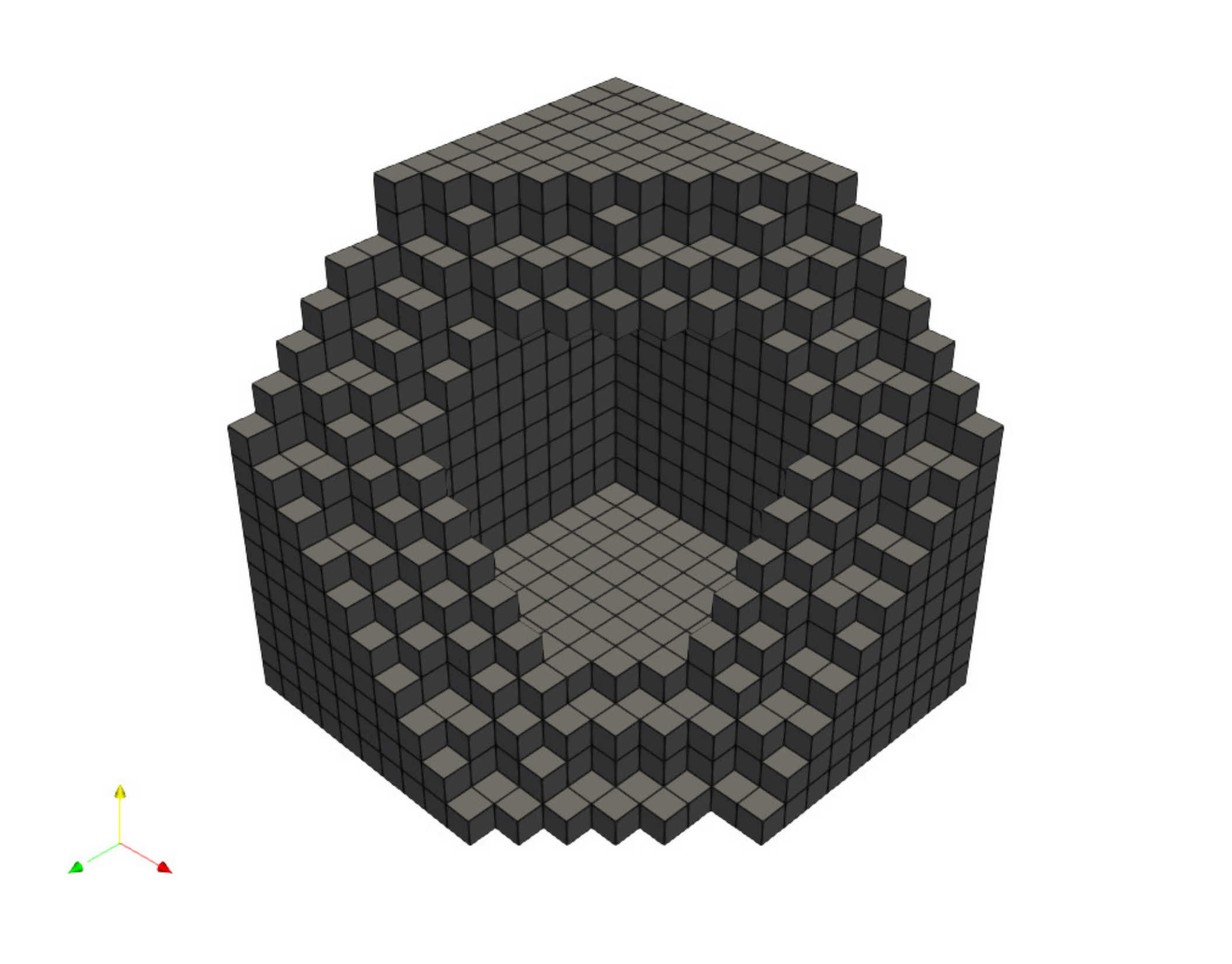}}
    \caption{Example 2. Polytopal discretisations used for the adaptive refinement test.}\label{meshes_adaptive}
\end{figure}
\subsection{Example 2. The estimator under adaptive refinement}\label{sec:example2}
In this experiment we consider \eqref{eq:MPr} loaded in the $x$-direction by a linearly distributed in-plane traction, i.e. the plane stress tensor function is given by 
$$\bkappa_{2\mathrm{D}} = \begin{pmatrix}
    1-\kappa_0 y & 0 \\
    0 & 0
\end{pmatrix} \quad \text{and} \quad  \bkappa_{3\mathrm{D}} = \begin{pmatrix}
    1-\kappa_0 y & 0 & 0 \\
    0 & 0 & 0 \\
    0 & 0 & 0
\end{pmatrix}, \quad \text{with} \quad \kappa_0\in \left\{0,\frac{2}{3},1,\frac{4}{3}\right\}.$$

We begin by examining the classical L-shaped and Fichera cube domains (shown in Figs~\ref{fig:LShape}-\ref{fig:Fichera}), given by $\Omega=(0,1)^2 \setminus (1/2,1)^2$ and $\Omega = (0,1)^3 \setminus (1/2,1)^3$. For this test, \textbf{CP} boundary conditions are imposed and the parameter $\kappa_0$ is set to 0 for the two-dimensional and $2/3$ for the three-dimensional case. Exact eigenvalues for the L-shaped domain in the isotropic case $\bkappa=\mathbb{I}_{d\times d}$ are available in the literature \cite{MoraVelasquez2020}. 

The second part of this test considers perfored domains. The two-dimensional domain consist on a circle centred in $(1/2,1/2)$ with radius $1/2$, the perforations are given by four circles with radius $1/16$ centred at $(1/4,1/4)$, $(1/4,3/4)$, $(3/4,1/4)$, and $(3/4,3/4)$ (see Fig~\ref{fig:CircleCircle}). Whereas, the three-dimensional domain is given by $(0,1)^3\setminus(1/4,3/4)^3$, a unit cube domain perfored with an inner cube (cf. Fig~\ref{fig:CubeCube}). In this test, we impose \textbf{SSP} conditions on the outer  boundary and is set free in the perforations, the respective values for $\kappa_0$ are set to $1$ and $4/3$. The first eigenvalue for a similar perfored circle domain (i.e. the perforations are located in different positions) in the isotropic case $\bkappa=\mathbb{I}_{d\times d}$ and \textbf{CP}-free boundary conditions is available in \cite{ADAK2023115763}. 

Note that the non-convexity of the domains proposed, together with the presence of perforations, naturally calls for adaptive refinement routines. These geometric features generate singularities at re-entrant corners and around internal boundaries, and adaptive refinement provides an efficient way to recover optimal accuracy while avoiding unnecessary global mesh refinement. We recall that to the best of the authors’ knowledge, no exact eigenvalues have been reported for the configurations proposed in this numerical example, which here are compute with the values $D,L=1$. 

Table~\ref{tab:convergence2D3DExample2} summarise the convergence history of the global error estimator $\eta_i$ ($i=1,2,3$) under adaptive refinement. We recover the convergence rate predicted by Theorem~\ref{th:convergenceRates} for all the tested values of $\kappa_0$ and meshes shown in Fig~\ref{meshes_adaptive}. Moreover, the eigenvalues $\lambda_{i,h}$ obtained by our scheme are displayed in Table~\ref{tab:convergence2D3DExample2} for each refinement step and the snapshot of the associated eigenfunctions for each test case are shown in Figs~\ref{fig:snapshotsLShape},\ref{fig:snapshotsPCircle},\ref{fig:snapshotsFichera}, and \ref{fig:snapshotsPCube}, respectively. 

On the other hand, we observe that the marking procedure is inherently tied to the eigenvalue we are interested in. For example, in the perforated cube section of Table~\ref{tab:convergence2D3DExample2}, the third eigenvalue computed by the scheme in the last refinement step has a value of $164.6964$, which is larger than the first eigenvalue $70.17242$. These values appear directly in the global volume estimator $\Xi$ and the global jump estimator $\mathcal{J}$. Then, when targeting the third eigenvalue, fewer elements need to be included in $\mathcal{T}_*^h$ in order to meet the criterion required by the D"orfler/Bulk marking strategy. This explains why four refinement steps for the first eigenvalue produce $\num{150684}$ total DoFs, whereas five refinement steps for the third eigenvalue yield only $\num{134967}$ total DoFs.

\begin{table}[!t]
    \setlength{\tabcolsep}{2pt}
    \begin{center}
        \resizebox{\textwidth}{!}{ 
            \begin{tabular}{| c | c | c | c | c | c | c | c | c | c | c | c | c | c | c | c | c | c |}
                \hline
                {$\kappa_0$} &
                {$\Omega^h$} & 
                {$\textnormal{\#DoFs}$} & 
                {$\lambda_{1,h}$} & 
                {$\eta_1^2$} & 
                {$r(\eta_1^2)$} &
                {$\textnormal{\#DoFs}$} & 
                {$\lambda_{2,h}$} & 
                {$\eta_2^2$} & 
                {$r(\eta_2^2)$} &
                {$\textnormal{\#DoFs}$} & 
                {$\lambda_{3,h}$} & 
                {$\eta_3^2$} & 
                {$r(\eta_3^2)$} \\ [1pt]
                \hline 
                \hline
                \multirow{8}{*}{0} & 
                \multirow{8}{*}{\rotatebox{90}{L-shaped}}
                & 21618 & 171.59229 & 7.64e+00 & * & 21618 & 257.95567 & 1.82e+01 & * & 21618 & 284.78435 & 1.69e+01 & * \\
                & & 28032 & 174.17526 & 4.65e+00 & 3.82e+00 & 28896 & 264.14787 & 1.16e+01 & 3.11e+00 & 33693 & 290.21865 & 1.03e+01 & 2.25e+00 \\
                & & 38547 & 175.50164 & 2.68e+00 & 3.46e+00 & 40839 & 267.53076 & 6.89e+00 & 3.01e+00 & 47931 & 292.98084 & 5.95e+00 & 3.09e+00 \\
                & & 54630 & 176.26211 & 1.51e+00 & 3.30e+00 & 56997 & 269.49691 & 4.02e+00 & 3.24e+00 & 71622 & 294.89883 & 3.30e+00 & 2.94e+00 \\
                & & 88581 & 176.91846 & 8.45e-01 & 2.40e+00 & 93060 & 270.94045 & 2.22e+00 & 2.43e+00 & 124317 & 296.28080 & 1.69e+00 & 2.41e+00  \\
                & & 137970 & 177.34288 & 4.70e-01 & 2.64e+00 & 146022 & 271.77504 & 1.23e+00 & 2.62e+00 & 191100 & 296.93796 & 9.46e-01 & 2.71e+00 \\
                & & 219516 & 177.51347 & 2.82e-01 & 2.20e+00 & 219957 & 272.20495 & 7.29e-01 & 2.54e+00 & 296409 & 297.25558 & 5.86e-01 & 2.18e+00 \\
                & & 366699 & 177.67225 & 1.73e-01 & 1.90e+00 & 364734 & 272.55606 & 4.55e-01 & 1.87e+00 & 514893 & 297.61744 & 3.57e-01 & 1.79e+00 \\
                \hline
                \multicolumn{2}{|c|}{Avg.} & * &  * & * & 2.70e+00 & * &  * & * & 2.69e+00 & * &  * & * & 2.48e+00 \\
                \hline
                \multirow{5}{*}{{$\frac{2}{3}$}} & 
                \multirow{5}{*}{\rotatebox{90}{Fichera cube}}
                & 17604 & 87.64595 & 4.07e+01 & * & 17604 & 124.34059 & 1.34e+02 & * & 17604 & 151.83164 & 1.81e+02 & * \\
                & & 31940 & 93.63479 & 2.97e+01 & 1.58e+00 & 32252 & 145.80585 & 1.10e+02 & 1.00e+00 & 38840 & 160.73295 & 1.39e+02 & 1.01e+00 \\
                & & 55084 & 97.62630 & 1.96e+01 & 2.30e+00 & 59324 & 154.18331 & 6.93e+01 & 2.25e+00 & 54120 & 173.03266 & 8.01e+01 & 4.96e+00 \\
                & & 111972 & 100.76571 & 1.32e+01 & 1.67e+00 & 99188 & 163.33531 & 4.85e+01 & 2.09e+00 & 83320 & 182.78053 & 5.96e+01 & 2.06e+00 \\
                & & 245936 & 103.50021 & 9.08e+00 & 1.42e+00 & 220916 & 171.40649 & 2.98e+01 & 1.82e+00 & 163480 & 192.79947 & 4.16e+01 & 1.59e+00 \\
                \hline
                \multicolumn{2}{|c|}{Avg.} & * &  * & * & 1.74e+00 & * &  * & * & 1.79e+00 & * &  * & * & 2.41e+00 \\
                \hline
                \multirow{10}{*}{1} & 
                \multirow{10}{*}{\rotatebox{90}{Perfored circle}}
                & 3147 & 100.38706 & 2.79e-01 & * & 3147 & 119.70873 & 3.63e-01 & * & 3147 & 177.87490 & 5.59e-01 & * \\
                & & 4077 & 108.57435 & 1.89e-01 & 3.00e+00 & 3945 & 132.72194 & 2.41e-01 & 3.61e+00 & 3777 & 203.42516 & 4.03e-01 & 3.58e+00 \\
                & & 5547 & 117.77372 & 1.21e-01 & 2.89e+00 & 5715 & 146.84074 & 1.59e-01 & 2.27e+00 & 4800 & 229.51096 & 2.71e-01 & 3.33e+00 \\
                & & 8484 & 122.67488 & 7.70e-02 & 2.13e+00 & 9225 & 155.03149 & 1.00e-01 & 1.92e+00 & 7296 & 252.59657 & 2.19e-01 & 1.02e+00 \\
                & & 13881 & 125.24802 & 4.65e-02 & 2.05e+00 & 14367 & 158.41632 & 6.01e-02 & 2.30e+00 & 10548 & 259.30938 & 1.20e-01 & 3.26e+00 \\
                & & 20448 & 126.80516 & 2.82e-02 & 2.58e+00 & 22188 & 160.76700 & 3.67e-02 & 2.27e+00 & 15894 & 265.42276 & 7.00e-02 & 2.62e+00 \\
                & & 31878 & 128.01394 & 1.81e-02 & 2.00e+00 & 35772 & 162.61391 & 2.34e-02 & 1.89e+00 & 24162 & 270.19999 & 4.36e-02 & 2.26e+00 \\
                & & 51735 & 128.79951 & 1.17e-02 & 1.80e+00 & 57180 & 163.95736 & 1.52e-02 & 1.85e+00 & 40176 & 273.62860 & 2.74e-02 & 1.83e+00 \\
                & & 81276 & 129.40151 & 7.38e-03 & 2.04e+00 & 92091 & 164.81363 & 9.74e-03 & 1.86e+00 & 65019 & 276.10676 & 1.70e-02 & 1.98e+00 \\
                & & 131361 & 129.78774 & 4.65e-03 & 1.92e+00 & 146646 & 165.36872 & 6.13e-03 & 1.99e+00 & 104229 & 277.61611 & 1.06e-02 & 2.01e+00 \\
                \hline
                \multicolumn{2}{|c|}{Avg.} & * &  * & * & 2.27e+00 & * &  * & * & 2.22e+00 & * &  * & * & 2.43e+00 \\
                \hline
                \multirow{5}{*}{{$\frac{4}{3}$}} & 
                \multirow{5}{*}{\rotatebox{90}{Perfored cube}}
                & 18280 & 67.65476 & 9.26e+00 & * & 18280 & 82.97469 & 2.05e+01 & * & 18280 & 123.46614 & 1.07e+02 & * \\
                & & 29940 & 68.41144 & 5.87e+00 & 2.77e+00 & 34872 & 84.29090 & 1.35e+01 & 1.95e+00 & 35168 & 143.63909 & 1.03e+02 & 1.84e-01 \\
                & & 66604 & 69.31176 & 3.71e+00 & 1.72e+00 & 79500 & 86.32963 & 8.77e+00 & 1.56e+00 & 49996 & 151.15309 & 7.08e+01 & 3.16e+00 \\
                & & 150684 & 70.17242 & 2.26e+00 & 1.82e+00 & 163592 & 88.27421 & 5.57e+00 & 1.89e+00 & 82652 & 156.66908 & 5.82e+01 & 1.16e+00 \\
                & & * & * & * & * & * & * & * & * & 134976 & 164.69644 & 4.22e+01 & 1.97e+00 \\
                \hline
                \multicolumn{2}{|c|}{Avg.} & * &  * & * & 2.10e+00 & * &  * & * & 1.80e+00 & * &  * & * & 1.62e+00\\
                \hline
            \end{tabular}
        }
    \end{center}
    \vspace{-0.5cm}
    \caption{Example 2. Convergence history for a variety of 2D and 3D meshes, and different values of $\kappa_0$ of the first three eigenvalues $\lambda_{i,h}$. The values of the global total error estimator $\eta_i^2$ associated to the $i-$th eigenvalue ($i=1,2,3$) together with their rate of convergence are also displayed.}
    \label{tab:convergence2D3DExample2}
\end{table}

\begin{figure}[!t]
    \centering
    \subfigure[$\Pi^{\nabla^2}u_{1,h}.$]{\includegraphics[width=0.32\textwidth,trim={8.cm 3.25cm 2.5cm 3.75cm},clip]{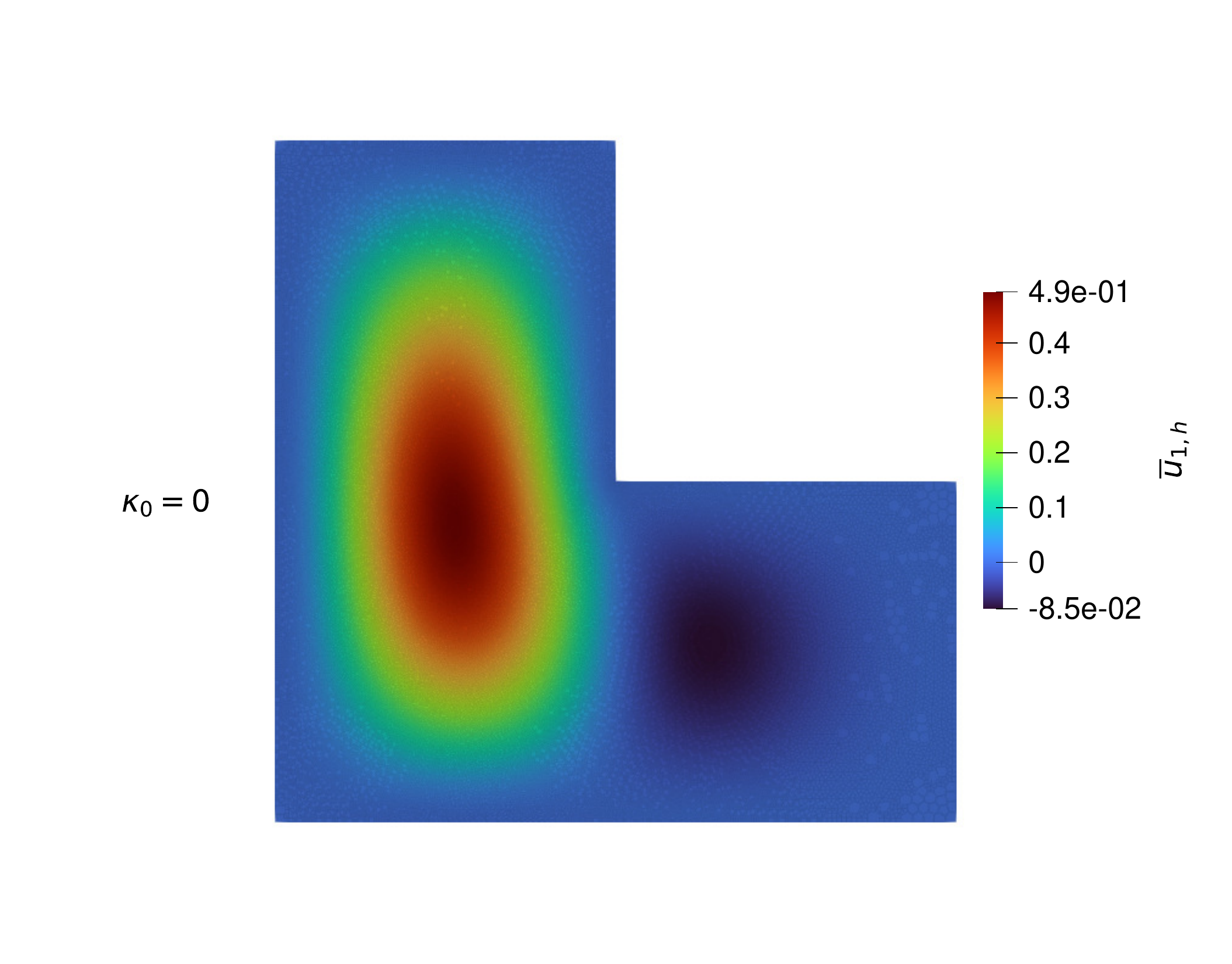}}
    \subfigure[$\Pi^{\nabla^2}u_{2,h}.$]{\includegraphics[width=0.32\textwidth,trim={8.cm 3.25cm 2.5cm 3.75cm},clip]{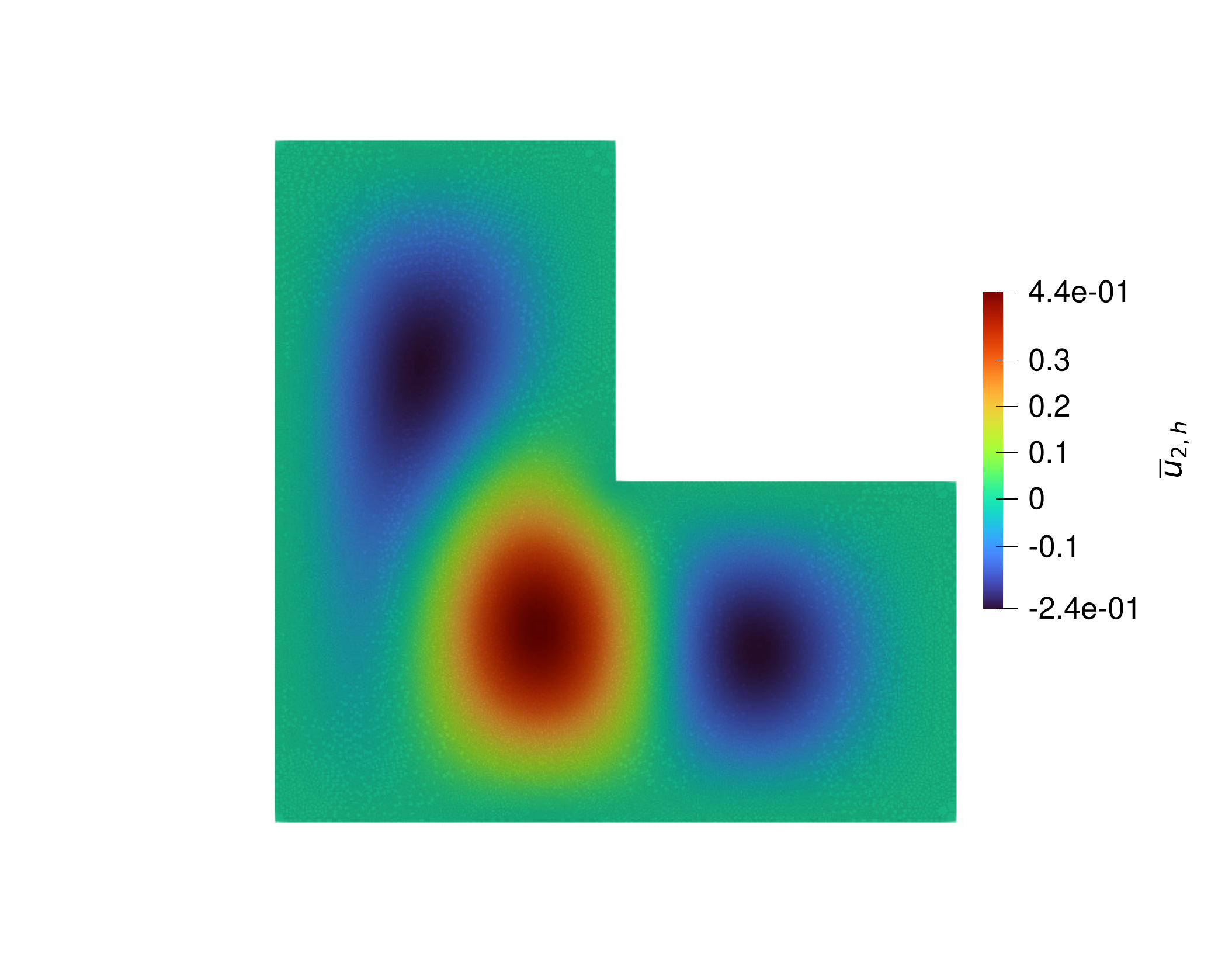}}
    \subfigure[$\Pi^{\nabla^2}u_{3,h}.$]{\includegraphics[width=0.32\textwidth,trim={8.cm 3.25cm 2.5cm 3.75cm},clip]{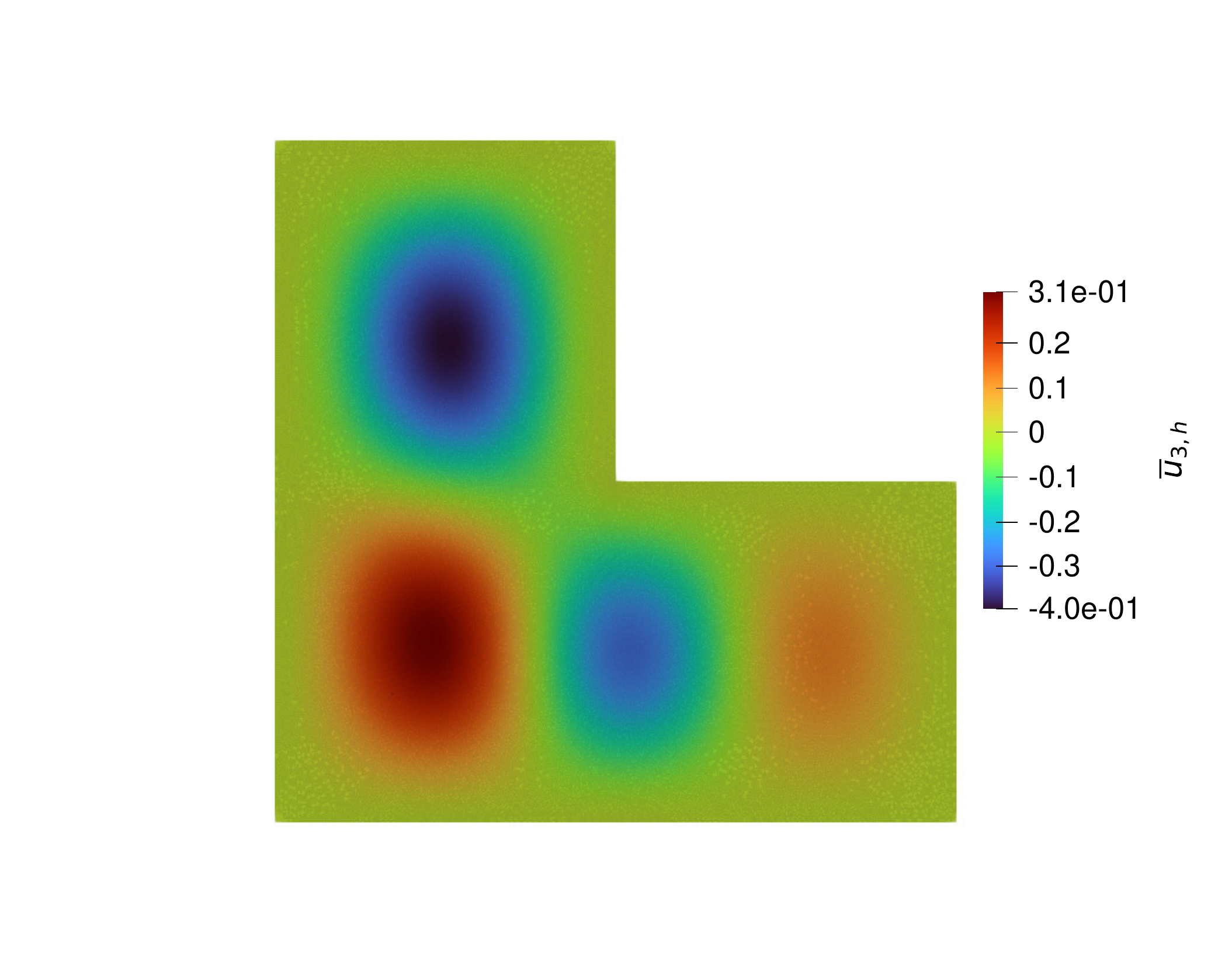}}
    \caption{Example 2. Snapshots of the polynomial projection of the first three eigenfunctions in the last refinement step for the L-shaped mesh with $\kappa_0=0$.}\label{fig:snapshotsLShape}
\end{figure}

\begin{figure}[!t]
    \centering
    \subfigure[$\Pi^{\nabla^2}u_{1,h}.$]{\includegraphics[width=0.32\textwidth,trim={8.cm 3.25cm 2.5cm 3.75cm},clip]{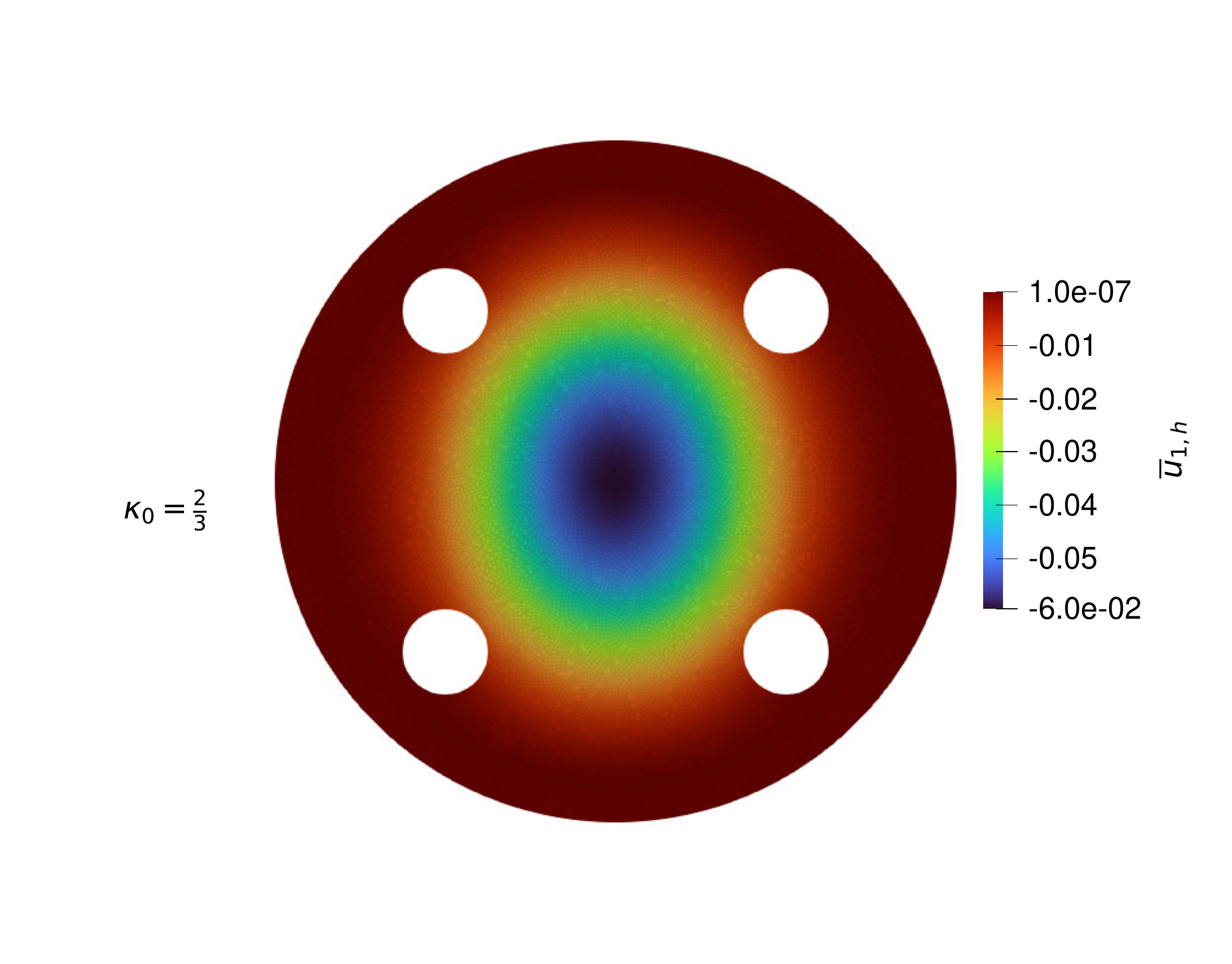}}
    \subfigure[$\Pi^{\nabla^2}u_{2,h}.$]{\includegraphics[width=0.32\textwidth,trim={8.cm 3.25cm 2.5cm 3.75cm},clip]{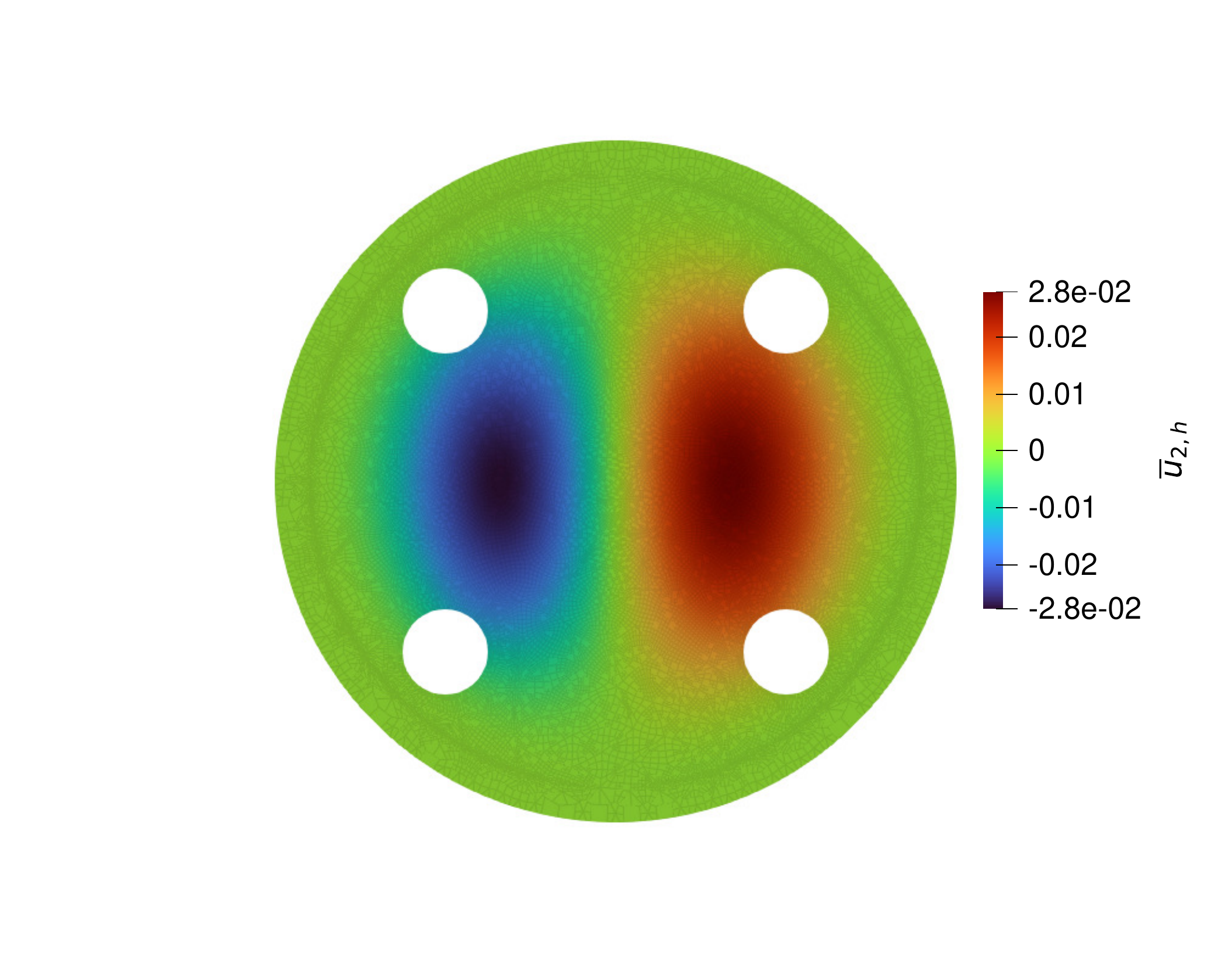}}
    \subfigure[$\Pi^{\nabla^2}u_{3,h}.$]{\includegraphics[width=0.32\textwidth,trim={8.cm 3.25cm 2.5cm 3.75cm},clip]{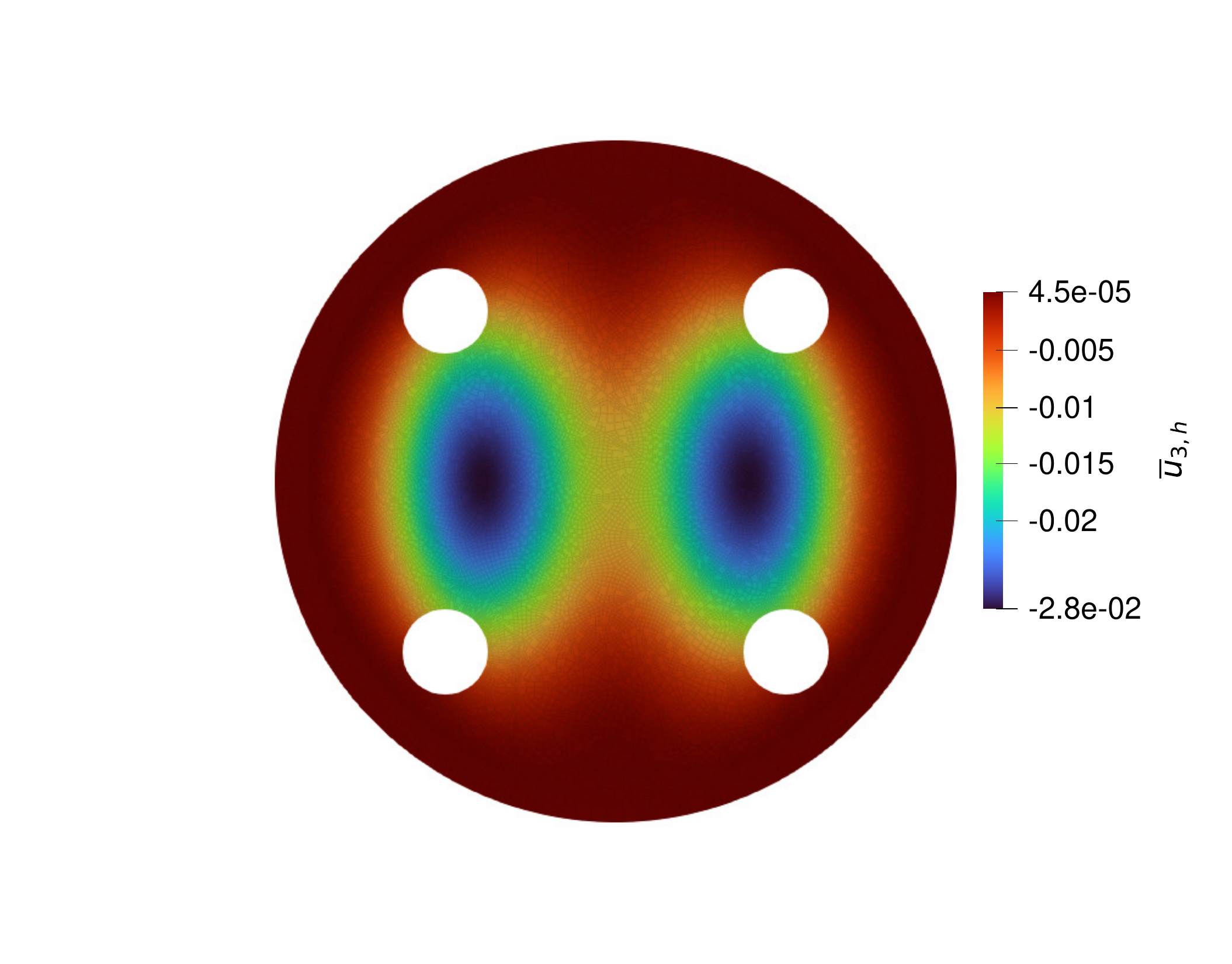}}
    \caption{Example 2. Snapshots of the polynomial projection of the first three eigenfunctions in the last refinement step for the perfored circle mesh with $\kappa_0=\frac{2}{3}$.}\label{fig:snapshotsPCircle}
\end{figure}

\begin{figure}[!t]
    \centering
    \subfigure[$\Pi^{\nabla^2}u_{1,h}.$]{\includegraphics[width=0.32\textwidth,trim={5.65cm 0.25cm 2.cm 4.25cm},clip]{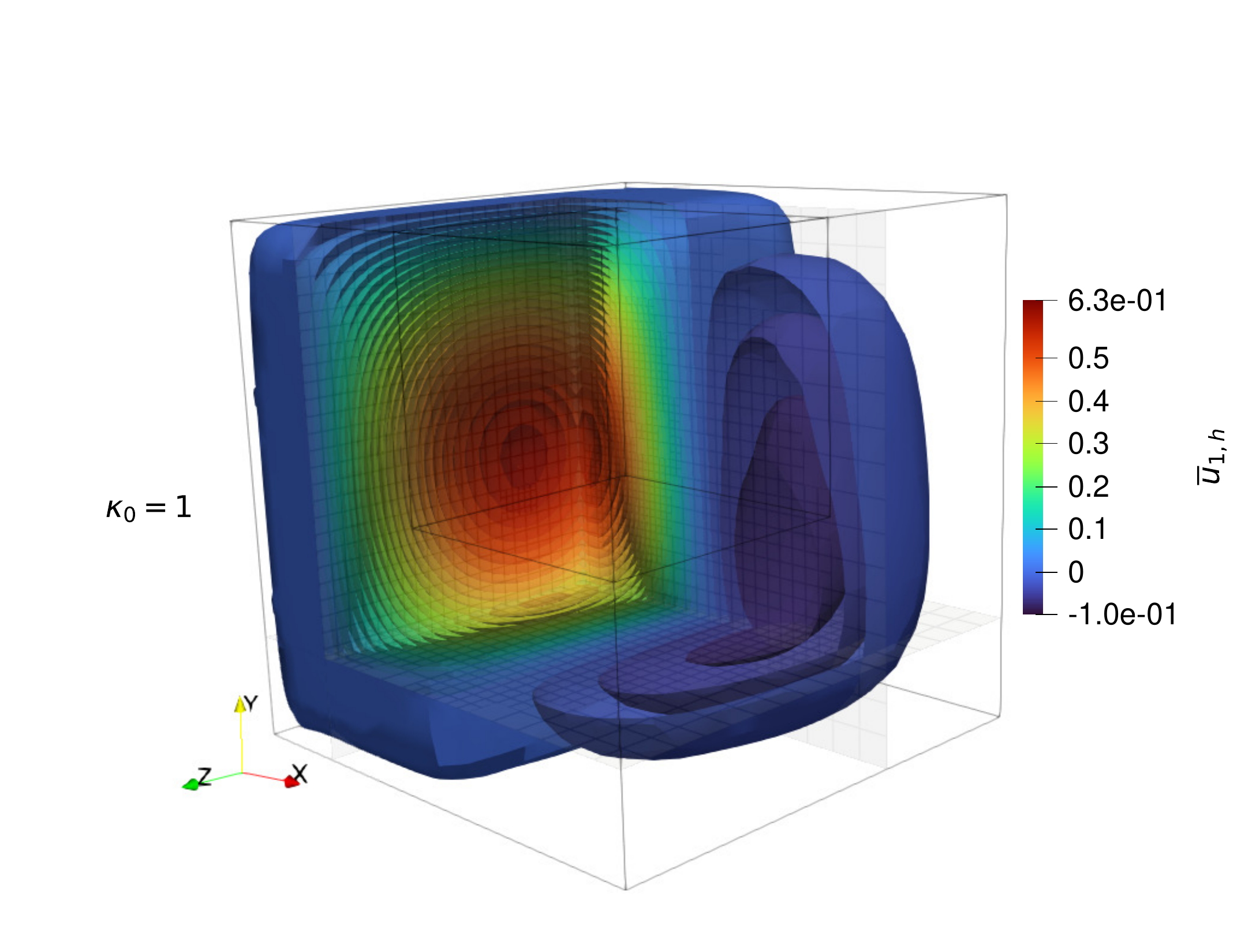}}
    \subfigure[$\Pi^{\nabla^2}u_{2,h}.$]{\includegraphics[width=0.32\textwidth,trim={5.65cm 0.25cm 2.cm 4.25cm},clip]{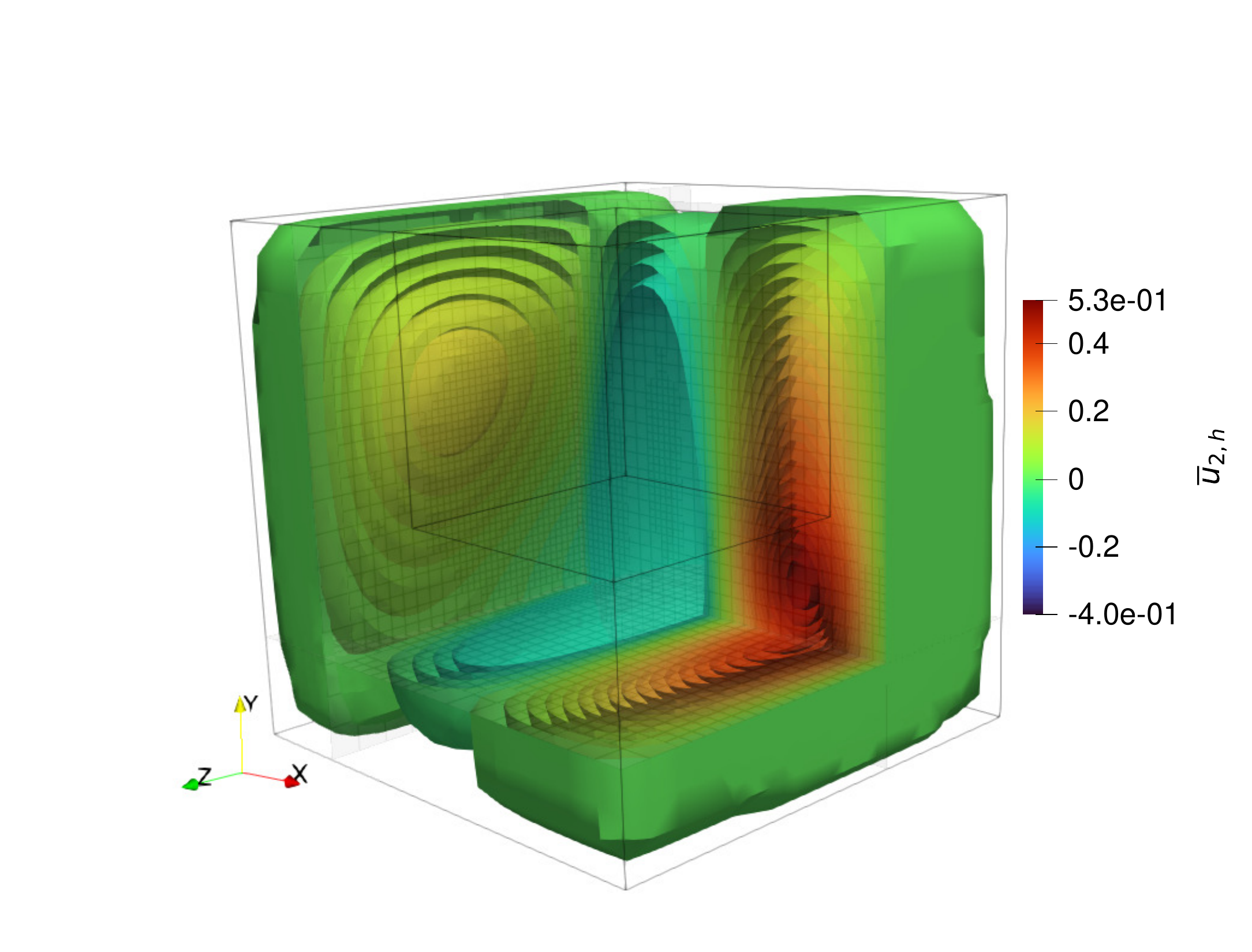}}
    \subfigure[$\Pi^{\nabla^2}u_{3,h}.$]{\includegraphics[width=0.32\textwidth,trim={5.65cm 0.25cm 2.cm 4.25cm},clip]{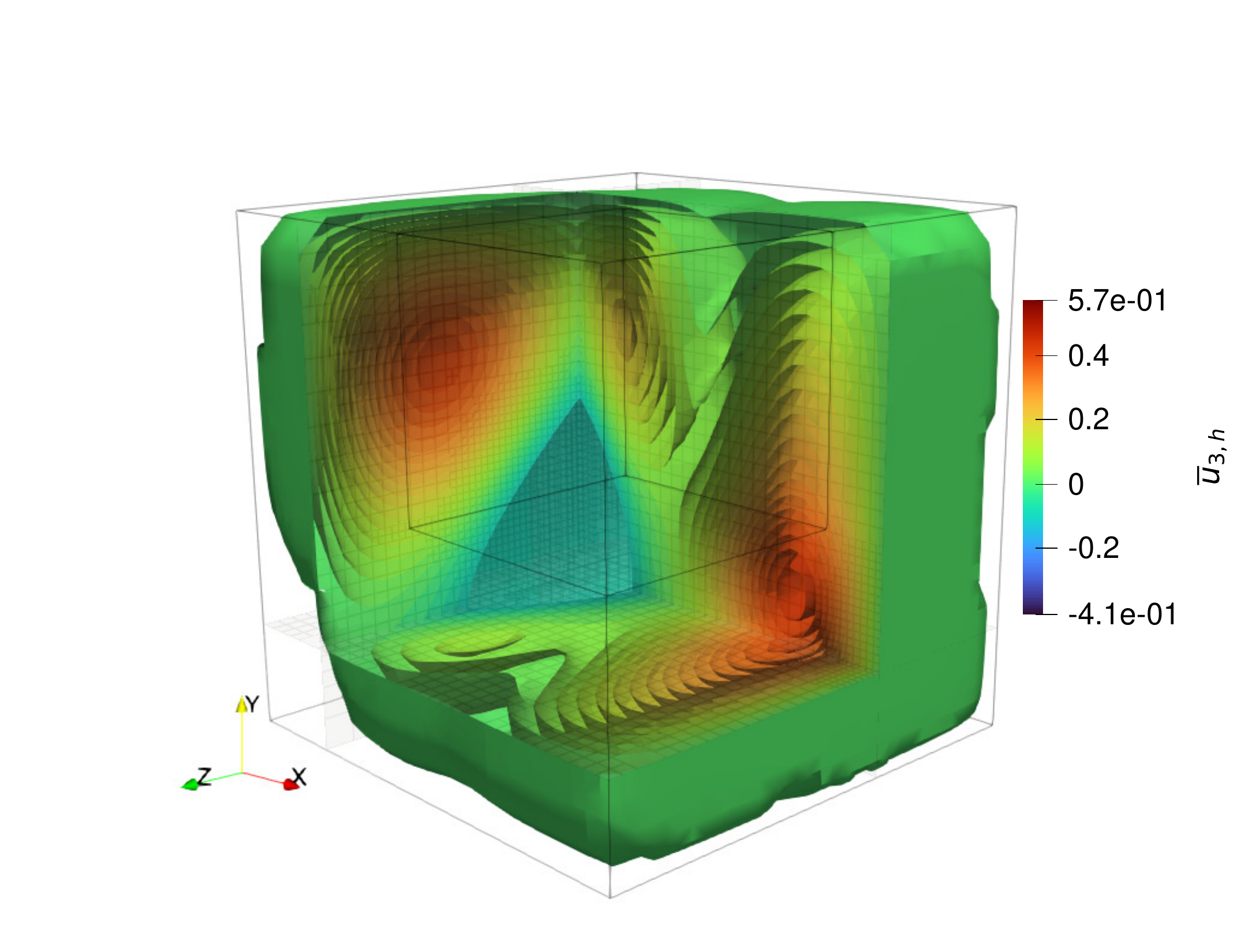}}
    \caption{Example 2. Snapshots of the polynomial projection of the first three eigenfunctions in the last refinement step for the perfored Fichera cube mesh with $\kappa_0=1$, the isosurfaces are computed from \texttt{paraview} starting from the mean value of the projected virtual function on each vertex.}\label{fig:snapshotsFichera}
\end{figure}

\begin{figure}[!t]
    \centering
    \subfigure[$\Pi^{\nabla^2}u_{1,h}.$]{\includegraphics[width=0.32\textwidth,trim={5.65cm 0.25cm 2.cm 2.75cm},clip]{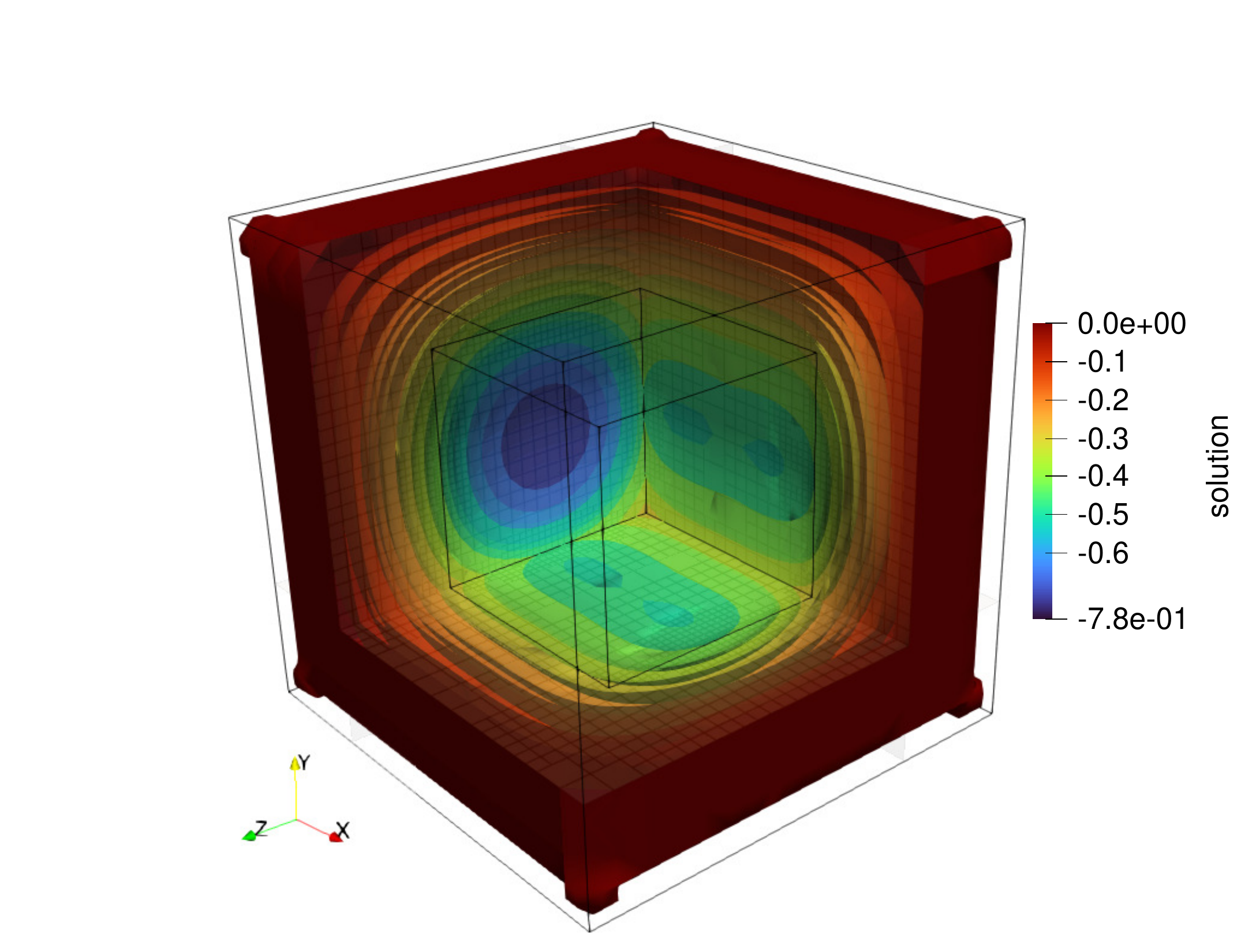}}
    \subfigure[$\Pi^{\nabla^2}u_{2,h}.$]{\includegraphics[width=0.32\textwidth,trim={5.65cm 0.25cm 2.cm 2.75cm},clip]{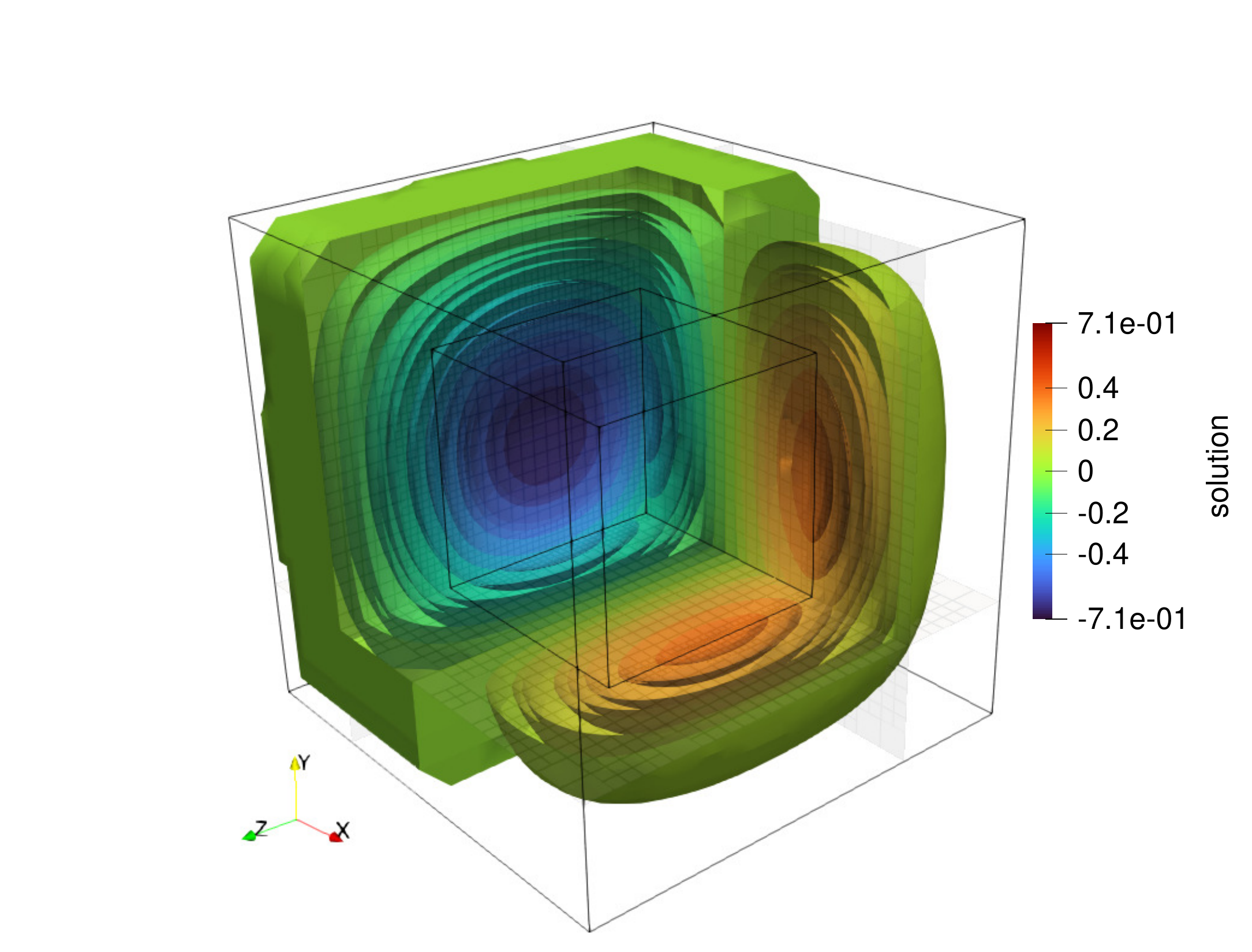}}
    \subfigure[$\Pi^{\nabla^2}u_{3,h}.$]{\includegraphics[width=0.32\textwidth,trim={5.65cm 0.25cm 2.cm 2.75cm},clip]{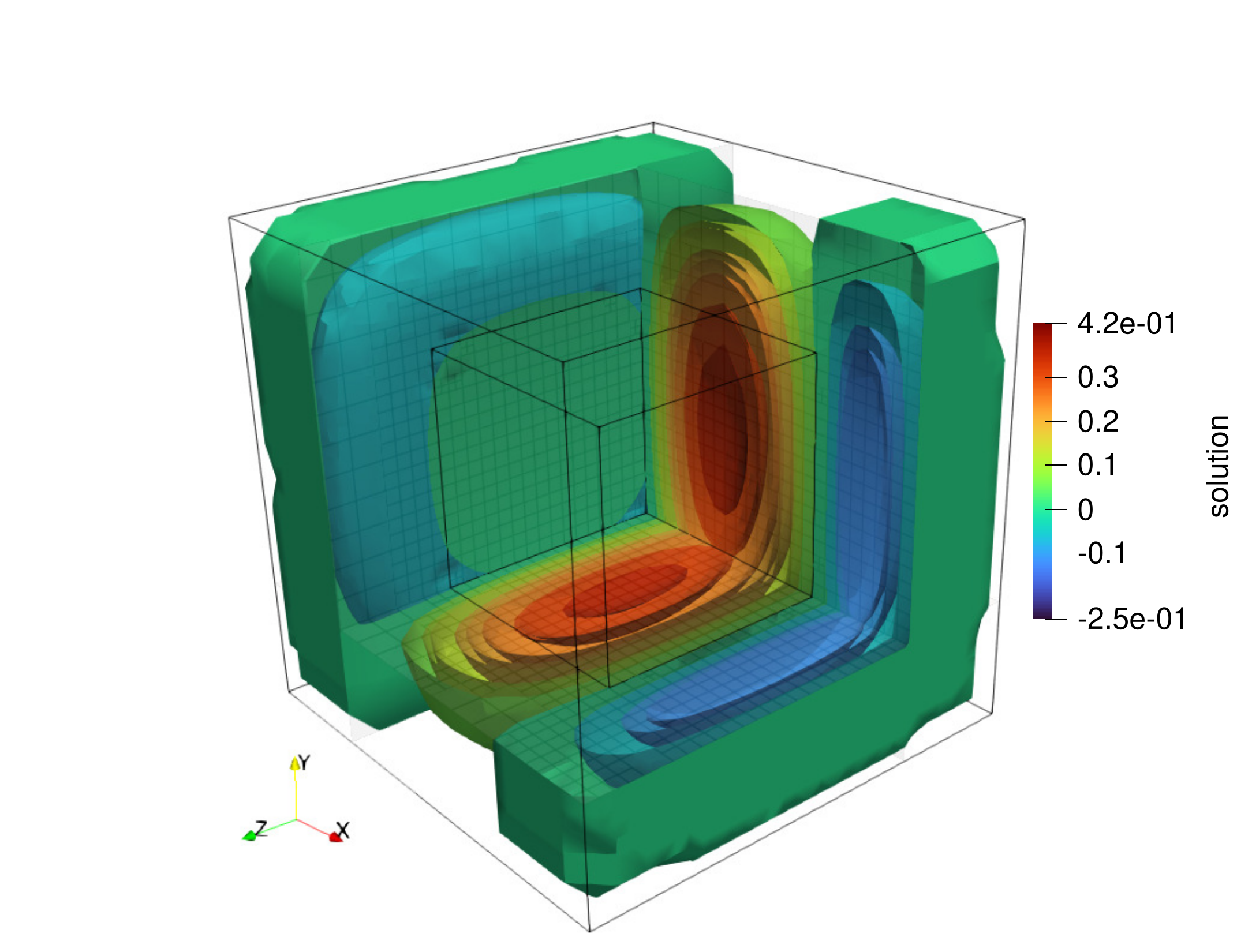}}
    \caption{Example 2. Snapshots of the polynomial projection of the first three eigenfunctions in the last refinement step for the perfored Perfored cube mesh with $\kappa_0=\frac{4}{3}$, the isosurfaces are computed from \texttt{paraview} starting from the mean value of the projected virtual function on each vertex.}\label{fig:snapshotsPCube}
\end{figure}

\subsection{Example 3. Buckling behaviour of turbine blades and the influence of internal air-cooling passages}
While Section~\ref{sec:example2} investigated perforated domains purely from a numerical analysis standpoint (cf. Figs~\ref{fig:CircleCircle}, \ref{fig:CubeCube}), this example illustrates their relevance in a engineering-type application. 

Modern jet turbine engines operate by compressing air to very high pressures and temperatures. When this compressed air mixes with fuel and ignites, the resulting hot, high-pressure gases expand through the turbine causing it to rotate and generate the mechanical forces needed to drive the engine. Naturally, the hot gases in the inner chambers of the engine raise the temperature of the turbine blades, which can be damaged when exposed to excessively high thermal loads. To overcome this issue internal air-cooling passages have been previously proposed (see eg. \cite{BANG2023123664,GAO20082139}). The air-cooling passages allow the cold air to flow through the engine blades which creates a film that protects the surface from high temperatures. 

In this simulation, we consider only the mechanical part of this phenomena with adimensional units, Figure~\ref{fig:domain_sketch} show the domain configuration, including the axial holes acting as the air-cooling passages. The domain is embedded in the unit square with boundary $\Gamma$ defined in both the external part of the blade turbine and the axial holes and shear load is imposed, i.e.,
$$\bkappa = \begin{pmatrix}
    0 & 1 \\
    1 & 0
\end{pmatrix}.$$
We explore the two type of boundary conditions presented in this paper (\textbf{SSP} and \textbf{CP}), with the material parameters (see \cite{TAKAGI2004348}) given by $L=1$, $t=10^{-3}$, $\nu=0.3$, and two different cases for the young modulus: room temperature case ($298\unit{\K}$) with a value of $E=156\unit{\GPa}$ and high temperature case ($1073\unit{\K}$) with a value of $E=103\unit{\GPa}$. These parameters correspond to Nickel-based superalloys, typically used in turbine blades and lead to $D=14.2857$ in the room temperature case and $D=9.4322$ in the high temperature case. We recall that the effect of the physical parameters relies on the computation of the first non-dimensional critical loading factor $\lambda_{1,h}$.

The critical load factor for each study are reported in Table~\ref{tab:loadingFactors} using the proposed scheme with both uniform and adaptive refinements. The adaptive routine provides very accurate solutions with respect to the uniformly refined, pointing out that approximately $94\%$ fewer degrees of freedom were require to achieve these results. Note that, at high temperature, the turbine blade section’s load-bearing capacity decreases by about $34\%$ compared to room temperature. Furthermore, the \textbf{CP} boundary condition consistently yields a higher critical load factor than the SSP case at both temperatures, in line with the expected physical behaviour. Finally, snapshots of the associated eigenfunctions are shown in Figure~\ref{fig:snapshotsTurbine}.

\begin{table}[!t]
    \setlength{\tabcolsep}{2pt}
    \begin{center}
        \resizebox{0.75\textwidth}{!}{ 
            \begin{tabular}{| c | c | c | c | r | c |}
                \hline
                {Temperature} &
                {Boundary condition} & 
                {Type of refinement} &
                {Refinement iterations} &
                {$\textnormal{\#DoFs}$} & 
                {$\lambda_{1,h}$} \\ [1pt]
                \hline 
                \hline
                \multirow{4}{*}{Room ($298\unit{\K}$)}
                & \multirow{2}{*}{\textbf{SSP}}
                    & uniform   & 5  & \num{1009299} & 5659 \\
                &   & adaptive  & 12 & \num{46947}   & 6004 \\
                \cline{2-6}
                & \multirow{2}{*}{\textbf{CP}}
                    & uniform   & 5  & \num{1009299} & 12632 \\
                &   & adaptive  & 12 & \num{73956}   & 12640 \\
                \hline
                \multirow{4}{*}{High ($1073\unit{\K}$)}
                & \multirow{2}{*}{\textbf{SSP}}
                    & uniform & 5 & \num{1009299} & 3736 \\
                &   & adaptive & 12 & \num{46947} & 3964 \\
                \cline{2-6}
                & \multirow{2}{*}{\textbf{CP}}
                    & uniform & 5 & \num{1009299} & 8340 \\
                &   & adaptive & 12 & \num{73956} & 8345 \\
                \hline
            \end{tabular}
        }
    \end{center}
    \vspace{-0.5cm}
    \caption{Example 3. Critical load factor for the turbine blade with air-cooling passages (axial holes) at different temperatures under (\textbf{SSP}) and (\textbf{CP}) boundary conditions, using uniform and adaptive refinement.}
    \label{tab:loadingFactors}
\end{table}

\begin{figure}[!t]
    \centering
    \includegraphics[width=0.75\textwidth,trim={2.25cm 0.cm 4.cm 0.35cm},clip]{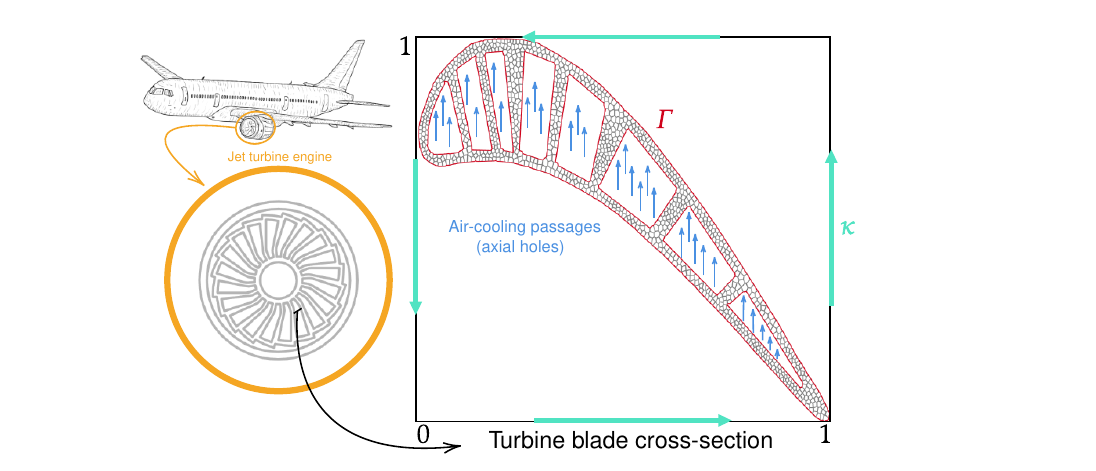}
    \caption{Example 3. Cross-section of a turbine blade with air-cooling passages (axial holes) embedded in the unit square, the effect of the shear load $\bkappa$ and the boundary $\Gamma$ are highlighted.}\label{fig:domain_sketch}
\end{figure}

\begin{figure}[!t]
    \centering
    \subfigure[\textbf{SSP}.]{\includegraphics[width=0.49\textwidth,trim={1.cm 0.25cm 1.5cm 1.75cm},clip]{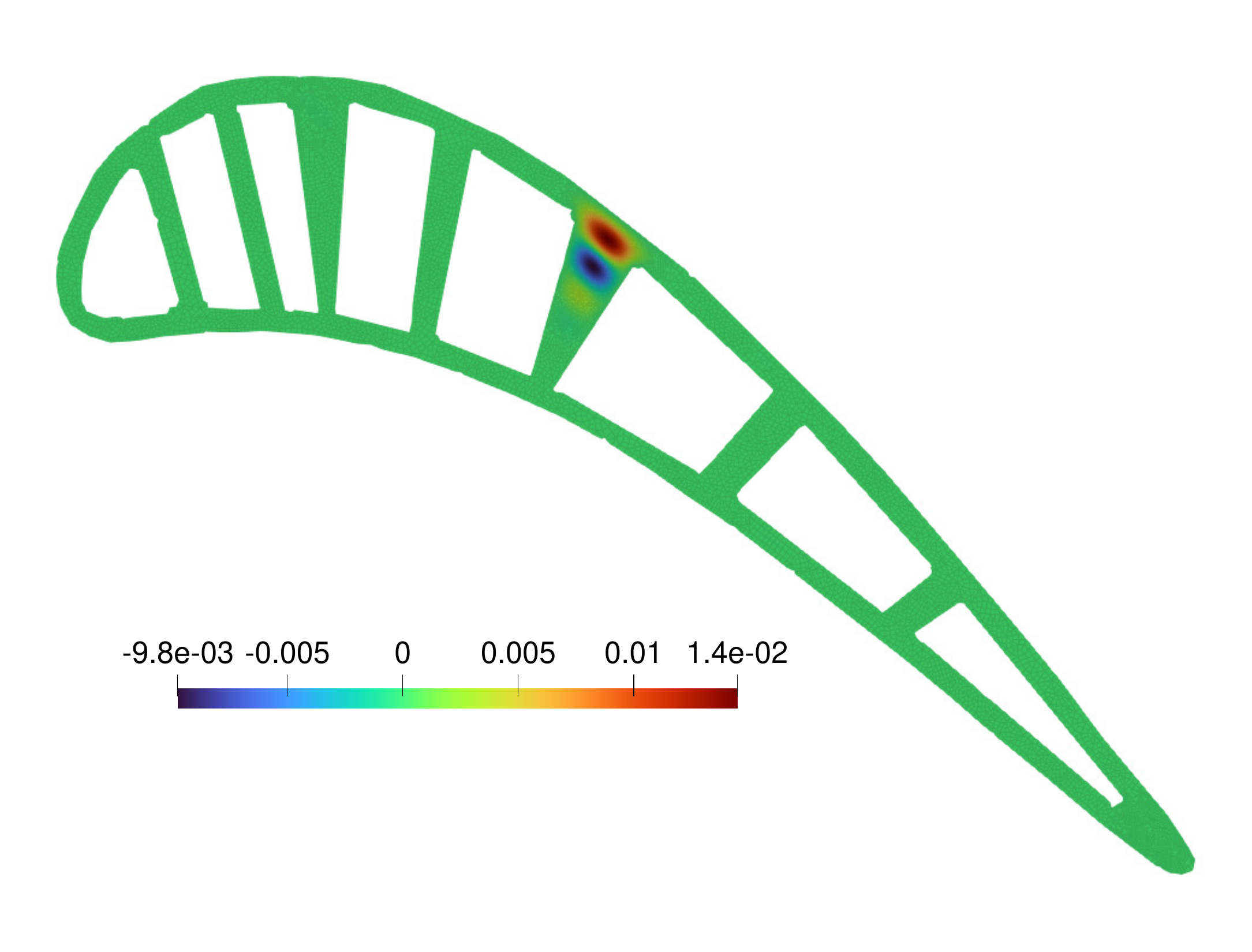}}
    \subfigure[\textbf{CP}.]{\includegraphics[width=0.49\textwidth,trim={1.cm 0.25cm 1.5cm 1.75cm},clip]{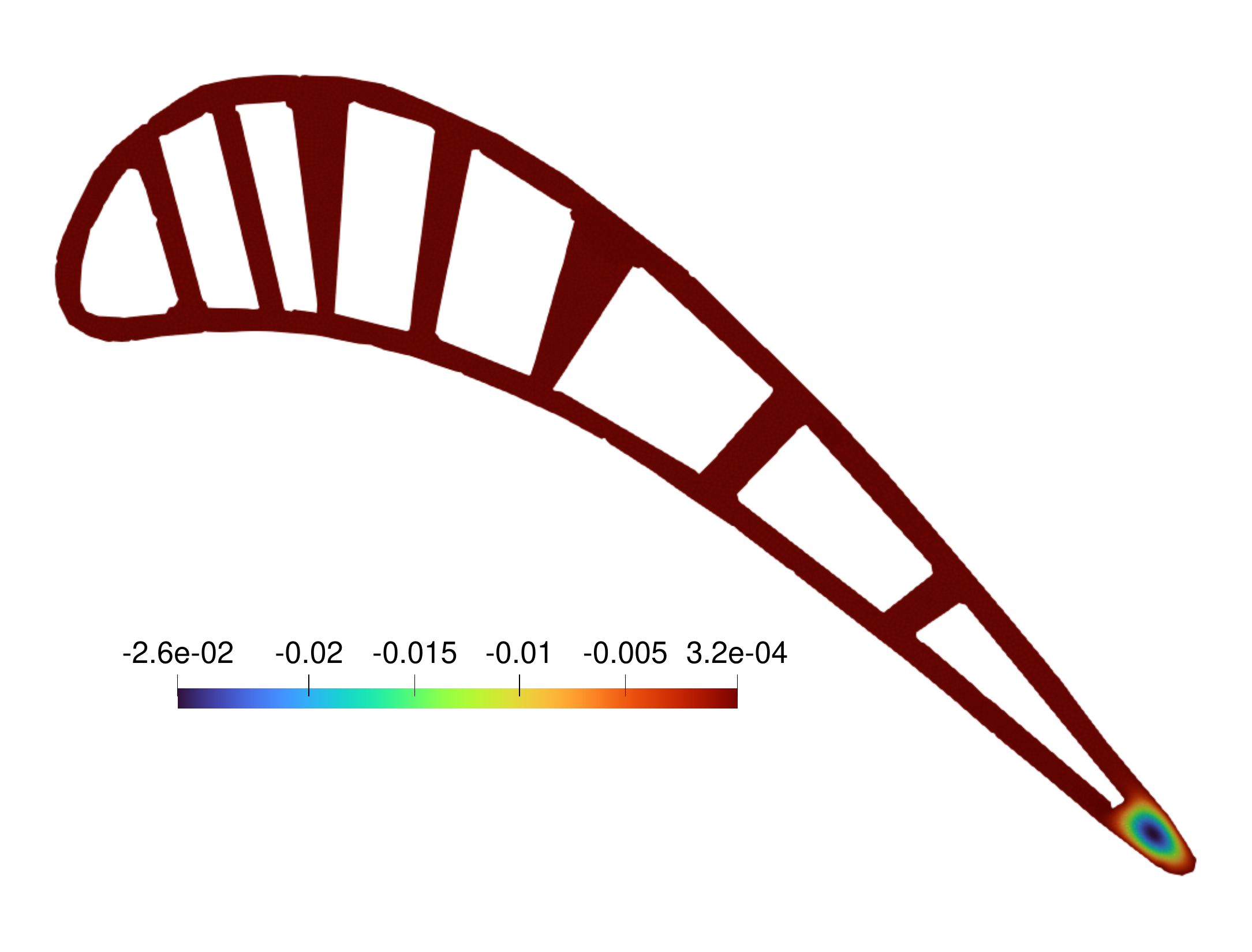}}
    \caption{Example 2. Snapshots of the polynomial projection for the first eigenfunctions ($\Pi^{\nabla^2}u_{1,h}$) in the last adaptive refinement step for the turbine blade with air-cooling passages (axial holes).}\label{fig:snapshotsTurbine}
\end{figure}

\section*{Acknowledgements}
Producto derivado del proyecto INV-CIAS-4321 financiado por la Universidad Militar Nueva Granada - Vigencia (2026)

FD was partially supported by the European Research Council project NEMESIS (Grant No. 101115663). AER has been partially supported by the Australian Research Council through the \textsc{Future Fellowship} grant FT220100496. IV was partially financially supported by Vicerrectoría de la Investigación de la Universidad Militar Nueva Granada (grant INV-CIAS-4321). 
We kindly thank Prof. Jesus Vellojin for the support provided in the SLEPC solver implementation.

\bibliographystyle{abbrv} \bibliography{main}

@book {BS-2008,
	AUTHOR = {Brenner, Susanne C. and Scott, L. Ridgway},
	TITLE = {The Mathematical Theory of Finite Element Methods},
	PUBLISHER = {Springer, New York},
	YEAR = {2008},
    doi = {https://doi.org/10.1007/978-0-387-75934-0},
    note = {DOI: \url{https://doi.org/10.1007/978-0-387-75934-0}}
}

@book{G,
author = {Grisvard, Pierre},
title = {Elliptic Problems in Nonsmooth Domains},
publisher = {SIAM},
year = {2011},
doi = {10.1137/1.9781611972030},
note = {DOI: \url{https://doi.org/10.1137/1.9781611972030}},
address = {Philadelphia, PA}
}

@book{ciarlet,
  author    = {Ciarlet, Philippe G.},
  title     = {The Finite Element Method for Elliptic Problems},
  series    = {Classics in Applied Mathematics},
  volume    = {40},
  publisher = {Society for Industrial and Applied Mathematics (SIAM)},
  address   = {Philadelphia, PA},
  year      = {2002},
  note = {DOI: \url{https://doi.org/10.1137/1.9780898719208}},
  doi       = {10.1137/1.9780898719208},
  isbn      = {9780898719208}
}

@book{adams2003sobolev,
  author    = {Robert A. Adams and John J. F. Fournier},
  title     = {Sobolev Spaces},
  edition   = {2nd},
  publisher = {Elsevier/Academic Press},
  address   = {Amsterdam, Netherlands},
  year      = {2003},
  isbn      = {978-0-12-044143-3},
  mrnumber  = {2424078}

}

@InProceedings{slepc,
author="Hern{\'a}ndez, Vicente
and Rom{\'a}n, Jose E.
and Vidal, Vicente",
editor="Palma, Jos{\'e} M. L. M.
and Sousa, A. Augusto
and Dongarra, Jack
and Hern{\'a}ndez, Vicente",
title="{SLEP}c: Scalable Library for Eigenvalue Problem Computations",
booktitle="High Performance Computing for Computational Science --- VECPAR 2002",
year="2003",
publisher="Springer Berlin Heidelberg",
address="Berlin, Heidelberg",
pages="377--391",
isbn="978-3-540-36569-3",
doi="https://doi.org/10.1145/1089014.1089019",
note = "DOI: \url{https://doi.org/10.1145/1089014.1089019}"
}

@article{CARSTENSEN2024,
title = {Rate-optimal higher-order adaptive conforming FEM for biharmonic eigenvalue problems on polygonal domains},
journal = {Comput. Methods Appl. Mech. Eng.},
fjournal = {Computer Methods in Applied Mechanics and Engineering},
volume = {425},
pages = {116931},
year = {2024},
issn = {0045-7825},
doi = {https://doi.org/10.1016/j.cma.2024.116931},
note = {DOI: \url{https://doi.org/10.1016/j.cma.2024.116931}},
author = {Carsten Carstensen and Benedikt Gräßle}
}

@article {dassi2025posteriorierrorestimatesc1,
    AUTHOR = {Dassi, Franco and Rubiano, Andr\'es E. and Vel\'asquez,
              Iv\'an},
     TITLE = {A posteriori error estimates for a {$C^1$} virtual element
              method applied to the thin plate vibration problem},
   JOURNAL = {Adv. Comput. Math.},
  FJOURNAL = {Advances in Computational Mathematics},
    VOLUME = {52},
      YEAR = {2026},
    NUMBER = {2},
     PAGES = {Paper No. 17},
      ISSN = {1019-7168,1572-9044},
   MRCLASS = {65N15 (65N25 65N30 74K20)},
  MRNUMBER = {5035573},
       DOI = {10.1007/s10444-026-10288-6},
       note = {DOI: \url{https://doi-org.ezproxyucor.unicordoba.edu.co/10.1007/s10444-026-10288-6}}
}

@article{Li2018,
author={Li, Hao
and Yang, Yidu},
title={Adaptive Morley element algorithms for the biharmonic eigenvalue problem},
journal={J. Inequal. Appl.},
fjournal={Journal of Inequalities and Applications},
year={2018},
month={Mar},
day={06},
volume={2018},
number={1},
pages={55},
issn={1029-242X},
doi={10.1186/s13660-018-1643-9},
note={DOI: \url{https://doi.org/10.1186/s13660-018-1643-9}}
}

@Article{feng2023,
  author={Feng, Jinhua and Wang, Shixi and Bi, Hai and Yang, Yidu},
  title={{An hp-mixed discontinuous Galerkin method for the biharmonic eigenvalue problem}},
  journal = {Appl. Math. Comput.},
  fjournal = {Applied Mathematics and Computation},
  year=2023,
  volume={450},
  number={C},
  pages={},
  month={},
  doi={10.1016/j.amc.2023.127969},
  note={DOI: \url{https://doi.org/10.1016/j.amc.2023.127969}}
}

@misc{yu2021implementationpolygonalmeshrefinement,
      title={Implementation of Polygonal Mesh Refinement in {MATLAB}}, 
      author={Yue Yu},
      year={2021},
      eprint={2101.03456},
      archivePrefix={arXiv},
      primaryClass={math.NA},
      Note={available at \url{https://arxiv.org/abs/2101.03456}}, 
}

@article{ABSV2016,
author = {Antonietti, P. F. and da Veiga, L. Beira͂o and Scacchi, S. and Verani, M.},
title = {A {C$^1$} Virtual Element Method for the Cahn--Hilliard Equation with Polygonal Meshes},
journal = {SIAM J. Numer. Anal.},
fjournal = {SIAM Journal on Numerical Analysis},
volume = {54},
number = {1},
pages = {34-56},
year = {2016},
doi = {10.1137/15M1008117},
note = {DOI: \url{https://doi.org/10.1137/15M1008117}}
}

@article {CWCC2017,
     AUTHOR = {Cao, Junying and Wang, Ziqiang and Cao, Waixiang and Chen,
               Lizhen},
      TITLE = {A mixed {L}egendre-{G}alerkin spectral method for the buckling
               problem of simply supported {K}irchhoff plates},
    JOURNAL = {Bound. Value Probl.},
   FJOURNAL = {Boundary Value Problems},
     VOLUME = {34},   
       YEAR = {2017},
      PAGES = {1--12},
      doi = {https://doi.org/10.1186/s13661-017-0767-z},
      note = {DOI: \url{https://doi.org/10.1186/s13661-017-0767-z}}
      }

@article{CKD,
  title={Discrete singular convolution approach for
  buckling analysis of rectangular Kirchhoff plates subjected to compressive loads on two-opposite edges},
  author={Civalek, {\"O}mer and Korkmaz, Arma{\u{g}}an and Demir, {\c{C}}i{\u{g}}dem},
  journal={Adv. Eng. Softw.},
  fjournal={Advances in Engineering Software},
  volume={41},
  number={4},
  pages={557--560},
  year={2010},
  doi = {https://doi.org/10.1016/j.advengsoft.2009.11.002},
  note = {DOI: \url{https://doi.org/10.1016/j.advengsoft.2009.11.002}}
}

@article {HLM2015,
    AUTHOR = {Hansbo, Peter and Larson, Mats G.},
     TITLE = {A posteriori error estimates for continuous/discontinuous
              {G}alerkin approximations of the {K}irchhoff-{L}ove buckling
              problem},
   JOURNAL = {Comput. Mech.},
  FJOURNAL = {Computational Mechanics},
    VOLUME = {56},
      YEAR = {2015},
    NUMBER = {5},
     PAGES = {815--827},
     doi = {https://doi.org/10.1007/s00466-015-1204-8},
     note = {DOI: \url{https://doi.org/10.1007/s00466-015-1204-8}}
     }

@article {I2,
      AUTHOR = {Ishihara, Kazuo},
       TITLE = {On the mixed finite element approximation for the buckling of
                plates},
     JOURNAL = {Numer. Math.},
    FJOURNAL = {Numerische Mathematik},
      VOLUME = {33},
        YEAR = {1979},
      NUMBER = {2},
       PAGES = {195--210},
       doi = {https://doi.org/10.1007/BF01399554},
       note = {DOI: \url{https://doi.org/10.1007/BF01399554}}
       }

@article {MM2015,
    AUTHOR = {Millar, Felipe and Mora, David},
     TITLE = {A finite element method for the buckling problem of simply
              supported {K}irchhoff plates},
   JOURNAL = {J. Comput. Appl. Math.},
  FJOURNAL = {Journal of Computational and Applied Mathematics},
    VOLUME = {286},
      YEAR = {2015},
     PAGES = {68--78},
     doi = {https://doi.org/10.1016/j.cam.2015.02.018},
     note = {DOI: \url{https://doi.org/10.1016/j.cam.2015.02.018}}
     }

@article {Ra,
     AUTHOR = {Rannacher, Rolf},
      TITLE = {Nonconforming finite element methods for eigenvalue problems
               in linear plate theory},
    JOURNAL = {Numer. Math.},
   FJOURNAL = {Numerische Mathematik},
     VOLUME = {33},
       YEAR = {1979},
     NUMBER = {1},
      PAGES = {23--42},
      doi = {https://doi.org/10.1007/BF01396493},
      note = {DOI: \url{https://doi.org/10.1007/BF01396493}}
      }

@article{MV2,
AUTHOR = {Mora, David and Vel\'{a}squez, Iv\'{a}n},
     TITLE = {Virtual element for the buckling problem of
              {K}irchhoff--{L}ove plates},
   JOURNAL = {Comput. Methods Appl. Mech. Eng.},
  FJOURNAL = {Computer Methods in Applied Mechanics and Engineering},
    VOLUME = {360},
      YEAR = {2020},
     PAGES = {112687},
     doi={https://doi.org/10.1016/j.cma.2019.112687},
     note={DOI: \url{https://doi.org/10.1016/j.cma.2019.112687}}
}

@article {Brenner_VK2017,
         AUTHOR = {Brenner, Susanne C. and Neilan, Michael and Reiser, Armin and
                   Sung, Li-Yeng},
          TITLE = {A {$C^0$} interior penalty method for a von {K}\'{a}rm\'{a}n plate},
        JOURNAL = {Numer. Math.},
       FJOURNAL = {Numerische Mathematik},
         VOLUME = {135},
           YEAR = {2017},
         NUMBER = {3},
          PAGES = {803--832},
          doi = {https://doi.org/10.1007/s00211-016-0817-y},
          note = {DOI: \url{https://doi.org/10.1007/s00211-016-0817-y}}
          }

@article {BDR2019C1Polyhedral,
	AUTHOR = {Beir\~{a}o da Veiga, L. and Dassi, F. and Russo, A.},
	TITLE = {A {C$^1$} Virtual Element Method on polyhedral meshes},
	JOURNAL = {Comput. Math. Appl.},
FJOURNAL = {Computers \& Mathematics with Applications. An International
Journal},
volume={79},
	pages={1936--1955},
	NUMBER = {7},
	year={2020},
    doi={https://doi.org/10.1016/j.camwa.2019.06.019},
    note={DOI: \url{https://doi.org/10.1016/j.camwa.2019.06.019}}
}

@book{gazzola,
  title={Polyharmonic boundary value problems: positivity preserving and nonlinear higher order elliptic equations in bounded domains},
  author={Gazzola, Filippo and Grunau, Hans-Christoph and Sweers, Guido},
  year={2010},
  publisher={{S}pringer {S}cience \& {B}usiness {M}edia},
  address   = {Berlin, Heidelberg},
  edition   = {1st},
  doi = {https://doi.org/10.1007/978-3-642-12245-3},
  note = {DOI: \url{https://doi.org/10.1007/978-3-642-12245-3}}
}

@ARTICLE{p4est,
  author = {Carsten Burstedde and Lucas C. Wilcox and Omar Ghattas},
  title = {{\texttt{p4est}}: Scalable Algorithms for Parallel Adaptive Mesh
           Refinement on Forests of Octrees},
  fjournal = {SIAM Journal on Scientific Computing},
journal = {SIAM J. Sci. Comput.},
  volume = {33},
  number = {3},
  pages = {1103-1133},
  year = {2011},
  doi = {https://doi.org/10.1137/100791634},
  note={DOI: \url{https://doi.org/10.1137/100791634}}
}

@article{gridap,
  year = {2020},
  publisher = {The Open Journal},
  volume = {5},
  number = {52},
  pages = {2520},
  author = {Santiago Badia and Francesc Verdugo},
  title = {Gridap: An extensible Finite Element toolbox in {J}ulia},
  journal = {J. Open Source Softw.},
  fjournal = {Journal of Open Source Software},
  note = {DOI: \url{https://doi.org/10.21105/joss.02520}},
  doi = {https://doi.org/10.21105/joss.02520}
}

@article{DV_camwa2022,
title = {Virtual element method on polyhedral meshes for bi-harmonic eigenvalues problems},
JOURNAL = {Comput. Math. Appl.},
volume = {121},
pages = {85-101},
year = {2022},
issn = {0898-1221},
doi = {https://doi.org/10.1016/j.camwa.2022.07.001},
note = {DOI: \url{https://doi.org/10.1016/j.camwa.2022.07.001}},
author = {Franco Dassi and Iván Velásquez},
}

@article{dassi2025vem++,
  title={{VEM++}, a {C++} library to handle and play with the Virtual Element Method},
  author={Dassi, Franco},
  journal={Numer. Algorithms},
  fjournal={Numerical Algorithms},
  pages={1--43},
  year={2025},
  publisher={Springer},
  doi={https://doi.org/10.1007/s11075-025-02059-z},
  note = {DOI: \url{https://doi.org/10.1007/s11075-025-02059-z}}
}

@article{MRV2018,
	author = {{Mora, David} and {Rivera, Gonzalo} and {Velásquez, Iván}},
	title = {A virtual element method for the vibration problem of Kirchhoff plates},
	DOI= "10.1051/m2an/2017041",
	note= "DOI: \url{https://doi.org/10.1051/m2an/2017041}",
	journal = {ESAIM. Math. Model. Numer. Anal.},
	year = 2018,
	volume = 52,
	number = 4,
	pages = "1437-1456",
}

@article {MR4047014,
    AUTHOR = {Wang, Liang and Xiong, Chunguang and Wu, Huibin and Luo,
              Fusheng},
     TITLE = {A priori and a posteriori error analysis for discontinuous
              {G}alerkin finite element approximations of biharmonic
              eigenvalue problems},
   JOURNAL = {Adv. Comput. Math.},
  FJOURNAL = {Advances in Computational Mathematics},
    VOLUME = {45},
      YEAR = {2019},
    NUMBER = {5-6},
     PAGES = {2623--2646},
      ISSN = {1019-7168,1572-9044},
   MRCLASS = {65N25 (65N15 65N30 65N55)},
  MRNUMBER = {4047014},
MRREVIEWER = {Svetozar\ D.\ Margenov},
       DOI = {10.1007/s10444-019-09689-7},
       note = {DOI: \url{https://doi.org/10.1007/s10444-019-09689-7}},
}

@article{MoraVelasquez2020,
  title={Virtual element for the buckling problem of Kirchhoff--Love plates},
  author={Mora, David and Vel{\'a}squez, Iv{\'a}n},
  journal={Comput. Methods Appl. Mech. Eng.},
  fjournal={Computer Methods in Applied Mechanics and Engineering},
  volume={360},
  pages={112687},
  year={2020},
  publisher={Elsevier},
  doi={https://doi.org/10.1016/j.cma.2019.112687},
  note={DOI: \url{https://doi.org/10.1016/j.cma.2019.112687}}
}

@article{MR2009,
author = {Mora, David and Rodriguez, Rodolfo},
year = {2009},
month = {10},
pages = {1891-1917},
title = {A piecewise linear finite element method for the buckling and the vibration problems of thin plates},
volume = {78},
journal = {Math. Comput.},
doi = {10.1090/S0025-5718-09-02228-5},
note = {DOI: \url{https://doi.org/10.1090/S0025-5718-09-02228-5}}
}

@article{ADAK2023115763,
title = {A {C$^0$}-nonconforming virtual element methods for the vibration and buckling problems of thin plates},
journal = {Comput. Meth. Appl. Mech. Eng.},
fjournal = {Computer Methods in Applied Mechanics and Engineering},
volume = {403},
pages = {115763},
year = {2023},
issn = {0045-7825},
note = {DOI: \url{https://doi.org/10.1016/j.cma.2022.115763}},
author = {Dibyendu Adak and David Mora and Iván Velásquez},
doi = {https://doi.org/10.1016/j.cma.2022.115763}
}

@article{GAO20082139,
title = {Film cooling on a gas turbine blade pressure side or suction side with axial shaped holes},
journal = {Int. J. Heat Mass Transf.},
fjournal = {International Journal of Heat and Mass Transfer},
volume = {51},
number = {9},
pages = {2139-2152},
year = {2008},
issn = {0017-9310},
doi = {https://doi.org/10.1016/j.ijheatmasstransfer.2007.11.010},
note = {DOI: \url{https://doi.org/10.1016/j.ijheatmasstransfer.2007.11.010}},
author = {Zhihong Gao and Diganta P. Narzary and Je-Chin Han}
}

@article{BANG2023123664,
title = {Augmented cooling performance in gas turbine blade tip with slot cooling},
journal = {Int. J. Heat Mass Transf.},
fjournal = {International Journal of Heat and Mass Transfer},
volume = {201},
pages = {123664},
year = {2023},
issn = {0017-9310},
doi = {https://doi.org/10.1016/j.ijheatmasstransfer.2022.123664},
note = {DOI: \url{https://doi.org/10.1016/j.ijheatmasstransfer.2022.123664}},
author = {Minho Bang and Seungyeong Choi and Seok Min Choi and Dong-Ho Rhee and Hee Koo Moon and Hyung Hee Cho}
}

@article{TAKAGI2004348,
title = {Measuring Young’s modulus of Ni-based superalloy single crystals at elevated temperatures through microindentation},
author = {Hidenari Takagi and Masami Fujiwara and Koji Kakehi},
journal = {Mater. Sci. Eng. A.},
volume = {387-389},
pages = {348-351},
year = {2004},
issn = {0921-5093},
doi = {https://doi.org/10.1016/j.msea.2004.01.061},
note = {DOI: \url{https://doi.org/10.1016/j.msea.2004.01.061}}
}

@article{BYUN2011274,
title = {Buckling analysis and optimal structural design of supercavitating vehicles using finite element technology},
journal = {Int. J. Nav. Archit. Ocean Eng.},
fjournal = {International Journal of Naval Architecture and Ocean Engineering},
volume = {3},
number = {4},
pages = {274-285},
year = {2011},
issn = {2092-6782},
doi = {https://doi.org/10.2478/IJNAOE-2013-0071},
note = {DOI: \url{https://doi.org/10.2478/IJNAOE-2013-0071}},
author = {Wanil Byun and Min Ki Kim and Kook Jin Park and Seung Jo Kim and Minho Chung and Jin Yeon Cho and Sung-Han Park},
}

@article{YORK1998665,
title = {Aircraft wing panel buckling analysis: efficiency by approximations},
journal ={Comput. Struct.},
fjournal = {Computers \& Structures},
volume = {68},
number = {6},
pages = {665-676},
year = {1998},
issn = {0045-7949},
doi = {https://doi.org/10.1016/S0045-7949(98)00050-9},
note = {DOI: \url{https://doi.org/10.1016/S0045-7949(98)00050-9}},
author = {C.B. York and F.W. Williams},
}

@article{Stein1994AdaptiveFE,
  title={Adaptive finite element analysis of geometrically non‐linear plates and shells, especially buckling},
  author={Erwin Stein and B. Seifert and Stephan Ohnimus and Carsten Carstensen},
  journal={Int. J. Numer. Methods Eng.},
  fjournal={International Journal for Numerical Methods in Engineering},
  year={1994},
  volume={37},
  pages={2631-2655},
  doi={https://doi.org/10.1002/nme.1620371508},
  note={DOI: \url{https://doi.org/10.1002/nme.1620371508}}
}

@book{verfurth1996review,
  title={A Review of A Posteriori Error Estimation and Adaptive Mesh-Refinement Techniques},
  author={Verf{\"u}rth, R{\"u}diger},
  isbn={9783519026051},
  lccn={gb96047389},
  series    = {Advances in Numerical Mathematics},
  publisher = {Wiley-Teubner},
  address   = {Stuttgart},
  year={1996}
}

@article{zhao2023interior,
  title={The interior penalty virtual element method for the biharmonic problem},
  author={Zhao, Jikun and Mao, Shipeng and Zhang, Bei and Wang, Fei},
  journal={Mathematics of Computation},
  volume={92},
  number={342},
  pages={1543--1574},
  year={2023},
  doi={https://doi.org/10.1090/mcom/3828}
}
\end{document}